\documentclass[10pt]{amsart}
\usepackage{amssymb}
\usepackage{graphicx}
\usepackage{xcolor} 
\usepackage{tensor}
\usepackage{fullpage} 
\usepackage{amsmath}
\usepackage{amsthm}
\usepackage{verbatim}
\usepackage{txfonts,bm,mathtools}

\usepackage{enumitem}
\setlist[enumerate]{leftmargin=1.5em}
\setlist[itemize]{leftmargin=1.5em}

\usepackage{hyperref}
\usepackage{seqsplit,mathtools} 
\usepackage{caption}
\usepackage{subcaption}
\usepackage{ulem}

\usepackage{thmtools}
\usepackage{thm-restate}
\usepackage{lipsum}

\definecolor{green}{rgb}{0,0.5,0} 

\newtheorem{thm}{Theorem}[section]
\newtheorem{cor}[thm]{Corollary}
\newtheorem{lem}[thm]{Lemma}
\newtheorem{prop}[thm]{Proposition}

\newtheorem{defn}[thm]{Definition}
\newtheorem{con}{Conjecture}

\newtheorem{quest}{Question}

\theoremstyle{definition}

\theoremstyle{remark}
\newtheorem{rmk}[thm]{Remark}

\numberwithin{equation}{section}
\newcommand{\nrm}[1]{\Vert#1\Vert}

\newcommand{\set}[1]{\{#1\}}
\newcommand{\tld}[1]{\widetilde{#1}}

\newcommand{\nnrm}[1]{{\vert\kern-0.25ex\vert\kern-0.25ex\vert #1 
		\vert\kern-0.25ex\vert\kern-0.25ex\vert}}

\newcommand{\supp}{{\mathrm{supp}}\,}

\newcommand{\lap}{\Delta}

\newcommand{\rd}{\partial}
\newcommand{\nb}{\nabla}

\newcommand{\imp}{\Rightarrow}

\newcommand{\ift}{\infty}

\newcommand{\alp}{\alpha}

\newcommand{\gmm}{\gamma}
\newcommand{\Gmm}{\Gamma}
\newcommand{\dlt}{\delta}

\newcommand{\eps}{\epsilon}

\newcommand{\lmb}{\lambda}

\newcommand{\tht}{\theta}

\newcommand{\omg}{\omega}
\newcommand{\Omg}{\Omega}

\newcommand{\bfu}{{\bf u}}

\newcommand{\bfx}{{\bf x}}
\newcommand{\bfy}{{\bf y}}
\newcommand{\bfz}{{\bf z}}

\newcommand{\bfU}{{\bf U}}

\newcommand{\bbR}{\mathbb R}

\newcommand{\calB}{\mathcal B}

\newcommand{\calE}{\mathcal E}

\newcommand{\calG}{\mathcal G}

\newcommand{\weakto}{\rightharpoonup}

\newcommand{\f}[2]{\frac{#1}{#2}}       
\newcommand{\ii}[2]{\int_{#1}^{#2}}     

\newcommand{\q}{\mbox{ }}       
\newcommand{\qd}{\quad }
\newcommand{\qqd}{\qquad }

\newcommand{\bfone}{\mathbf{1}}

\begin{document}
	
	\bibliographystyle{plain}
	
	\title{	\Large{On existence of Sadovskii vortex patch:} \\ 
		\large{A touching pair of symmetric counter-rotating uniform {vortices}} } 
		
	\author{Kyudong Choi}
	\address{Department of Mathematical Sciences, Ulsan National Institute of Science and Technology, 50 UNIST-gil, Eonyang-eup, Ulju-gun, Ulsan 44919, Republic of Korea.}
	\email{kchoi@unist.ac.kr}
	
	\author{In-Jee Jeong}
	\address{Department of Mathematical Sciences and RIM, Seoul National University, 1 Gwanak-ro, Gwanak-gu, Seoul 08826, and School of Mathematics, Korea Institute for Advanced Study, Republic of Korea.}
	\email{injee$ \_ $j@snu.ac.kr}
	
	\author{Young-Jin Sim}
	\address{Department of Mathematical Sciences, Ulsan National Institute of Science and Technology, 50 UNIST-gil, Eonyang-eup, Ulju-gun, Ulsan 44919, Republic of Korea. }
	\email{yj.sim@unist.ac.kr}
	
	\date\today
	\maketitle
 
	\begin{abstract}
		The Sadovskii vortex patch is a traveling wave for the two-dimensional incompressible Euler equations consisting of an odd symmetric pair of vortex patches touching the symmetry axis. Its existence was first suggested by numerical computations of Sadovskii in [J. Appl. Math. Mech., 1971], and has gained significant interest due to its relevance in inviscid limit of planar flows via Prandtl--Batchelor theory and as the asymptotic state for vortex ring dynamics. In this work, we  prove the existence of a Sadovskii vortex patch, by solving the energy maximization problem under the exact impulse condition and an upper bound on the circulation.  
	\end{abstract}
 

\tableofcontents

\newpage

\section{Introduction}\label{sec_introduction}\q
	
	We consider the Cauchy problem for the 2D incompressible Euler equations in vorticity form:
	
 \begin{equation}\label{eq_Euler eq.}
		\left\{
		\begin{aligned}
			&\qd\rd_{t}\tht+\bfu\cdot\nb\tht=0,\\
         &\qd\bfu=k*\tht,\\
			&\qd\tht|_{t=0}=\tht_{0},
		\end{aligned}
		\right.
\end{equation}
	where $\tht:[0,\ift)\times\bbR^2\to\bbR$. {Based on the Biot--Savart law,} the kernel $k$ is given as $$k(\bfx)=(2\pi)^{-1}(-x_2,x_1)/|\bfx|^2,$$ implying that $\bfu=(u^1,u^2)=\nb^\perp\psi\q$ where $ \q\nb^\perp = (\rd_{x_2}, -\rd_{x_1})$ with the stream function  $\psi=\psi[\tht]$ defined by $$\psi:=(-\lap)_{\bbR^2}^{-1}\tht=-\Big(\f{1}{2\pi}\log|\cdot|\Big)*\tht.$$ 
 In this paper, we are concerned with traveling waves of \eqref{eq_Euler eq.} with the following properties: The function $\tht:[0,\ift)\times\bbR^2\to\bbR$ given by 
$$\tht(t,\bfx)=\Big(\bfone_{A}-\bfone_{A_-}\Big)(\bfx-(W,0)t)$$ solves \eqref{eq_Euler eq.} for some constant $W>0$ and for a bounded open set $A\subset\set{\q(x_1,x_2)\in\bbR^2:x_2>0\q}$ satisfying that,  for some $R>0$, we have $$\set{\bfx\in\bbR^2:|\bfx|<R}\subset \overline{A\cup A_-}$$ where $A_-$ is defined as $A_-:=\set{\q(x_1,x_2)\in\bbR^2:(x_1,-x_2)\in A\q}.$ In other words, we deal with vortex patch solutions whose support in the upper half plane \textit{touches the horizontal axes} and travels with the constant velocity $(W,0)$ without change in its profile. Hereafter, we will refer to such a solution as a \textit{Sadovskii vortex patch} in honor of Sadovskii's contribution \cite{SADOVSKII.1971} in 1971. While there has been extensive research on such \textit{touching} vortex patches (including \cite{Chernyshenko.1988, Moor.Saffman.Tanveer.1988, Overman.1986, Saffman.Szeto.1980, Saffman.Tanveer.1982, Yang.Kubota.1994, Pierrehumbert.1980, Smith.1986}), most of it has focused on their properties rather than proving their existence. We shall defer the detailed historical discussion to Subsection \ref{subsec_sadovskii_important} below. 

Our main result in this paper gives the existence of a Sadovskii vortex patch as a bounded open  set in the upper half plane $\bbR^2_+:=\set{(x_1,x_2)\in\bbR^2\q:\q x_2>0}$.

\begin{thm}\label{Main thm 1} There exists a bounded open set $\Omg\subset\bbR^2_+$, {normalized to have unit impulse $\int_\Omg x_2\q d\bfx=1,$} satisfying the following properties: \begin{enumerate}
		\item[(i)] \textup{(Touching)} There exists a constant $R>0$ such that $$\big\{\q\bfx\in\bbR^2\q:\q|\bfx|<R\q\big\}\q\subset\q \overline{\Omg\cup\Omg_-}.$$ 
		\item[(ii)] \textup{(Traveling wave)} The $(L^1\cap L^\ift)(\bbR^2)-$ function $\tht_0:\bbR^2\to\bbR$ defined by $$\tht_0:=\bfone_\Omg-\bfone_{\Omg_-}$$ generates a traveling wave in the sense that, for the constant $W>0$ given by 
		$$W:=\f{1}{3}\int_{\bbR^2}|\bfu_0|^2\q d\bfx\qd\mbox{ where }\bfu_0:=k*\tht_0,$$  the function $\tht:[0,\ift)\times\bbR^2\to\bbR$ defined by $$\tht(t,\bfx):=\tht_0(\bfx-(W,0)t)$$ solves the 2D incompressible Euler equations \eqref{eq_Euler eq.}.
		\item[(iii)] \textup{(Steiner symmetric with continuous boundary)} There exists a continuous function  
		$l:[0,\ift)\to[0,\ift)$ such that 
		$$\Big\{\q(x_1,x_2)\in \overline{\Omg} \q:\q x_2=s\q\Big\}\q=\q 
		\Big\{\q(x_1,s)\in {\overline{\bbR^2_+}}\q:\q|x_1| \le  l(s)\q\Big\}\qd\mbox{ for each }s \ge 0.$$ 
	\end{enumerate}
\end{thm}

 The rest of the introduction is organized as follows. We begin with describing the relevance of Sadovskii vortex patch in fluid dynamics in Subsection \ref{subsec_sadovskii_important}. Then, in Subsection \ref{subsec_technical}, we divide the statement of Theorem \ref{Main thm 1} into three theorems, which together gives the main result. Subsection \ref{subsec_ideas} provides key ideas that goes into the proofs. Important technical discussions and open problems are given in Subsections \ref{subsec_variational} and \ref{subsec_open}, respectively.  

\subsection{Significance of Sadovskii vortex patch}\label{subsec_sadovskii_important} \q

Counter-rotating vortex pairs are one of the simplest flow patterns and easily observed in various situations, most notably in the wake of any flying objects. There have been numerous studies of counter-rotating vortex pairs from experimental, computational, and theoretical point of view. An excellent review of the literature is given in Leweke--Diz\`{e}s--Williams \cite{Leweke.Dizes.Williamson.2016} (see also Saffman \cite{Saffman.1992}). While Sadovskii vortex can be thought of just an example of such vortex pairs, it has several distinctive characters as we shall explain below. 

\medskip

\noindent \textbf{History  of Sadovskii vortex}. The Prandtl--Batchelor theorem states that, if the vanishing viscosity limit of a two-dimensional flow has closed streamlines, then the limiting vorticity is piecewise constant in each region separated by the streamlines. Therefore, there has been significant interest in understanding potential shape of limiting closed streamlines in various flow geometries. The case of a streamline intersecting the fluid boundary is of particular interest. 
This was precisely the motivation of Sadovskii (1971) \cite{SADOVSKII.1971}: in the upper half-plane $\bbR^2_+$, he considered the problem of determining a region $\Omg$ of uniform vorticity corresponding to a steady flow which is irrotational in $\bbR^2_+ \backslash \Omg$. He assumed that the steady velocity is asymptotic to $- W e_{1}$  at infinity, which is precisely \eqref{eq_steady Euler eq.}. The key assumption on $\Omg$ was $\overline{\Omg\cap {\bbR^2_+}} = [-1,1]$, and  the problem was to determine the streamline connecting $(-1,0)$ with $(1,0)$ in the interior of $\bbR^2_+$ (see Figure \ref{fig:Sadovskii}). Sadovskii was able to numerically compute the family of such streamlines, parameterized by the strength of the vortex sheet (size of the velocity jump) across this interior streamline. Sadovskii's original computation   indicated that the case when the vortex sheet is absent is distinguished by the fact that the corresponding streamline is orthogonal to the domain boundary; see \cite[Figure 3]{SADOVSKII.1971}.  After this work, the term \textit{Sadovskii vortex} was widely used to denote equilibrium inviscid flows involving uniform vorticity regions, potentially with vortex sheets, in various geometries (\cite{Chernyshenko.1993.Stratified,Chernyshenko.1993.Desity-stratified,Smith.Freilich.2017,Moor.Saffman.Tanveer.1988,Antipov.Zemlyanova,Christopher.Llewellyn}). Therefore, to avoid confusion, in the current work we reserve the phrase \textit{Sadovskii vortex patch} to denote a uniformly translating vortex patch, touching the boundary of the half-plane, \underline{without} an accompanying vortex sheet. Sadovskii himself wrote \cite[p. 734]{SADOVSKII.1971} that ``\textit{An analytical proof of the existence and uniqueness (or otherwise) of the solution of this system does not seem possible.}'' While many authors have studied the properties of Sadovskii vortex patch since then (\cite{MR0734586, Chernyshenko.1988, Chernyshenko.1993.Stratified, Chernyshenko.1993.Desity-stratified, Smith.Freilich.2017, Habibah.Fukumoto.1913, Kim.2003, Moor.Saffman.Tanveer.1988,Overman.1986, Saffman.Szeto.1980,Saffman.Tanveer.1982, Pierrehumbert.1980,Smith.1986,Yang.Kubota.1994}), the rigorous proof of its existence seems to be still open. In particular, it is not clear whether the boundary of the vortex region defines a smooth curve up to $\partial \bbR^{2}_{+}$, the property assumed by Sadovskii and most of the subsequent authors. {The paper \cite[Table I]{MR0734586} includes numerical computations for a Sadovskii vortex patch, which shows that when the maximum height from the axis is normalized to be $1$, the length touching the $x_{1}$-axis is approximately $3.398$, the area (or mass) is about $2.47$, and the vertical center of mass is around $0.415$.}  Lastly, we point out that essentially the same problem of Sadovskii was considered earlier by Gol'dshtik (1962, \cite{Goldshtik}) and Childress (1966, \cite{Childress.1966}), although these authors were using certain approximations of the Euler equations to compute the solutions.

\begin{figure}
	\centering
	\includegraphics[scale=1]{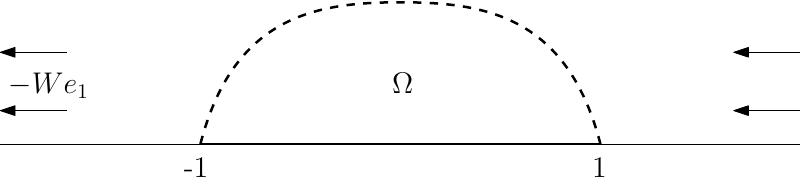}  
	\caption{Setup of Gol'dshtik (1962), Childress (1966), and Sadovskii (1971). The steady vorticity is 0 and 1 in $\bbR^{2}_{+}\backslash \Omg$ and $\Omg$, respectively, and the problem is to find the dashed curve.} \label{fig:Sadovskii}
\end{figure}

\medskip

\noindent \textbf{Sadovskii vortex as the endpoint of the bifurcation diagram}. There exist a large set of counter-rotating vortex pairs, even if we restrict them to the class of symmetric and uniformly-translating ones. In what follows, we shall {further} assume that the uniformly translating vortex pairs have uniform normalized vorticity (i.e. vortex patch). We further assume for simplicity that the vortex is odd symmetric with respect to the $x_{1}$-axis and non-negative on the upper half-plane $\bbR^{2}_{+}$. The important physical parameters are the circulation $\Gmm$ (=mass), traveling velocity $W$, and the impulse $\mu$. (Note that all of them are positive under our setting.) Using scale invariance of the incompressible Euler equations, one may normalize one of them: our choice is to normalize the impulse to be 1, as seen in the statement of Theorem \ref{Main thm 1}. In the upper half plane, a point vortex (vorticity given as the Dirac delta measure at a point) is formally a traveling wave solution of the Euler equations, and one may de-singularize it to obtain uniformly traveling vortex patches with $\nrm{\omg}_{L^\infty} = O(\eps^{-1})$ and $\Gmm,W,\mu = O(1)$ for sufficiently small $\eps>0$ (\cite{Burton.2021,CQZZ.2023,Wang.2024}). Normalizing such a vortex to be of unit vorticity can give a uniformly traveling vortex patch with $\Gmm = \mu = 1$ and $W = O(\eps)$. As $\eps\to 0^+$, the patches are asymptotically circular and $W \to 0^+$. Then, one may find a branch of uniformly traveling patch solutions by increasing $W$ and keeping $\mu = 1$. Numerical studies show that there is a critical threshold at which the patch becomes a Sadovskii vortex and bifurcation cannot be continued further (\cite{Overman.1986,Fegiz.Willamson.2012,Leweke.Dizes.Williamson.2016}).  This phenomenon seems to show relevance of Sadovskii vortex patch in dynamics of solutions; we expand this point further in Subsections \ref{subsec_ideas} and \ref{subsec_variational} below.

\medskip

\noindent \textbf{Sadovskii vortex in three-dimensional flows}. In three-dimensional incompressible inviscid flows, a counter-rotating pair of vortex rings can approach each other, increasing their distance from the symmetry axis and at the same time giving rise to various instabilities. This ``head-on collision'' scenario of vortex rings was the subject of intense numerical and computational studies (\cite{KaMi,KaMy,ShLe,CWCCC,Shariff.Leonard.Zabusky.Ferziger.1988,LN,SSH,PeRi,Os,CLL,GWRW,SMO,Shariff.Leonard.Zabusky.Ferziger.1988,Shariff.Leonard.Ferziger.2008,Carley.2002}). Employing the axisymmetric Euler equations and imposing exact odd symmetry (this is equivalent with posing the Euler equations on the upper half-space $\bbR^{3}_{+}$), one can suppress various instabilities and ask the asymptotic behavior of the axial cross section of the rings for large times. 

Based on arguments based on variational principles, Childress \cite{Childress.2007,Childress.2008} argued that the growth of the vorticity maximum   $\nrm{\omg(t,\cdot)}_{L^\infty}$ as well as the distance of the vortex from the symmetry axis can grow at most by $t^{4/3}$ as $t\to\infty$. (The exponent $4/3$ can be explained by dimensional arguments based on kinetic energy conservation.) Later, Childress and his collaborators (\cite{Childress.Gilbert.Valiant.2016,Childress.Gilbert.2018}) performed numerical computations and provided theoretical and computational evidence that the optimal $t^{4/3}$ growth is indeed achieved by vortex rings whose axial cross sections are asymptotic to the Sadovskii vortex after a time-dependent rescaling. As long as the cross section of the vortex ring is sufficiently localized, this is reasonable since far away from the symmetry axis, the axisymmetric Biot--Savart kernel converges to that for the two-dimensional Euler equations (\cite{Childress.1987}). Furthermore, note the numerical computations done by Carley \cite{Carley.2002}.\footnote{See dynamical appearance of Sadovskii-like vortex in  \url{https://people.bath.ac.uk/ensmjc/Research/Vortex/headon.html}.} Here, it seems that the appearance of this limiting asymptotic profile is not depending sensitively on the shape of the initial vortex ring, and in particular, it is not needed that the initial ring is ``attached'' to the boundary of the upper half-space. Indeed, this dynamical appearance of the Sadovskii vortex patch was observed earlier and explicitly by Shariff--Leonard--Ferziger in 2008 \cite{Shariff.Leonard.Ferziger.2008} for the head-on collision of two Hill's spherical vortices, see \cite[Figures 8, 9]{Shariff.Leonard.Ferziger.2008}. So far, the best known lower bound in time on the vorticity growth is slower than linear-in-time (\cite{Choi.Jeong.2021.vortex_streatching,Gustafson.Miller.Tsai}) and it seems that understanding stability of the Sadovskii vortex in the two-dimensional case is the natural first step towards establishing the sharp growth $t^{4/3}$.

\subsection{Technical statements} \label{subsec_technical}  \q
In this subsection, in order to more easily convey the ideas needed for the proof of our main result (Theorem \ref{Main thm 1}), we divide Theorem \ref{Main thm 1} into    the following three theorems : Variational method (Theorem \ref{Main thm: variational method}), Touching property (Theorem \ref{Main thm: central speed}), and Boundary behavior (Theorem \ref{Main thm: bdry}).\\

To find a traveling wave solution, for some constant $W>0$, we are looking for a steady solution $(\bfU,\omg)$ that solves 
 \begin{equation}\label{eq_steady Euler eq.}
		\left\{
		\begin{aligned}
			&\qd\bfU\cdot\nb\omg=0\qd\mbox{in}\qd\bbR^2,\\
         &\qd\bfU=k*\omg-(W,0),
		\end{aligned}
		\right.
	\end{equation} so that the solution $(\bfu,\tht)$ in the form of 
 \begin{align}
     \bfu(t,\bfx)&=\bfU(\bfx-(W,0)t)+(W,0),\\
     \tht(t,\bfx)&=\omg(\bfx-(W,0)t),
 \end{align} solves \eqref{eq_Euler eq.}. The equations \eqref{eq_steady Euler eq.} admits a \textit{vortex patch} solution of a form 
 $$\omg=\bfone_{\set{\q\Psi>0\q}}=\
 \left\{
		\begin{aligned}
			\q {1} \qd&\mbox{for}\q \Psi>0,\\
                0\qd&\mbox{for}\q \Psi\leq0
		\end{aligned}
		\right. 
  \qd\mbox{in }\q\bbR^2_+,  $$
  with the odd symmetry $\omg(x_1,-x_2)=-\omg(x_1,x_2)\q\mbox{ for each }(x_1,x_2)\in\bbR^2_+,\q$ and $\bfU=\nb^\perp\Psi$ where $$\Psi(\bfx):=\psi[\omg](\bfx)-Wx_2-\gmm$$ for some constant $\gmm\geq0$ which is referred to as the flux constant in the literature.
   One may easily observe that touching of a vortex patch implies $\gmm=0$. We will prove that such a solution with $\gmm=0$ can be found, and it is a Sadovskii vortex patch.\\
 
We will obtain Theorem \ref{Main thm 1} by the following steps. In Theorem \ref{Main thm: variational method}, we will find a solution $\omg$ of \eqref{eq_steady Euler eq.} satisfying that, with the constants $W>0$ and $\gmm=0$, we have 
$$\omg=\bfone_{\set{\bfx\in\bbR^2_+\q:\q\psi[\omg](\bfx)-Wx_2>0}}\qd\mbox{ a.e. in }\bbR^2_+.$$ In Theorem \ref{Main thm: central speed}, for the function $\omg$ we have found above, we will show that there exists $R>0$ such that we have  
$$\big[\mbox{ Touching }\big]\q\qd \set{\q\bfx\in\bbR^2_+:|\bfx|<R\q}\subset\set{\q\bfx\in\bbR^2_+:\psi[\omg](\bfx)-Wx_2>0\q}.$$ In Theorem \ref{Main thm: bdry} we will say that the boundary of the set $\set{\psi[\omg]-Wx_2>0}$ is continuous.\\

 In Theorem \ref{Main thm: variational method} below,  we use a variational method to prove that there exists a patch-type solution of \eqref{eq_steady Euler eq.} which consequently generates a traveling wave solution of \eqref{eq_Euler eq.} with the traveling speed $W>0$ and the constant $\gmm=0$. In specific, we will solve the variational problem of finding a maximizer $\omg$ of the kinetic energy with a fixed constraint on the impulse $\int_{\bbR^2_+}x_2\omg\q d\bfx$. The proof of Theorem \ref{Main thm: variational method} can be found in the proof of Theorem \ref{thm_ no mass cond}.
\begin{thm}\label{Main thm: variational method}
    Define $K$ by
$$ K:= \Bigg\{\q\omg\in L^1(\bbR^2)\q:\q \omg=\bfone_A-\bfone_{A_-}\q\mbox{ a.e. in }\bbR^2\qd\mbox{for some open bounded set }A\subset\bbR^2_+\q\q\mbox{satisfying} \q\int_A x_2\q d\bfx=1\Bigg\}$$
    and define $S\subset K$ by the set of maximizers of the kinetic energy $$\int_{\bbR^2}|\bfu|^2\q d\bfx,\qd\mbox{ where } \q\bfu:=k*\omg\qd\mbox{ for each }\q\omg\in K.$$ Then the maximal kinetic energy in the class $K$ is finite, and the set $S$ is nonempty. Moreover, there exist absolute constants $C_0, C_1, C_2\in(0,\ift)$ satisfying $C_1<C_2$ such that, for each $\omg\in S$, there exists a constant $\tau\in\bbR$ and a set $\Omg\subset\bbR^2_+$ which is open bounded, satisfying the following properties:\\

\noindent (i) By denoting $\omg_\tau$ as the translation of $\omg$ of amount $(\tau,0)$ by  $\omg_\tau(x_1,x_2):=\omg(x_1-\tau,x_2)$, we have 
$$\omg_\tau=\bfone_\Omg-\bfone_{\Omg_-}\qd\mbox{ a.e. in }\q\bbR^2.$$ Moreover, we get $$\Omg\q=\q\big\{\q\bfx\in\bbR^2_+\q:\q\psi[\omg_\tau](\bfx)-Wx_2>0\q\big\}$$ with the constant $W>0$ given by $$W:=\f{1}{3}\int_{\bbR^2}|\bfu|^2\q d\bfx\qd\mbox{ where }\bfu:=k*\omg.$$  

\noindent(ii) There exists a function $l:(0,\ift)\to[0,\ift)$ such that we have
$$\Big\{\q(x_1,x_2)\in\Omg\q:\q x_2=s\q\Big\}\q=\q 
\Big\{\q(x_1,s)\in\bbR^2_+\q:\q|x_1|< l(s)\q\Big\}\qd\mbox{ for each }s>0.$$  Moreover, we have   $$\Omg\q\subset\q \big\{\q\bfx\in\bbR^2\q:\q|\bfx|<C_0\q\big\}\qd\mbox{ and }\qd
    C_1\q<\q|\Omg|\q<\q C_2$$ where $|\cdot|$ denotes the (Lebesgue) measure in $\bbR^2$.\\

\end{thm}

\begin{rmk} In Theorem \ref{Main thm: variational method}, the traveling speed is proportional to the maximal kinetic energy. Moreover, observe that the condition $\int_A x_2\q d\bfx=1$ in the definition of $K$ can be replaced by $\int_A x_2\q d\bfx=\mu$ for any $\mu\in(0,\ift)$. Indeed, the scaling  $$\omg'(\bfx):=\omg(\mu^{-1/3}\bfx)\qd\mbox{ for }\q \omg\in S$$  gives a one-to-one correspondence between the sets of maximizers. Then we get the same results with adjusted constants $W', C_0', C_1', C_2'>0$ under the relations 
$${W'=\f{1}{3\mu}\int_{\bbR^2}|\bfu'|^2\q d\bfx=W\mu^{1/3},}\qd C_0'=C_0\mu^{1/3},\qd C_1'=C_1\mu^{2/3},\qd C_2'=C_2\mu^{2/3}.$$ \end{rmk}

Theorem \ref{Main thm: central speed} below says that, for each traveling vortex patch that is symmetrically concentrated to a vertical line with odd symmetry with respect to the horizontal line and travels with a constant velocity $(W,0)$, the fluid velocity at the center of the dipolar vortex has the direction $(1,0)$, with the speed greater than $2W$. Its proof will be deduced from the proof of Lemma \ref{lem_ central speed > 2W }.

\begin{thm}\label{Main thm: central speed} Let $W>0$ be any constant. Let $\Omg\subset\bbR^2_+$ be an open bounded set satisfying that $\displaystyle\int_\Omg x_2\q d\bfx>0$, and for some function $l:(0,\ift)\to[0,\ift)$ assume that
$$\Big\{\q(x_1,x_2)\in\Omg\q:\q x_2=s\q\Big\}\q=\q 
\Big\{\q(x_1,s)\in\bbR^2_+\q:\q|x_1|< l(s)\q\Big\}\qd\mbox{ for each }s>0,$$ and that the function $\tht:[0,\ift)\times\bbR^2\to\bbR$ given by  $$\tht(t,\bfx):=\Big(\bfone_\Omg-\bfone_{\Omg_-}\Big)(\bfx-(W,0)t)$$ solves \eqref{eq_Euler eq.}. 
Then for the velocity function $\bfu:[0,\ift)\times\bbR^2\to\bbR^2$ defined by $\bfu=k*\tht,\q$ we have $$u^1\Big(t,(Wt,0)\Big)=u^1\Big(0,(0,0)\Big)>2W\qd\mbox{ and }\qd u^2\Big(t,(Wt,0)\Big)=u^2\Big(0,(0,0)\Big)=0 \qd\mbox{ for each }t>0. $$    
\end{thm}

\begin{rmk}\label{rmk_main thm}[Proof of Theorem \ref{Main thm 1} (i)--(ii) assuming Theorem \ref{Main thm: variational method} and Theorem \ref{Main thm: central speed}] 
    By Theorem \ref{Main thm: variational method}, there exists a function $\omg$ having the odd symmetry with respect to $x_1$-axis such that we have $\int_{\bbR^2_+}x_2\omg\q d\bfx=1$ and the relation $$\omg=\bfone_{\set{\psi[\omg]-Wx_2>0}}\qd\mbox {a.e. in }\bbR^2_+$$ where the set $\set{\q\psi[\omg]-Wx_2>0\q}\cap\bbR^2_+$ is symmetrically concentrated along $x_2$-axis. Note that $\omg$ solves \eqref{eq_steady Euler eq.} and generates a traveling wave solution of \eqref{eq_Euler eq.}. Then by Theorem \ref{Main thm: central speed}, we have 
     $$\Bigg(\f{\rd}{\rd x_2}\psi[\omg]\Bigg)(0,0)=u^1\Big(0,(0,0)\Big)>2W.$$  Due to the continuity of $\nb\psi[\omg]$ in $\bbR^2$, there exists $R>0$ such that $$\Bigg(\f{\rd}{\rd x_2}\psi[\omg]\Bigg)>2W\qd \mbox{ in }\q\q B_R(0,0):=\set{\q\bfx\in\bbR^2\q:\q |\bfx|<R\q}.$$ This implies that we obtain 
     $$B_R(0,0)\cap\bbR^2_+\q\subset\q \set{\q\psi[\omg]-Wx_2>0\q}\cap\bbR^2_+,$$ since the following inequality holds for each $\bfx\in B_R(0,0)\cap\bbR^2_+$:
    $$\psi[\omg](\bfx)=\ii{0}{x_2}\Bigg(\f{\rd}{\rd x_2}\psi[\omg]\Bigg)(x_1,s)\q ds\q>\q 2Wx_2\q>\q Wx_2.$$ Here we used the fact that $\psi[\omg](x_1,0)=0$ for any $x_1\in\bbR$, which is a consequence of the odd-symmetry of $\omg$. 
\end{rmk}
{
\begin{rmk}
Remark \ref{rmk_main thm} above   says that each $\omg\in S$ in Theorem \ref{Main thm: variational method} has its own \textit{touching} radius $R=R(\omg)>0$ in the sense that 
  $$B_R(-\tau,0)\cap\bbR^2_+\q\subset\q \set{\q\psi[\omg]-Wx_2>0\q}\cap\bbR^2_+,$$ where $\tau\in\mathbb{R}$ is determined by $\omg$ so that $\omg$ is even symmetric with respect to the line $x_1=-\tau$.  Since 
the uniqueness of maximisers (up to translations) is unknown, the radius $R=R(\omg)>0$ may vary for each  $\omg\in S$ of maximisers. However, we can get an estimate about $R$ which is \textit{uniform} on the  set $S$ of maximisers in the following sense: \ \\

  The upper bound $R\leq C_0$ is simply written in $(ii)$ in Theorem \ref{Main thm: variational method}. For a lower bound, 
    we first recall the log-Lipschitz continuity of $\nb\psi(\omg)$:
    $$\Big|\nb\psi(\bfx)-\nb\psi(\bfy)\Big|\q\leq \q C\nrm{\omg}_{L^1\cap L^\ift}|\bfx-\bfy|\log(1/|\bfx-\bfy|)
    \q \qd\mbox{ for }\q|\bfx-\bfy|\leq 1/2.$$
 Thanks to $\|\omg\|_{L^\infty}=1$ and $\|\omg\|_{L^1}\leq C_2$ from $(ii)$ of Theorem \ref{Main thm: variational method}, we can get
$$
   \Big|\nb\psi(\bfx)-\nb\psi(\bfy)\Big|\q\leq \q C |\bfx-\bfy|^{1/2}
    \q \qd\mbox{ for }\q|\bfx-\bfy|\leq 1/2,
$$    where $C>0$ is a universal constant. 
  We claim that for each $\omg\in S$, we have  $$R\geq R_0:= \min\{\frac{W^2}{C^2},\frac{1}{2}\}>0.$$ For a proof, we may assume $\tau=0$. Then from the estimate $\psi_{x_2}(0,0)>2W$, we get 
    $$\psi_{x_2}(\bfx)\geq \psi_{x_2}(0,0)-C|\bfx|^{1/2}>2W-W=W\qd\mbox{ for each }\bfx\in B_R(0,0)\cap\bbR^2_+,$$ which leads to $$\psi(\bfx)=\ii{0}{x_2}\Bigg(\f{\rd}{\rd x_2}\psi\Bigg)(x_1,s)\q ds\q>\q Wx_2\qd\mbox{ for each }\q \bfx\in B_R(0,0)\cap\bbR^2_+.$$ It gives $R\geq R_0>0$.  
\end{rmk}
 }

Lastly, Theorem \ref{Main thm: bdry} below implies Theorem \ref{Main thm 1} (iii). For its proof, see the proof of Theorem \ref{thm_boundry is continuous}.

\begin{thm}\label{Main thm: bdry}
    Let $S$ be the set of functions in Theorem \ref{Main thm: variational method}. Choose any $\omg\in S$ with the corresponding function $l:(0,\ift)\to[0,\ift)$. Then $l$ is continuous in $(0,\ift)$ and the limit $$\q\displaystyle\lim_{s\searrow0}l(s)\in(0,\ift)\q\quad\mbox{ exists.}$$ Moreover, there exists $L\in(0,\ift)$ such that, for each $\eps\in(0,L)$, we have $l\in C^{1,r}([\eps,L])\qd\mbox{ for each }r\in(0,1).$
\end{thm}

\subsection{Key ideas}\label{subsec_ideas}  \q 

\noindent Let us describe the key ideas that are involved in the proof of the main theorem.

\medskip 

\noindent\textbf{Steady Euler equations and vorticity function}. 
We recall the steady Euler equation \eqref{eq_steady Euler eq.} in $\bbR^2_+$, for an arbitrary traveling speed $W>0$. This equation allows a solution of the form $$\omg=f(\Psi)\qd\mbox{ in }\q\bbR^2_+$$ for a function $f$ called a \textit{vorticity function} where the adjusted stream function $\Psi$ is given by 
$$\Psi(\bfx):=\psi[\omg](\bfx)-Wx_2-\gmm\qd\mbox{ for some constant }\gmm\geq0$$ with the relation $\bfU=\nb^\perp\Psi$. The constant $\gmm$ is referred to as a \textit{flux constant}. We extend the odd-symmetric solution $\omg$ of \eqref{eq_steady Euler eq.} as
 $$\omg(x_1,-x_2)=-\omg(x_1,x_2)\qd\mbox{ for each } (x_1,x_2)\in\bbR^2_+.$$
In this paper, we put the Heaviside step function 
\begin{equation}\label{vort_heavi}
 f(\cdot)=\bfone_{\set{\cdot>0}},
\end{equation}
  {(which is the limit of 
$
f(t)=t^{1/(p-1)}_{+}
$ as $p$ goes to infinity\footnote{cf. Lamb--Chaplygin dipole for $p=2$ (e.g. see \cite{AC2019}).} )}
as the vorticity function to obtain an odd-symmetric vortex patch solution given by
\begin{equation}\label{eq_traveling vortex formula}
    \omg=f(\Psi)=\bfone_{\set{\psi[\omg]-Wx_2-\gmm>0}}\qd\mbox{in }\q\bbR^2_+.
\end{equation} In Theorem \ref{thm_maximizer}, we prove that such a solution exists as a maximizer of the kinetic energy among functions under some constraints on their impulse and mass. Each maximizer is symmetrically concentrated about $x_2$-axis in the upper half plane $\bbR^2_+$ (so-called \textit{Steiner symmetric about $x_2$-axis}; see Definition \ref{def_Steiner sym} for details).\\

\noindent\textbf{Touching implies zero flux constant ($\gmm=0$)}.
We here discuss by which conditions a solution \eqref{eq_traveling vortex formula} would be a Sadovskii vortex patch. We observe that any solution with positive flux constant $\gmm>0$ cannot touch $x_1$-axis. Indeed, if $\gmm>0$, then for all sufficiently small $x_2>0$ we have $$\bfx=(x_1,x_2)\q\notin\q\set{\q\bfx\in\bbR^2_+:\psi[\omg](\bfx)-Wx_2-\gmm>0\q},$$ due to  $\q\psi(x_1,x_2)\Big|_{x_2=0}\equiv0.\q$  In other words, for $\omg$ to generate a Sadovskii vortex patch, we should have $\gmm=0$. However, it is not clear that $\gmm=0$ is a sufficient condition for \textit{touching} of the dipole, as the distribution of the stream function $\psi-Wx_2$ is unknown.\\

\noindent\textbf{Zero flux constant ($\gmm=0$) implies touching when symmetrically concentrated about $x_2$-axis}. 
In this paper, we show that any vortex patch solution, which has zero flux constant and is symmetrically concentrated about $x_2$-axis, always touches $x_1$-axis. Unlike \cite{Burton96}(also see \cite[Section 6]{AC2019}) which used the moving plane method \cite{GNN} to produce a Lamb--Chaplygin dipole, we directly compare the distribution of $\psi$ and $Wx_2$. More specifically, we want to compare the traveling speed $W>0$ and the horizontal velocity $u^1=\psi_{x_2}[\omg]$ to guarantee that $$\psi_{x_2}-W>0\qd\mbox{ near the origin},$$ so that we get $$B\cap\bbR^2_+\q\subset\q\set{\q\psi-Wx_2>0 \q}$$ for some ball $B\subset\bbR^2$ centered on $x_1$-axis. For the comparison, we first split the dipole as $$\omg=\bfone_A-\bfone_{A_-}=\omg_++\omg_-,\qd 
\bfu=\nb^\perp\psi[\omg_+]+\nb^\perp\psi[\omg_-]$$ where $\omg_+$ is the nonnegative (upper) part of $\omg$ whose support is in $\bbR^2_+$, and obtain the following relation that, on $x_1$-axis, 
$$u^1=2\psi_{x_2}[\omg_-]$$ thanks to the symmetry between $A$ and $A_-$. We may observe that the eddy flow generated by the lower part $\omg_-$ makes the upper part $\omg_+$ propagate in the direction $(W,0)$, and vice versa. That is, the traveling speed $W$ is the average value of the nonlocal velocity $\psi_{x_2}[\omg_-]$ in the set $\supp(\omg_+)\subset\overline{\bbR^2_+}$, i.e. we may compare the traveling speed given by $$W=\fint_{\supp(\omg_+)}\psi_{x_2}[\omg_-]\q d\bfx$$ and the horizontal velocity on $x_1$-axis, given by $\q u^1(x_1,0)=2\psi_{x_2}[\omg_-](x_1,0),\q$ if we can estimate the distribution of the nonlocal velocity $\psi_{x_2}[\omg_-]$ in the domain $\overline{\bbR^2_+}$. In Lemma \ref{lem_ central speed > 2W }, we will prove that the nonlocal velocity $\psi_{x_2}[\omg_-]$ has certain monotonicity and attains its unique maximum at the origin among all points in $\overline{\bbR^2_+}$ if the set $\supp(\omg)$ is symmetrically concentrated about $x_2$-axis. Therefore we obtain $$u^1(0,0)>2W,$$ which means that the speed at the center of the dipole is more than twice as fast as the traveling speed of the entire dipole. Therefore, $\gmm=0$ is an equivalent condition for \textit{touching}, when the dipole vortex is symmetrically concentrated about $x_2$-axis. From now on, we introduce the variational method to obtain a solution with $\gmm=0$. \\

\noindent\textbf{Variational problem on exact impulse and mass upper bound}. 
It is classical to use a variational method when seeking a steady solution to the Euler equations. In this paper, in some admissible set, we shall find a maximizer $\omg\in L^1(\bbR^2)$ of the kinetic energy $$E(\omg)=\f{1}{4}\int_{\bbR^2}|\bfu|^2\q d\bfx,\qd \bfu=k*\omg.$$ Given a triple of parameters
$$\big(\mu,\q\nu,\q\lmb\big)\qd\mbox{ where }\mu,\nu,\lmb>0,$$ the admissible set
$K_{\mu,\nu,\lambda}$
 will be taken as the set of functions having the simple forms $\omg=\lmb\cdot(\bfone_A-\bfone_{A_-})$ for some $A\subset\bbR^2_+$, with the prescribed impulse condition $$\int_{\bbR^2}x_2\omg\q d\bfx=\mu$$ and the upper bound of mass $$\nrm{\omg}_{L^1(\bbR^2)}\leq \nu.$$ Here, $\lmb$ is the strength of the vortex patch. 
We denote the set of maximizers as $S_{\mu,\nu,\lmb}$ and the maximal energy as $I_{\mu,\nu,\lmb}$ (see Remark \ref{rmk_prop.of.G} for detailed discussions). The simple scaling argument of Remark \ref{rmk_prop.of.G}-(v) says that we can kill two parameters of the triple $(\mu,\nu,\lmb)$. For instance, the variational problem of parameters $(\mu,\nu,\lmb)$ is equivalent to the problem of $(M,1,1)$ where $M=\lmb^{\f{1}{2}}\nu^{-\f{3}{2}}\mu$. Indeed, the sets $S_{\mu,\nu,\lmb}$ and $S_{M,1,1}$ are bijective under some scaling map. Therefore we may assume $\nu=\lmb=1$ and regard the impulse $\mu$ as the only parameter of our variational problem.\\

\noindent\textbf{Existence of maximizer}. 
Theorem \ref{thm_maximizer} says that, in the variational problem of the parameter $(\mu,1,1)$, there exists a maximizer, and any maximizer $\omg\in S_{\mu,1,1}$ has the form \eqref{eq_traveling vortex formula}: 
$$\omg=\bfone_{\set{\psi[\omg]-Wx_2-\gmm>0}}\qd\mbox{ a.e. in }\q\bbR^2_+\qd\mbox{ for some constants }\q W>0,\q\gmm\geq0$$ with the odd-symmetry in $\bbR^2$. The proof consists of several steps: \\

\noindent(i) In Lemma \ref{lem_existence of maximizer in larger set}, we prove the existence of maximizers having the form of a simple function $\bfone_A-\bfone_{A_-}$. We first find a weak subsequential limit $\omg'\in L^2(\bbR^2)$ of an energy-maximizing sequence $\omg_n$ in the larger admissible set: the set of odd-symmetric bounded functions $\omg$ satisfying that $0\leq\omg\leq1$ in $\bbR^2_+$  with the upper bounds of impulse and mass $$\int_{\bbR^2}x_2\omg\q d\bfx\leq\mu,\qd \nrm{\omg}_{L^1(\bbR^2)}\leq1.$$ We may assume that $\omg'$ and $\omg_n$ satisfy the Steiner symmetrization condition, since we have $$E(\omg^*)\geq E(\omg)$$ if $\omg^*$ is the Steiner symmetrization of $\omg$ about $x_2$-axis. We will see that the sequence of kinetic energy $E(\omg_n)$ converges to $E(\omg')$, implying that $\omg'$ is a maximizer of kinetic energy in the larger admissible set. Then we prove that $\omg'$ has the full impulse $\int_{\bbR^2}x_2\omg'\q d\bfx=\mu$ and has the form $\omg'=\bfone_A$ in $\bbR^2_+$.\\

\noindent(ii) Proposition \ref{prop_maximizer is TW} says that each maximizer $\omg$ solves \eqref{eq_steady Euler eq.} in the sense that, for some constant $W,\gmm\geq0$, the relation \eqref{eq_traveling vortex formula} holds. Knowing that $\omg$ is a simple function $\bfone_A-\bfone_{A_-}$ for some measurable set $A\subset\bbR^2_+$, we first prove that there exist some constant $W,\gmm\geq0$ such that $$A=\set{\q\bfx\in\bbR^2_+:\psi[\omg](\bfx)-Wx_2-\gmm>0\q}.$$ For this, we generate a perturbation of $\omg=\bfone_A-\bfone_{A_-}$ in the larger admissible set. We shall construct a function $\eta\in L^\ift_c(\bbR^2_+)$ satisfying that $$\eta\leq0\q\mbox{ in }A,\qd \eta\geq0\q\mbox{ in }A^c,\qd\mbox{and }\int_{\bbR^2_+}\eta=\int_{\bbR^2_+} x_2\eta=0.$$ Then for small $\eps>0$, we have $0\leq\q\bfone_A+\eps\cdot\eta\q\leq1$, and the function $\q\bfone_A+\eps\cdot\eta\q$ has the same mass and impulse with $\bfone_A$. The perturbed kinetic energy $E_\eps$ of the odd-extension of $\q(\bfone_A+\eps\eta)\q$ is maximized at $\eps=0$ so as to $$\f{d}{d\eps}E_\eps\Big|_{\eps=0}=\int_{\bbR^2_+}\psi[\omg]\cdot\eta\q d\bfx\leq0.$$ We will use this inequality to prove that $A=\set{\psi[\omg]-Wx_2-\gmm>0}$ for some constants $W,\gmm$, and the construction of $\eta$ is the key part of the proof. We will define $\eta$ so that the above inequality gives 
$$\int_{\bbR^2_+}(\psi[\omg]-Wx_2-\gmm)\q h\q d\bfx\leq0$$ for an arbitrary function $h$ satisfying that $h\leq0$ in $A$ and $h\geq0$ in $A^c$, which would eventually gives us the result. Meanwhile, the emergence of the constants $W,\gmm$ from the formula of $\eta$ are necessary. However, we observe that $W,\gmm$ are determined by the boundary points of the set $\set{\psi[\omg]-Wx_2-\gmm>0}$ (see \eqref{eq_W,gmm are unique}), while we yet only know that $A$ is a measurable set. It means that the classical definition of (topological) boundary points is not applicable, and we shall define boundary-like points of general measurable sets which are called \textit{the exceptional points} (see Definition \ref{def_excep.pt}). This approach was developed in \cite{Stability.of.Hill.vortex}.  \\

\noindent\textbf{Small impulse implies $\gmm=0$}. 
For the variational problem with parameters $(\mu,1,1)$, Lemma \ref{lem_small impulse to gmm 0} and the estimate \eqref{eq_small impulse-mass} in Lemma \ref{lem_ests for small impulse} guarantee that, if $\mu$ is small enough, then we can get $\gmm=0$, i.e. \textit{touching}, and prove that maximizers cannot achieve the full mass. It confirms our heuristic intuition for the case of $\lmb=1$: if the impulse imposed on the maximizers is small enough compared to the mass, each maximizer would rather be concentrated in a bounded region close to the axis than be stretched widely and thin, allowing for mass loss. On the other hand, one may infer that if the impulse dominates over the mass sufficiently, the maximizers would move away from the axis, that is, $\gmm>0$ will be large enough.  \\

\noindent\textbf{Estimates using small impulse: admissible set without mass bound}. 
Another result of dealing with small impulse, in addition to \textit{touching}, is that we can find maximizers of the variational problems where the mass bound is omitted. In the below, we explain the reason why obtaining the following estimate: 
$$\sup_{\omg\in S_{\mu,1,1}}\nrm{\omg}_{L^1}\q\lesssim\q\mu^{2/3}\qd\mbox{ for small }\mu>0$$ is sufficient to guarantee the existence of the maximizer of the variational problem without mass bound. The above estimate follows from \eqref{eq_small impulse-W} in Lemma \ref{lem_ests for small impulse}, and we will shortly give the key idea of the estimate.   \\

\noindent(i) \textit{Variational problem without mass bound and prior estimate: } We first consider the variational problem of maximizing the kinetic energy in an admissible set given as the set of $L^1$ simple functions $\omg=\bfone_A-\bfone_{A_-}$ having the constraint only on the impulse $$\int_{\bbR^2}x_2\omg\q d\bfx=\mu\qd\mbox{ for }\q\mu>0,$$  where the set of its maximizer functions is denoted by $S_{\mu,\ift}$, and the maximal energy is denoted by $I_{\mu,\ift}$. 
Assuming that the set $S_{1,\ift}$ is nonempty, we can obtain $\omg\in S_{\mu,\ift}$ from each $\omg'\in S_{1,\ift}$ by the scaling $\omg(\cdot)=\omg'(\mu^{-1/3}(\cdot)),$ and get the priori estimate 
\begin{equation}\label{eq_priori.est.}
    \nrm{\omg}_{L^1}\sim\mu^{2/3}.
\end{equation}

\noindent(ii) \textit{Difficulty: } It is nontrivial that the set $S_{1,\ift}$ is nonempty because it is not clear that $$I_{\mu,\ift}<\ift.$$ Indeed, the kinetic energy is bounded using $\nrm{\omg}_{L^1}$ from the inequality \eqref{eq_energy:finite}:
$$E(\omg)\q\lesssim\q \nrm{\omg}_{L^1}\nrm{\omg}_{L^2}^{1/2}\Big(\int_{\bbR^2}x_2\omg\q d\bfx\Big)^{1/2},$$
and it is not clear that $$\sup_{\omg\in S_{1,\ift}}\nrm{\omg}_{L^1}<\ift.$$ 
Although the kinetic energy of each function in the admissible set is finite, the maximal energy may be infinite, and there might not be any maximizer in the admissible set. \\

\noindent(iii) \textit{Small impulse estimate guarantees existence of maximizers: }
In this paper, in Lemma \ref{lem_ests for small impulse}, we  obtain the estimation : for small $\mu>0$,  
\begin{equation}\label{eq_intro_mass bound}
    \sup_{\omg\in S_{\mu,1,1}}\nrm{\omg}_{L^1}\q\lesssim\q\mu^{2/3}.
\end{equation} 
One may notice that this estimate coincides well with the priori estimate. From now on, we explain how the inequality above would imply that the set $S_{1,\ift}$ is nonempty (For the formal proof, see Section \ref{sec_ no mass bound}). In the meantime, we will observe that the power $2/3$ is the critical number to solve the variational problem. First, to prove $S_{1,\ift}\neq\emptyset$, it suffices to show that there exists $\nu_1>0$ such that 
\begin{equation}\label{eq_energy stop increasing}
    S_{1,\nu,1}=S_{1,\nu_1,1}\qd\mbox{ for any }\q\nu>\nu_1.
\end{equation} It means that, for fixed impulse, any energy maximizer does not attain arbitrarily large mass, and that the maximal energy $I_{1,\nu,1}$ stops increasing for large $\nu>0$. It would imply that 
$I_{1,\ift}=I_{1,\nu_1,1}<\ift$ and therefore $S_{1,\nu_1,1}\subset S_{1,\ift}\neq\emptyset$. To show \eqref{eq_energy stop increasing}, we need to show that, for some $\nu_1>0$ and each $\nu>\nu_1$, it holds that $$\sup_{\omg\in S_{1,\nu,1}}\nrm{\omg}_{L^1}\leq\nu_1,$$ or equivalently $$\sup_{\omg\in S_{\nu^{-3/2},1,1}}\nrm{\omg}_{L^1}\q\leq\q\f{\nu_1}{\nu}.$$ If we put $\mu:=\nu^{-3/2}$, we obtain the inequality $$\sup_{\omg\in S_{\mu,1,1}}\nrm{\omg}_{L^1}\q\leq\q\nu_1\cdot\mu^{2/3}$$ which must holds for small $\mu>0$. Therefore, we note that the estimate \eqref{eq_intro_mass bound} is strong enough to prove that the set $S_{1,\ift}$ is nonempty. \\

\noindent (iv) \textit{To obtain \eqref{eq_intro_mass bound} for small impulse: }
In Lemma \ref{lem_ests for small impulse}, we obtain \eqref{eq_intro_mass bound} by several steps. We sketch the strategy of the proof here: for each $\omg\in S_{\mu,1,1}$ with the traveling speed $W>0$ and the flux constant $\gmm\geq0$, we have $$\nrm{\omg}_{L^1}\leq\big|\q\set{\bfx\in\bbR^2_+:\psi[\omg](\bfx)>Wx_2}\q\big|$$ where $|\cdot|$ denotes the (Lebesgue) measure in $\bbR^2$. One may find it inevitable to find the lower bound of $W$ in terms of small $\mu>0$ for estimating the measure of the set $\set{\q\psi[\omg]>Wx_2\q}$, knowing that function $\psi[\omg]$ can be controlled by impulse $\mu$ (see Lemma \ref{lem_est.of.G} and Lemma \ref{lem_gw:estiate}). We can show that $$\mu^{1/3}\lesssim W$$ which follows by the Poho\v{z}aev identity \eqref{eq_energy and W} in Lemma \ref{lem_WM=energy} (see \cite{Burton.1988}):
$$I_{\mu,1,1}=\f{3}{4}W\mu+\f{1}{2}\gmm\nrm{\omg}_{L^1}\qd(\mbox{ where we know }\gmm=0\q\mbox{ for small }\mu\q)$$  and by the lower bound estimate of maximal energy $$I_{\mu,1,1}\gtrsim \mu^{4/3}\qd(\mbox{ thanks to a scaling })$$ given in Lemma \ref{lem_lower bound of energy}. \\


 \subsection{Further technical discussions} \label{subsec_variational} \q
  
\medskip
 
\noindent \textbf{Variational methods for vortex dipoles}.
A vortex dipole is a symmetric pair of vortices with opposite signs that move along the symmetry axis. This type of dipole is often referred to as a \textit{counter-rotating vortex pair}. It is a model of coherent vortex structures in large-scale geophysical flows (for experimental works, see  \cite{VF89}, \cite{FV94} and references therein). The existence via variational method goes back to  Kelvin in 1880 \cite{kelvin1880}, Arnold \cite{Arnold66} in 1966, and Benjamin in 1976 \cite{Ben76}.
 In particular, we consider the \textit{vorticity method},  a variational principle based on the \textit{vorticity}. This method for vortex pairs was further developed by Turkington \cite{Tu83} and Burton \cite{Burton.1988}.
 We refer to the work of Norbury \cite{Norbury75} and Yang \cite{Yang91} for the stream function method.\\

Our variational setting is based on the vorticity method of Friedman--Turkington \cite{FT81} in 1981  which was developed for vortex rings. Following \cite{FT81}, we consider the class of vorticity \begin{equation*}
	\begin{split}
		\int_{\mathbb{R}^2_+}\omega \quad \mbox{\textit{is smaller than or equal to given } } \nu>0. 
	\end{split}
\end{equation*} After fixing $\nu$ (e.g. $\nu=1$), as in \cite{FT81, AC2019, Stability.of.Hill.vortex}, the maximizers for small impulse $\mu=\int_{\mathbb{R}^2_+}x_2\omega>0$ \textit{strictly lose} part of its mass (i.e. $\int_{\mathbb{R}^2_+}\omega<1$). Then we prove that the parameter $\nu$ on mass becomes unnecessary. In other words, the maximizers are characterized by a single constraint on impulse (when small).
Our Sadovskii vortex patch is obtained as such a maximizer in $K_{\mu,1,1}$ (i.e. $\omega\in S_{\mu,1,1}$ for $\mu\ll 1$). Thus the  Sadovskii   patch can be considered as the limiting   case $\mu\to 0$ of the one parameter family 
$\mu>0$.  \\

\noindent \textbf{Approximating the point vortex dipole}.
The opposite limit ($\mu\to\infty$ in $S_{\mu,1,1}$) can be understood in the following sense.
By scaling, we observe 
$$
\omega\in K_{\mu,1,1} \iff \hat{\omega}\in K_{1,1,(1/\varepsilon^2)}
$$ via  the transform
$$\hat{\omega}(\bold{x})=\frac{1}{\varepsilon^2}\omega(\frac{\bold{x}}{\varepsilon})\quad \mbox{with} \quad \mu=1/\varepsilon.
$$
The class $K_{1,1,(1/\varepsilon^2)}$ consists of 
patches of strength $1/\varepsilon^2$ with the unit vertical center of mass ($
 1=\int_{\mathbb{R}^2_+}x_2\omega/\int_{\mathbb{R}^2_+} \omega
$). 
As $\varepsilon$ goes to $0$, any sequence $\{\omega_\varepsilon\}$ from the class $K_{1,1,(1/\varepsilon^2)}$   converges to the Dirac point mass: in other words, 
\begin{equation}\label{eq_point_mass}
\omega_\varepsilon \longrightarrow 
\delta_{(0,1)}-\delta_{(0,-1)}
\quad \mbox{in}\quad \mathbb{R}^2.
\end{equation}
The study of approximations to the point dipole, known as the \textit{desingularization problem}, has been a subject of interest and research for a long time (e.g. see  the classical book \cite{Lamb} of Lamb).
We refer to \cite{Tu83, Yang.Kubota.1994, MR2729322, MR3607460, MR4295232} and references therein. In general, for small $\varepsilon$, the boundary is smooth and closed to a circle, {with the very recent paper \cite{CQZZ.2025} settling the uniqueness in this regime}. For an analogue in vortex ring, see the Helmholtz's rings of small cross-section \cite{Helm1858}.  One may wonder interaction between vortex dipoles/rings. For instance, we refer to \cite{MR2012616, https://doi.org/10.1002/cpa.22199, Hass_leap} and references therein for leapfrogging phenomena.  

\medskip

\noindent\textbf{Lamb--Chaplygin dipole and Hill's spherical vortex}.
While we do not know the exact shape of dipoles including our Sadovskii vortex patch, there are other interesting examples  which can be explicitly written. One is the Lamb--Chaplygin dipole \cite{Lamb} while the other is the Hill's spherical vortex \cite{Hill}. We also refer to \cite{Hicks}, \cite{Moffatt}, \cite{Abe_swirl} and  references therein. The \textit{Lamb--Chaplygin dipole}  is an exact solution to \eqref{eq_steady Euler eq.} with zero flux constant $\gmm=0$ and the vorticity function given by the Lipschitz function 
 \begin{equation}\label{eq_vorticity.fct.of.Lambdipole}
      f(\cdot)=\max\set{\cdot,0}
 \end{equation}
(see \cite{Lamb}). 
Its support is simply a closed ball centered at the origin, and the dipole is symmetrically placed with opposite signs with respect to $x_1$-axis.  Thanks to the structure of the vorticity function \eqref{eq_vorticity.fct.of.Lambdipole}, the moving plane method \cite{GNN} can be applied to characterize the dipole. A variational characterization of the dipole can be found in \cite{Burton96, Burton05b} (also see \cite{AC2019} for stability). For Hill's spherical vortex, we refer to \cite{Stability.of.Hill.vortex} and references therein.

\medskip
\noindent \textbf{Stability questions}.
 The variational method produces an extremizer of a conserved quantity in an appropriate admissible class. Since Arnold's work {\cite{Arnold66}}, it is quite classical to deduce (Lyapunov) stability  for steady solutions  characterized by critical points of the energy (also see the explanation in \cite{MR4713111}). The simplest example would be the circular patch (Rankine vortex), whose stability was rigorously proved by \cite{WP85}. (Strictly speaking, patch solutions are not regular enough for Arnold's argument to be directly applicable, cf. \cite{SV09, MR4417385}.)    
 For vortex pairs, orbital stability was first obtained in \cite{BNL13} for a certain solutions by the vorticity method. Stability of  Lamb--Chaplygin dipoles was  proved in    \cite{AC2019} (also see \cite{Wang.2024, MR4350517}).  The recent paper of Wang \cite{Wang.2024}, which is a refinement of Burton \cite{Burton.2021}, proved stability an approximation for point vortex \eqref{eq_point_mass}.
We also refer to \cite{MR3894915, MR4565676} and references therein.  All the stability results are restricted on  \textit{symmetric} perturbations. In general,  non-symmetric disturbances may increase the corresponding energy as explained in \cite[p2342]{Saffman.Szeto.1980}. 
Lastly,  uniformly rotating patch   solutions are commonly referred to as ``$V$--states''. For existence, rigidity and stability, we refer to  \cite{Burbea1982,  Wan86,Tang, HMV, CCG,MR4312192, MR4493276} and references therein.   \\

 
  \noindent\textbf{Vortex atmosphere is touching}.
  In the steady Euler equations \eqref{eq_steady Euler eq.}, we consider a solution $\omg$ satisfying the Steiner symmetry condition and having a form $\bfone_A-\bfone_{A_-}$ where $$A=\set{\q\psi[\omg]-Wx_2-\gmm>0\q}$$ for some constant $W>0,\q \gmm\geq0$. Then each streamline $\set{\q\bfx\in\overline{\bbR^2_+}:\Psi(\bfx)=C\q}$, generated by the adjusted stream function $$\Psi:=\psi[\omg]-Wx_2-\gmm$$ and characterized by $C\in\bbR$, is a integral curve of the flow velocity. We note that if a single streamline forms a closed curve and surrounds any bounded subset of $\overline{\bbR^2_+}$, the set travels along with the traveling vortex patch without change in its profile. By the \textit{vortex atmosphere} we mean the union of all such subsets of $\bbR^2_+$. One may observe that the set $\set{\bfx\in\bbR^2_+:\psi(\bfx)-Wx_2>0}$ is the vortex atmosphere from the observations that the set $\set{\bfx\in\bbR^2_+:\psi(\bfx)-Wx_2-C>0}$ is not bounded for each $C<0$, and is bounded for $C\geq0$ (see the proof of Lemma \ref{lem_w:cpt_supp}). As we can prove that the set $\set{\psi-Wx_2>0}$ touches the axis of symmetry (regardless of whether $\gmm>0$ or $\gmm=0$), we may conclude that any vortex atmosphere touches the axis of symmetry even when the flux constant $\gmm$ of the solution is positive. In other words, any traveling dipole vortex patch which is symmetrically concentrated about $x_2$-axis carries its certain neighborhood fluid including a part of the axis of symmetry.  \\


 \subsection{Open problems}\label{subsec_open} \q

  \noindent We clarify several open problems regarding Sadovskii vortex.
 
 \medskip 
 
 \noindent \textbf{Detailed shape}. Due to limitations of the vortex method, unfortunately we are unable to determine many conjectured properties on the shape of the Sadovskii vortex patch. 
 
 \begin{con}
 	The bounded open set $\Omg$ defining a Sadovskii vortex patch (as described in Theorem \ref{Main thm 1}) is connected, convex, and makes the right angle with the horizontal axis. 
 \end{con}
 
 It seems that in general, one needs 
 certain regularity of the corresponding vorticity function
 in order to prove connectedness of maximizers from variational method    (e.g. see \cite{FB74, MR0560219}). Our discontinuous vorticity function $f$ in \eqref{vort_heavi} is not regular enough to deduce $\Omg$ consists of a single connected component. 
 
 Moreover, while several references claim to prove that the touching angle of $\Omg$ with the horizontal {axis} should be $90$ degrees using formal expansions of the velocity on the boundary (\cite{MR0734586,Yang91,Saffman.Tanveer.1982}), it is not even clear that the boundary function $l(s)$ given in Theorem \ref{Main thm 1} defines an angle; it might keep oscillating as $s\to 0^+$, although the size of the oscillations should go to zero from continuity of $l$ up to the boundary. However, under the assumptions that the angle is well-defined and $l$ is \textit{smooth up to the boundary} (not merely continuous), it can be made rigorous that the angle must be either $0, 90,$ or $180$ degrees (cf. \cite{EJSVP1}). We also remark that proving the existence of $V$--states (uniformly rotating vortex patches) having 90 degree corners seems to remain open (\cite{HMW,GaHa}).  
 
 Furthermore, the existence of Sadovskii vortex accompanying vortex sheets on its boundary seems to be a very difficult open problem. \\

 \medskip 
 
 \noindent \textbf{Variational properties}. There is a possibility that uniqueness still holds  for small impulse $\mu>0$ which would give stability for symmetric perturbations as in the case of Lamb--Chaplygin dipole \cite{AC2019} and Hill's vortex ring \cite{Stability.of.Hill.vortex}. Formally, we have the following conjecture. 
 
 \begin{con}
 	 There exists a universal constant $\mu^* > 0$ such that, for all $0<\mu\le\mu^*$, the set of maximizers $S_{\mu,1,1}$ coincides with $x_1$-translation and rescaling of a single patch, which then can be referred to as \textup{the} Sadovskii vortex patch. Furthermore, for all $\mu > \mu^*$, $S_{\mu,1,1}$ only contains patches which are separated from the horizontal axis. 
 \end{con}
 
 This conjecture is supported by the fact that so far no one was able to numerically find two different shapes of Sadovskii vortex patch, and if true, it would have significant consequences for the initial value problem of the 2D Euler equations. Furthermore, this would also explain why one could not continue, even at the numerical level, the bifurcation curve consisting of vortex dipoles beyond Sadovskii vortex patch. We remark that the analogous conjecture for the vortex rings was established recently in \cite{Stability.of.Hill.vortex}. 
 
 Related to this, Pierrehumbert \cite{Pierrehumbert.1980} and Saffman--Szeto \cite[p2342]{Saffman.Szeto.1980} discussed the following more general problem. 
 
 \begin{quest}
 	Is there a traveling wave  $\omega=\mathbf{1}_\Omega-\mathbf{1}_{\Omega_-} $ such that $\Omg$ is a bounded open set of $\bbR^2_+$ satisfying  $$
 	|\Omg|=1\quad\mbox{while} \quad \quad \mu:= \int_{\Omg} x_2 \,d\bfx < \mu^{*}? 
 	$$ 
 \end{quest}
 
 It is possible that there are no such traveling wave patches, even if we do not require them to be maximizers of the energy in a suitable class. We note that Pierrehumbert calculated up to $\mu=0.29$ and anticipated that $\mu^* \simeq 0.26$ is the possible limit. It almost matches to the limiting number in \cite[Table I]{MR0734586}, after scaling.

\subsection{Outline of the paper}\label{subsec_outline}

The rest of the paper is organized as follows. In Section \ref{sec:prelim}, we fix notations and collect several technical estimates. Existence of maximizers of patch-type for the variational problem is established in Section \ref{sec:exist-maximizer}. Then in Section \ref{sec:small-impulse}, we show that when the impulse is sufficiently small, the corresponding flux constant is zero and therefore the maximizing patch touches the axis. Lastly, in Section \ref{sec_ no mass bound}, we consider the variational problem without a mass bound on the admissible class, and show that the maximizer coincides with that of the variational problem with a mass bound. 

\section{Preliminaries}\label{sec:prelim}
\subsection{Notations}
\begin{itemize}
    \item We use the following notations for the norms in $\bbR^2_+$, 
    $$\nrm{f}_1\q:=\q\nrm{f}_{L^1(\bbR^2_+)}\qd (Mass),\qqd
    \nrm{f}_2\q:=\q\nrm{f}_{L^2(\bbR^2_+)},\qqd
    \nrm{x_2f}_1\q:=\q\int_{\bbR^2_+}x_2|f(\bfx)|\q d\bfx\qd (Impulse). $$
    \item For each $R>0$ and $\bfx\in\bbR^2$, we define the open ball
    $$B_R(\bfx)\q:=\q\{\q\bfy\in\bbR^2\q:\q|\bfy-\bfx|<R\q\}.$$
    \item For each measurable set $A\in\bbR^2$, we define the characteristic function $\bfone_A:\bbR^2\to\bbR$ as 
    $$\bfone_A(\bfx)\q:=\q  \left\{\begin{matrix}
    1\qqd\mbox{ if }\bfx\in A,\\
    \q0\qqd\mbox{ if }\bfx\in A^c
    \end{matrix}  \right.$$
    \item For each function $f:\bbR^2_+\to\bbR$, we can define its odd extension to $\bbR^2$,  denoted by $\tld{f}:\bbR^2\to\bbR$ and defined as 
    $$\tld{f}(x_1,x_2)\q:=\q\left\{\begin{matrix}
        f(x_1,x_2)\\
        -f(x_1,-x_2)\\
        0\\
    \end{matrix}
    \begin{matrix}
    \qqd\mbox{ if }x_2>0,\\
    \qqd\mbox{ if }x_2<0,\\
    \qqd\mbox{ if }x_2=0.
    \end{matrix}\right.$$ \end{itemize}
    
\subsection{Some technical estimates}\q

We will collect some elementary lemmas which will be used in the sequel. In $\bbR^2_+$, we denote the Green's function $G:\bbR^2_+\times\bbR^2_+\to\bbR$ defined as 
    \begin{equation}\label{eq_G}G(\bfx,\bfy)\q:=\q\f{1}{4\pi}\log\Bigg(1+\f{4x_2y_2}{|\bfx-\bfy|^2}\Bigg)>0,\qd \mbox{for}\qd  \bfx\neq \bfy.\end{equation}
    The following lemma implies that the Green function is integrable for each $x\in\bbR^2_+$, in the domain $\set{\bfy\in\bbR^2_+\q:\q y_2<\alp}$ for any $\alp>0$. This lemma will be used later in the proof of Lemma \ref{lem_small impulse to gmm 0} and Lemma \ref{lem_ests for small impulse}. \\
    
\begin{lem}\label{lem_est.of.G} There exists a constant $C>0$ such that, for each $\alp>0$ and $\bfx=(x_1,x_2)\in\bbR^2_+$, we have
\begin{equation}\label{eq_est.of.G}
\int_{0< y_2\leq\alp}G(\bfx,\bfy)\q d\bfy\q \leq\q C\q x_2^{1/2}\alp^{3/2}.
\end{equation}
\end{lem}

\begin{proof} For each $\alp>0$ and $\bfx=(x_1,x_2)\in\bbR^2_+$, we have 
$$\int_{0< y_2\leq\alp}G(\bfx,\bfy)\q d\bfy\q\leq\q \f{1}{4\pi}\int_{0< y_2\leq\alp}\log\Bigg(1+\f{4\alpha x_2}{y_1^2}\Bigg)\q d\bfy\q\leq\q\f{\alpha(4\alpha x_2)^\f{1}{2}}{4\pi}\int_\bbR\log\Big(1+\f{1}{t^2}\Big)\q dt\q\lesssim\q x_2^{1/2}\alp^{3/2},$$ using the fact that 
$$\ii{0}{\ift}\log\Big(1+\f{1}{x^2}\Big)\q d\bfx<\ift$$ which is the consequence of integration by parts.
\end{proof}
\begin{defn}\label{defn_gw}
    For each $\omg\in (L^1\cap L^2)(\bbR^2_+)$, we can define the function $\calG[\omg]:\bbR^2_+\to\bbR$ as 
$$\calG[\omg](\bfx):=\int_{\bbR^2_+}G(\bfx,\bfy)\omg(\bfy)\q d\bfy.$$ 
\end{defn}

The function $\calG[\omg]$ is well defined by the following lemma, which is a slight variation of \cite[Proposition 2.1]{AC2019}. We will prove that, for each $\omg\in (L^1\cap L^2)(\bbR^2_+)$ with $\nrm{x_2\omg}_1<\ift$, we get $\calG[\omg]\in L^\ift(\bbR^2_+)$.

\begin{lem}\label{lem_gw:estiate}
    There exists a constant $C_1\in(0,\ift)$ such that, for each $\omg\in (L^1\cap L^2)(\bbR^2_+)$, we have
    \begin{equation}\label{eq_gw<theta}
    \Big|\int_{\bbR^2_+}G(\bfx,\bfy)\omg(\bfy)\q d\bfy\Big|\q\leq\q C_1\q x_2^{1/2}\nrm{\omg}_1^{1/2} \nrm{\omg}_2^{1/2}\qd \mbox{ for each }\q\bfx=(x_1,x_2)\in\bbR^2_+.
    \end{equation}     
    If we additionally assume that $\nrm{x_2\omg}_1<\ift$, then there exists an absolute constant $C_2\in(0,\ift)$ such that 
    \begin{equation}\label{eq_gw:bdd}
        \Big|\int_{\bbR^2_+}G(\bfx,\bfy)\omg(\bfy)\q d\bfy\Big|\q\leq \q
        C_2\Big(\q
        \nrm{\omg}_1^{1/3} \nrm{\omg}_2^{1/3} \nrm{x_2\omg}_1^{1/3}+
         \nrm{\omg}_2^{1/2}\nrm{x_2\omg}_1^{1/2}
        \q\Big) \qd \mbox{ for each }\q\bfx=(x_1,x_2)\in\bbR^2_+,
    \end{equation}
    and that 
    \begin{equation}\label{eq_energy:finite}
    \Bigl|\int_{\bbR_+^2}\int_{\bbR_+^2}G(\bfx,\bfy)\omg(\bfx)\omg(\bfy)\q d\bfx d\bfy\Bigl|\q\leq\q
C_2\q\nrm{\omg}_1\nrm{\omg}_2^{1/2}\nrm{x_2\omg}_1^{1/2}.
\end{equation} 

\end{lem}

\begin{proof} The estimates \eqref{eq_gw<theta} and \eqref{eq_energy:finite} were borrowed from \cite[Proposition 2.1]{AC2019}. To show \eqref{eq_gw:bdd}, assume that $\nrm{x_2\omg}_1<\ift$ and split the integral as 
$$\Big|\int_{\bbR^2_+}G(\bfx,\bfy)\omg(\bfy)\q d\bfy\Big|\q\leq\q\Big|\int_{|x-y|\geq x_2/2}\Big|+\Big|\int_{|x-y|<x_2/2}\Big|\q=\q\mbox{(I) + (II)}.$$
For (I), using the inequality, $$\log(1+t)\q\leq\q t\qd \mbox{ for each } t>0,$$ we have 
$$\mbox{(I)}\q\leq\q
\int_{|x-y|\geq x_2/2}\f{x_2y_2}{\pi|\bfx-\bfy|^2}|\omg(\bfy)|\q d\bfy\q\leq\q \f{4}{\pi x_2}\nrm{x_2\omg}_1\q=: A.$$
On the other hand, by \eqref{eq_gw<theta}, we obtain 
$$\mbox{(I)}\q\lesssim\q x_2^{1/2}\nrm{\omg}_1^{1/2}\nrm{\omg}_2^{1/2}\q=:B.$$ In sum, we have 
$$\mbox{(I)}\q\lesssim\q A^\f{1}{3}B^\f{2}{3}\q\lesssim\q
 \nrm{\omg}_1^{1/3} \nrm{\omg}_2^{1/3} \nrm{x_2\omg}_1^{1/3}.$$
 For (II), we put $\calB:=B_{x_2/2}(\bfx)$ and apply \eqref{eq_gw<theta} to have 
 $$\mbox{(II)}\q=\q\Big|\int_{\bbR^2_+}G(\bfx,\bfy)\Big(\omg\cdot\bfone_\calB\Big)(\bfy)\q d\bfy\Big|\q\leq\q C x_2^{1/2}\nrm{\omg}_{L^1(\calB)}^{1/2}
 \nrm{\omg}_{L^2(\calB)}^{1/2}.$$ Here observe that 
$$\nrm{\omg}_{L^1(\calB)}\q\leq\q\f{2}{x_2}\int_\calB|y_2\omg(\bfy)|\q d\bfy\q\leq\q\f{2}{x_2}\nrm{x_2\omg}_1$$ due to the definition of $\calB$. Therefore we have 
$$\mbox{(II)}\q\lesssim\q\nrm{\omg}_2^{1/2}\nrm{x_2\omg}_1^{1/2}$$ which completes the proof of \eqref{eq_gw:bdd}.
\end{proof}

The following lemma gives the relations between $\calG[\omg]$ and $\omg$. For the regularity argument of $\calG[\omg]$, we refer to \cite[Lemma 8.1]{majda-bertozzi}, \cite[Appendix 2.3]{pulvirenti} (see also \cite{G-B.elliptic.pde}, \cite{Evans.PDE}).

\begin{lem}\label{lem_gw:C1}
For each $\omg\in (L^1\cap L^\ift)(\bbR^2_+)$, the function  $\tld{\calG[\omg]}$, which is the odd extension of $\calG[\omg]$ in Definition \ref{defn_gw}, satisfies $\tld{\calG[\omg]}\q\in\q C^{1,r}(\bbR^2)\q$ for each $r\in(0,1)$ and 
$$\tld{\calG[\omg]}\q\in H^2_{loc}(\bbR^2)\q\q\mbox{ with }\q
-\lap\q\tld{\calG[\omg]}\q=\q\tld{\omg}\qd\mbox{ a.e. }\q\mbox{ in }\q \bbR^2$$ where $\tld{\omg}$ is the odd extension of $\omg$ to $\bbR^2$. 
\end{lem}
\begin{proof} It is due to the relation $$\tld{\calG[\omg]}=\psi[\tld{\omg}]\qd\mbox{ in }\bbR^2$$ and potential estimates of $\psi$.
\end{proof}

We also obtain the asymptotic behavior of $\calG[\omg]$ and a different representation of the kinetic energy which were borrowed from \cite[Proposition 2.2]{AC2019}).

\begin{lem}\label{lem_gw:decay}
    For each $\omg\in (L^1\cap L^2)(\bbR^2_+)$ with $\nrm{x_2\omg}_1<\ift$, we have 
    $$\lim_{|\bfx|\to\ift}\q \calG[\omg](\bfx)\q=\q0.$$
    Moreover, we have 
    $$\f{1}{2}\int_{\bbR^2_+}\calG[\omg]\omg\q d\bfx=\f{1}{4}\int_{\bbR^2}\big|\bfu\big|^2\q d\bfx$$ 
    where $\bfu=k*\tld{\omg}$. Here $\tld{\omg}$ is the odd extension of $\omg$ to $\bbR^2$, and the kernel $k$ is given in the equation \ref{eq_Euler eq.}.
    \end{lem}
\begin{proof} By \eqref{eq_gw:bdd} in Lemma \ref{lem_gw:estiate}, there exists a constant $C>0$ such that, for any  $\omg\in (L^1\cap L^2)(\bbR^2_+)$ with $\nrm{x_2\omg}_1<\ift$ and for the sequence $\{\omg_n\}=\{\omg\cdot\bfone_{B_n(0)}\}$, we have  
$$\Big|\calG[\omg](\bfx)\Big|\q\leq\q 
\Big|\calG[\omg-\omg_n](\bfx)\Big|+
\Big|\calG[\omg_n](\bfx)\Big|\q\leq\q
 C\q\Bigg(C_n+\f{|\bfx|}{(|\bfx|-n)^2}\nrm{x_2\omg}_1\Bigg)\qd\mbox{ for each }|\bfx|>n,$$ where 
$$C_n:=\nrm{\omg-\omg_n}_1^{1/3} \nrm{\omg-\omg_n}_2^{1/3} \nrm{x_2(\omg-\omg_n)}_1^{1/3}+
 \nrm{\omg-\omg_n}_2^{1/2}\nrm{x_2(\omg-\omg_n)}_1^{1/2}.$$  Therefore we have 
 $$\limsup_{|\bfx|\to\ift}{|\calG[\omg](\bfx)|}\q\leq\q C\cdot C_n\qd\mbox{ for each } n\geq1.$$ Observe that $C_n\to0$ as $n\to\ift$. As we take the limit $n\to\ift$, we obtain  
 $$\lim_{|\bfx|\to\ift}\q \calG[\omg](\bfx)=0. $$This finishes the proof.  \end{proof}

\section{Existence of maximizers of patch-type }\label{sec:exist-maximizer} \q

In this section, we will use the variational method to prove that there exists a patch-type maximizer of the kinetic energy in the admissible class whose setting is from \cite{FT81} (see also \cite{AC2019}). Each maximizer will generate a traveling wave solution of \eqref{eq_Euler eq.}. The main result of this section is Theorem \ref{thm_maximizer}.

\subsection{Variational problem on exact impulse and mass bound}\q

Given $\mu,\nu,\lmb>0,$  we define an admissible set 
    \begin{eqnarray*}
    K_{\mu,\nu,\lambda}\q:=\q\Bigl\{\q\omg\in L^1(\bbR_+^2)\q:\q\omg=\lmb\cdot\bfone_A&\mbox{a.e.}&\mbox{where }A\subset\bbR_+^2 \mbox{ is (Lebesgue) measurable, }\\
&\qd&\nrm{x_2\omg}_1=\mu,\qd\nrm{\omg}_1\leq\nu\q\Bigl\}
    \end{eqnarray*}
    on which its energy functional $E(\cdot):K_{\mu,\nu,\lambda}\to\bbR$ is defined as 
    $$
    E(\omg)\q:=\q\frac{1}{2}\int_{\bbR_+^2}\calG[\omg](\bfx)\omg(\bfx)\q d\bfx=\q\frac{1}{2}\int_{\bbR_+^2}\int_{\bbR_+^2}G(\bfx,\bfy)\omg(\bfy)\omg(\bfx)\q d\bfy d\bfx.
    $$ 
    Note that, if $\omg$ has a bounded support, then Lemma \ref{lem_gw:decay} implies that 
    $$E(\omg)=\f{1}{4}\int_{\bbR^2}|\bfu|^2\q d\bfx$$ where $\bfu=k*\widetilde{\omg}$ with $k$ from \eqref{eq_Euler eq.} and $\widetilde{\omg}$ is the odd extension of $\omg$. We formulate the maximization problem by
$$
I_{\mu,\nu,\lambda}\q:=\sup_{\omg\in K_{\mu,\nu,\lambda}}E(\omg)$$
and define the set of maximizers as 
$S_{\mu,\nu,\lambda}:=\big\{\q\omg\in K_{\mu,\nu,\lambda}\q:\q E(\omg)=I_{\mu,\nu,\lambda}\q\big\}.$\\

\begin{rmk}\label{rmk_prop.of.G}\q\\
\noindent(i) Recall that, for each $(\bfx,\bfy)\in\bbR^2_+\times\bbR^2_+$ with $\bfx\neq \bfy$, we have $$G(\bfx,\bfy)\q=\q G(\bfy,\bfx)>0.$$
It follows that,for each $\q\omg\in K_{\mu,\nu,\lmb}$, we have $E(\omg)\geq0$ and $I_{\mu,\nu.\lmb}>0$. Moreover, we get $I_{\mu,\nu,\lmb}<\ift$ by \eqref{eq_energy:finite} in Lemma \ref{lem_gw:estiate}. \\
\noindent(ii) If $f_\tau:\bbR^2\to\bbR^2$ denotes the translation of some fixed amount $\tau\in\bbR$ along the $x_1$-axis, $\q$i.e. $$\q f_\tau(x_1,x_2)\q=\q (x_1+\tau,x_2),$$ then we have
$$G(\bfx,\bfy)\q=\q G(f_\tau(\bfx),f_\tau(\bfy)).$$
 Therefore, if $\omg\in K_{\mu,\nu,\lmb}$, then any translation $\omg_\tau:=\omg\circ f_\tau$ satisfies $$\omg_\tau\in K_{\mu,\nu,\lmb}\qd\mbox{ and }\qd E(\omg)=E(\omg_\tau).$$ 
\noindent(iii) If $g_h:\bbR^2\to\bbR^2$ denotes the lift by some fixed amount $h>0$ along the $x_2$-axis, $\q$i.e. $$\q g_h(x_1,x_2)\q=\q (x_1,x_2+h),$$ then for each $\omg\in K_{\mu,\nu,\lmb}$, we have  
 $$ E(\omg)<E(\omg\circ g_{-h}).$$ 
\noindent(iv) For each $\bfx,\bfy\in\bbR^2_+$ with $\bfx\neq \bfy$, we have 
 $$G(\bfx,\bfy)=G(r\bfx,r\bfy)$$ for any $r>0$. By this property, for each $\omg\in (L^1\cap L^\ift)(\bbR^2_+)$ with $\nrm{x_2\omg}_1<\ift$, the scaling map given by $\hat{\omg}(\bfx):=\omg(r\bfx)$ gives the relation 
 $$E(\hat{\omg})=r^{-4}E(\omg)\qd\mbox{ for each }r>0.$$
\noindent(v) As in \cite{FT81} (and also in \cite{AC2019}), we can reduce the maximization problem to the case containing a single parameter. Indeed, for each $\omg\in K_{\mu,\nu,\lambda}$, consider the scaling 
$$
\hat{\omg}(\bfx)=\lmb^{-1}\cdot\omg(\lmb^{-1/2}\nu^{1/2}\bfx).
$$
Observe that $\hat{\omg}\in K_{M,1,1}$ where $M=\lmb^{1/2}\nu^{-3/2}\mu>0$ with the energy $E(\hat{\omg})=\nu^{-2}E(\omg).$ Here, given $\mu,\nu,\lmb>0,$  the map $\omg\mapsto\hat{\omg}$ is a bijective map from $K_{\mu,\nu,\lmb}$ onto $K_{M,1,1}$, and from $S_{\mu,\nu,\lmb}$ onto $S_{M,1,1}$. Therefore, the two maximization problems below are congruent with each other, 
$$
I_{\mu,\nu,\lambda}\q=\sup_{\omg\in K_{\mu,\nu,\lambda}}E(\omg)
\quad\mbox{and}\quad
I_{M,1,1}\q=\sup_{\omg\in K_{M,1,1}}E(\omg),
$$ with the relation $I_{M,1,1}=\nu^{-2} I_{\mu,\nu,\lmb}\q$. We hereafter abbreviate the notations as 
$$K_M:=K_{M,1,1},\qd I_M:=I_{M,1,1},\qd S_M:=S_{M,1,1}.$$
\end{rmk}\q
 
\subsection{Existence of maximizers}\label{subsec_exis.of.maxi.}\q

For convenience, we start with a definition that will be frequently used in a sequel. 
\begin{defn}\label{def_Steiner sym}
    We say that a nonnegative measurable function 
    $f:\bbR^2_+\to\bbR$  satisfies the \textit{Steiner symmetry condition} if 
    $$f(x_1,x_2)=f(-x_1,x_2)\qd\mbox{ for each }\bfx\in\bbR^2_+,$$ and for each fixed $x_2>0$, the function $f(x_1,x_2)$ is non-increasing in $x_1\geq0$. We also say that a measurable set $A\subset\bbR^2_+$ satisfies the Steiner symmetry condition if the function $g:\bbR^2_+\to\bbR$ given by $g:=\bfone_A$ satisfies the Steiner symmetry condition.
\end{defn}\q

The main theorem of Subsection \ref{subsec_exis.of.maxi.} is Theorem \ref{thm_maximizer} below. We will prove Theorem \ref{thm_maximizer} at the end of Subsection \ref{subsec_exis.of.maxi.} after the proof of several key lemmas. Each lemma is quite classical and stems from \cite{Stability.of.Hill.vortex,AC2019}. 

\begin{thm}\label{thm_maximizer}
For any $M>0$, we have the following properties.\\

    \noindent(i) The set $S_M$, which is the set of maximizers in $K_M$, is nonempty. \\

   \noindent (ii) For each $\omg\in S_M$, there exist some constants $W>0,\q \gmm\geq0$ that are uniquely determined by $\omg$ satisfying
\begin{equation}\label{eq_maximizer}
    \omg\q=\q\bfone_{\big\{ \bfx\in \bbR^2_+\q:\q \calG[\omg](\bfx)-Wx_2-\gmm>0\big\}}\qd \mbox{a.e.} \qd \mbox{in} \qd \bbR^2_+.
\end{equation}

\noindent(iii) Each $\omg \in S_M$ has a bounded support and satisfies the Steiner symmetry condition in Definition \ref{def_Steiner sym} up to translation. More precisely, for each $\omg\in S_M$, there exist $\tau\in\bbR$ and an open bounded set $A\subset\bbR^2_+$ satisfying the Steiner symmetry condition such that, for the translation $\omg_\tau$ given by $\omg_\tau(\cdot):=\omg(\cdot+(\tau,0))$, we have $$\omg_\tau\q=\q\bfone_A\qd\mbox{ a.e. in }\q\bbR^2_+.$$ 
\noindent (iv) For each $\omg\in S_M$, if we have $\nrm{\omg}_1<1$, then we get $\gmm=0$, where $\gmm$ is given in (ii).
\end{thm}\q

Note that we have the following corollary by the scaling in Remark \ref{rmk_prob.congruent}.
\begin{cor}\label{cor_maximizer}
For any $\mu,\q\nu,\q\lmb>0$, the set $S_{\mu,\nu,\lmb}$ is nonempty. For each maximizer $\omg\in S_{\mu,\nu,\lmb}$, there exist some constants $W>0$, $\gmm\geq0$ that are uniquely determined by $\omg$ satisfying 
    $$
    \omg\q=\q\lmb\cdot\bfone_\Omg\qd \mbox{a.e.} \qd \mbox{in} \qd \bbR^2_+
    \qd\mbox{ where }\Omg\q=\q\Big\{\q \bfx\in \bbR^2_+\q:\q \calG[\omg](\bfx)-Wx_2-\gmm>0\q \Big\}.$$
    Moreover, each maximizer satisfies the Steiner symmetry condition in Definition \ref{def_Steiner sym} up to translation. In addition, if we have  $\nrm{\omg}_1<\nu$, then we get $\gmm=0$.
\end{cor}\q
\begin{rmk}\label{rmk_prob.congruent} Let $\mu,\nu,\lmb>0$. Recall that the scaling map $\omg\mapsto\hat{\omg}$ given by 
$$
\hat{\omg}(\bfx)=\lmb^{-1}\cdot\omg(\lmb^{-1/2}\nu^{1/2}\bfx)
$$ is a bijective map from $S_{\mu,\nu,\lmb}$ onto $S_M$ where $M=\lmb^{1/2}\nu^{-3/2}\mu>0$. Moreover, under the scaling map $\omg\mapsto\hat{\omg},\q$ for the constants $W,\hat{W}>0$ and $\gmm,\hat{\gmm}\geq0$ satisfying 
$$\omg\q=\q\lmb\cdot\bfone_{\big\{\calG[\omg]-Wx_2-\gmm>0\big\}}\qqd \mbox{and}\qqd\hat{\omg}\q=\q\bfone_{\big\{\calG[\hat{\omg}]-\hat{W}x_2-\hat{\gmm}>0\big\}}\qqd \mbox{a.e.} \qd \mbox{in} \qd \bbR^2_+,$$
one can easily check the relations $$\hat{W}=W\cdot\lmb^{-1/2}\cdot\nu^{-1/2}\qqd\mbox{and}\qqd\hat{\gmm}=\gmm\cdot\nu^{-1}.$$
\end{rmk}

\subsubsection{Existence of maximizers in a larger set}\q

For the proof of Theorem \ref{thm_maximizer}, for each $M>0$ we define a larger admissible set $\tld{K}_M\supset K_M$ as 
$$
\tld{K}_M\q:=\q\Bigl\{\q\omg\in L^1(\bbR_+^2)\q:\q0\leq\omg\leq1\q\mbox{ a.e. }\mbox{ in }\q\bbR_+^2,\qd\nrm{x_2\omg}_1\leq M,\qd\nrm{\omg}_1\leq1\q\Bigl\},
$$ 
and
$$
\tld{I}_M\q:=\q\sup_{\omg\in\tld{K}_M}E(\omg)\q\geq\q I_M,\quad\quad\tld{S}_M\q:=\q\Big\{\q\omg\in \tld{K}_M\q:\q E(\omg)=\tld{I}_M\q\Big\}.
 $$ 
Note that $K_M$ is a set of characteristic functions having full impulse $M$, and $\tld{K}_M$ is a set of nonnegative, bounded functions for which the bound of impulse is $M$. We will prove the existence of a maximizer in $\tld{K}_M$ and show that the maximizer lies actually on $K_M$. This would imply that $S_M=\tld{S}_M$, and therefore $S_M$ is nonempty.

\begin{lem}\label{lem_existence of maximizer in larger set}
For each $M>0$, the set $\tld{S}_M$ is nonempty. In other words, there exists $\omg\in\tld{K}_M$ such that  
$$E(\omg)\q=\q\tld{I}_M.$$
\end{lem} 

\begin{proof}\q Consider a maximizing sequence 
$\{\omg_n\}\subset\tld{K}_M$ , i.e. $$E(\omg_n)\to \tld{I}_M\qd\mbox{ as }n\to\ift.$$ Since $\nrm{\omg_n}_{L^2(\bbR_+^2)}\leq\nrm{\omg_n}_1^{1/2}\leq1$ for all $n\geq1,$ there exists a subsequence $\{\omg_{n_k}\}\subset\{\omg_n\}$ and $\omg\in L^2(\bbR_+^2)$ such that $\omg_{n_k}$ converges weakly to $\omg$ in $L^2(\bbR_+^2)$. For convenience, we assume $\omg_n\weakto\omg$ in $L^2(\bbR_+^2)$ as $n\to\ift$. It implies that $\omg\in \tld{K}_M.$ In other words, we have $$\q0\leq\omg\leq1\qd\mbox{ a.e. }\mbox{ in }\q\bbR_+^2,\quad\nrm{x_2\omg}_1\leq M,\quad\mbox{ and }\qd\nrm{\omg}_1\leq1.$$
It remains to prove that $|E(\omg_n)-E(\omg)|\to0$, which would imply that $E(\omg)=\tld{I}_M$ and $\omg\in\tld{S}_M$. The convergence $E(\omg_n)\to E(\omg)$ is obtained from the following lemmas, which are borrowed from \cite{AC2019}. Lemma \ref{AbeChoi_Prop3.1} is a standard argument for the Steiner symmetrization which is well-known from previous works; see, e.g., \cite[Appendix I]{FB74}, \cite[p.1053]{Tu83}.

\begin{lem}\label{AbeChoi_Prop3.1}\cite[Proposition 3.1]{AC2019}\qd For $\omg\geq0$ satisfying $\omg\in L^1\cap L^2(\bbR^2_+)$ and $x_2\omg\in L^1(\bbR^2_+)$, there exists $\omg^*\geq0$ such that $\omg^*(x_1,x_2)=\omg^*(-x_1,x_2)$, and for each $x>0$, the function $\omg^*(x_1,x_2)$ is non-increasing in $x_1>0.\q$ Moreover, $$\nrm{\omg^*}_q=\nrm{\omg}_q\qd 1\leq q\leq 2,$$
$$\nrm{x_2\omg^*}_1=\nrm{x_2\omg}_1,$$
$$E(\omg^*)\geq E(\omg).$$
\end{lem}

\begin{rmk}  $\omg^*$ in Lemma \ref{AbeChoi_Prop3.1} is defined to satisfy 
$|\set{\omg>\alp}|=|\set{\omg^*>\alp}|$ for all $\alp\geq0$. It implies that $0\leq\omg^*\leq1$. In particular, $\omg^*$ can be chosen to satisfy the Steiner symmetry condition in Definition \ref{def_Steiner sym}. \end{rmk}\q

By Lemma \ref{AbeChoi_Prop3.1}, we may assume that the maximizing sequence $\omg_n$ and the weak limit $\omg$ satisfy the Steiner symmetry condition in Definition \ref{def_Steiner sym}. Then the convergence $E(\omg_n)\to E(\omg)$ is obtained by the Lemma \ref{AbeChoi_Lem3.5} below.

\begin{lem}\label{AbeChoi_Lem3.5}\cite[Lemma 3.5]{AC2019}\qd Let $\{\omg_n\}$ be a sequence such that 
$$\sup_{n\geq1}\Big\{\nrm{\omg_n}_{L^1\cap L^2}+\nrm{x_2\omg_n}_{L^1}\Big\}<\ift,$$
$$\omg_n\weakto\omg\qd\mbox{ in }L^2(\bbR^2_+)\qd\mbox{ as }n\to\ift.$$ Assume that each $\omg_n$ satisfies the Steiner symmetry condition. Then, 
$E(\omg_n)\to E(\omg)\q \mbox{ as }n\to\ift.$
\end{lem}\q

As we obtained that $E(\omg)=\tld{I}_M$, we complete the proof of Lemma \ref{lem_existence of maximizer in larger set}.
\end{proof} 

\subsubsection{Maximizers in smaller set}

\begin{lem}\label{lem_two max.set are same}
    For each $M>0$, we have $$\tld{S}_M=S_M.$$
\end{lem}
\begin{proof}
    Let $M>0$. It suffices to show that $\tld{S}_M\subset K_M$. As the set $\tld{S}_M$ is nonempty by Lemma \ref{lem_existence of maximizer in larger set}, we choose any $\omg\in\tld{S}_M$.\\
    
    \noindent We will first prove that $\nrm{x_2\omg}_1=M$. Suppose that we have
$\nrm{x_2\omg}_1\q<\q M.$ Consider the translation of $\omg$ along $x_2$-axis of amount $\tau>0$, say $\omg_\tau$. Then for small $\tau>0$, we still have 
$\nrm{x_2\omg_\tau}_1\q<\q M$, and by Remark\ref{rmk_prop.of.G}-(iii) we have $$E(\omg_\tau)>E(\omg),$$ which contradicts our assumption that $\omg\in\tld{S}_M$. Therefore we have $\nrm{x_2\omg}_1=M$. \\

\noindent It remains to show that $\omg$ is a characteristic function. In other words, we will show that $$\big|\{0<\omg<1\}\big|\q=\q0,$$  which would imply that, for some measurable set $\Omg\subset \bbR^2_+$, we have $\omg=\bfone_\Omg$ a.e. in $\bbR^2_+$. We suppose that $$\big|\{0<\omg<1\}\big|>0.$$
Then there exists a small constant $\dlt_0>0$ such that $$\big|\{\dlt_0<\omg<1-\dlt_0\}\big|\q>\q0.$$ Observe that 
$\big|\{\dlt_0<\omg<1-\dlt_0\}\big|\q<\q\ift$ due to $\nrm{\omg}_1\leq1$. We can choose two functions $h_1, h_2\in (L^1\cap L^\ift)(\bbR^2_+)$ such that 
$$\supp{(h_i)}\q\subset\q\{\q\dlt_0<\omg<1-\dlt_0\q\}$$ for each $i\in\{1,2\}$ satisfying 
$$\int_{\bbR^2_+} h_1\q d\bfx=1,\qd \int_{\bbR^2_+} x_2h_1\q d\bfx=0,$$ 
$$\int_{\bbR^2_+} h_2\q d\bfx=0,\qd \int_{\bbR^2_+} x_2h_2\q d\bfx=1.$$ 
So far, note that $\dlt_0>0$ and $h_1,h_2\in L^1\cap L^\ift$ are fixed. For an arbitrary $\dlt\in(0,\dlt_0)$, we consider an arbitrary function $h\in (L^1\cap L^\ift)(\bbR^2_+)$ such that $\nrm{x_2h}_1<\ift$ and
\begin{eqnarray*}
    h\geq0&&\mbox{ on }\q\{\q0\leq\omg\leq\dlt\q\},\\
    h\leq0&&\mbox{ on }\q\{\q1-\dlt\leq\omg\leq1\q\}.
\end{eqnarray*}
We denote $$\eta\q:=\q h-\Bigg(\int_{\bbR^2_+} h\q d\bfx\Bigg)\q h_1-\Bigg(\int_{\bbR^2_+}x_2h\q d\bfx\Bigg)\q h_2.$$ Then we get $\eta\in (L^1\cap L^\ift)(\bbR^2_+)$ and $$\int_{\bbR^2_+}\eta\q d\bfx\q=\q\int_{\bbR^2_+} x_2\eta\q d\bfx=0.$$

 \noindent\textit{(Claim 1)} There exists a small constant $\eps_0>0$ such that, for each $\eps\in(0,\eps_0)$, we have $\q\omg+\eps\eta\q\in\q\tld{K}_M$.\\
 
 First, we will show that there exists a small constant $\eps_0>0$ such that, for each $\eps\in(0,\eps_0)$, we have $0\leq\q\omg+\eps\eta\leq1$. 
Observe that, on $\{0\leq\omg\leq\dlt_0\}$, we have $h_1=h_2=0$ and thus $\eta=h\geq0$ due to our assumption on $h$. Therefore $\omg+\eps\eta\geq0$ on $\{0\leq\omg\leq\dlt_0\}$. And by choosing small $\eps_0>0$, we have $$0\q\leq\q\omg+\eps\eta\q\leq\q\dlt_0+\eps_0\q\nrm{\eta}_{L^\ift}\q\leq\q1\qd\mbox{ for each }\q\eps\in(0,\eps_0),$$ on the set $\{0\leq\omg\leq\dlt_0\}$.
Similarly, for each $\eps\in(0,\eps_0)$, we have $0\leq\omg+\eps\eta\leq1$ on the set $\{1-\dlt_0\leq\omg\leq1\}$.
Lastly, by choosing $\eps_0<\dlt_0\q/\q\nrm{\eta}_{L^\ift}$, we obtain 
    $$0\leq\omg+\eps\eta\leq1\qd\mbox{ for each }\eps\in(0,\eps_0)$$ on the set $\{\dlt_0<\omg<1-\dlt_0\}.$ As we have $0\leq\omg+\eps\eta\leq1$ for each $\eps<\eps_0$, we have    
    $$\nrm{\omg+\eps\eta}_1=\int(\omg+\eps\eta)\q d\bfx=\int\omg\q d\bfx\leq1\qd\mbox{ and }\qd\nrm{x_2(\omg+\eps\eta)}_1=\int x_2(\omg+\eps\eta)\q d\bfx=\int x_2\omg\q d\bfx\leq M,$$ and therefore $\omg+\eps\eta\q\in\q\tld{K}_M.\q$ This completes the proof of \textit{(Claim 1)}.\\

\noindent Using the assumption that $\omg\in\tld{S}_M$, the following function $F:[0,\eps_0)\to\bbR$ attains its maximum at $\eps=0$,
$$F(\eps)\q:=\q E(\omg+\eps\eta)\q=\q\f{1}{2}\int_{\bbR^2_+}\int_{\bbR^2_+}G(\bfx,\bfy)\Big(\omg(\bfy)+\eps\eta(\bfy)\Big)\Big(\omg(\bfx)+\eps\eta(\bfx)\Big)\q d\bfy d\bfx.$$ 
Observe that $F$ is a polynomial of degree 2. Clearly, $F'(0)\leq0$. By defining the constants $$W:=\int_{\bbR^2_+}\calG[\omg]\q h_2\q d\bfx \qd\mbox{ and }\qd \gmm:=\int_{\bbR^2_+}\calG[\omg] \q h_1\q d\bfx,$$ we have
\begin{eqnarray*}
    0&\geq& F'(0)=\int_{\bbR^2_+}\int_{\bbR^2_+}G(\bfx,\bfy)\omg(\bfx)\eta(\bfy)\q d\bfy d\bfx\\
    &=&\int_{\bbR^2_+}\calG[\omg](\bfy)\q\eta(\bfy)\q d\bfy\\
    &=&\int_{\bbR^2_+}\calG[\omg](\bfy)\Bigg[h(\bfy)-\Bigg(\int_{\bbR^2_+}h(\bfx)\q d\bfx\Bigg)h_1(\bfy)-
    \Bigg(\int_{\bbR^2_+}x_2h(\bfx)\q d\bfx\Bigg)h_2(\bfy)\Bigg]d\bfy\\
    &=&\int_{\bbR^2_+}\calG[\omg](\bfx)\q h(\bfx)\q d\bfx-
    \underbrace{\Bigg(\int_{\bbR^2_+}\calG[\omg](\bfy)\q h_2(\bfy)\q d\bfy\Bigg)}_{=W}\cdot
    \int_{\bbR^2_+}x_2h(\bfx)\q d\bfx-
    \underbrace{\Bigg(\int_{\bbR^2_+}\calG[\omg](\bfy)\q h_1(\bfy)\q d\bfy\Bigg)}_{=\gmm}\cdot
    \int_{\bbR^2_+}h(\bfx)\q d\bfx\\
    &=&\int_{\bbR^2_+}\Big(\calG[\omg](\bfx)-Wx_2-\gmm\Big)h(\bfx)\q d\bfx
\end{eqnarray*}
Then  we split the integral into three terms:
$$
    0\geq\int_{\bbR^2_+}\Big(\calG[\omg]-Wx_2-\gmm\Big)\cdot h\q 
 d\bfx=\int_{0\leq\omg\leq\dlt}+\int_{\dlt<\omg<1-\dlt}+\int_{1-\dlt\leq\omg\leq1}.
$$

\noindent Recall that, for each $\dlt\in(0,\dlt_0)$, the function $h\in (L^1\cap L^\ift)(\bbR^2_+)$ is arbitrarily chosen just to satisfy 
$$    h\geq0\qd\mbox{ on }\q\{\q0\leq\omg\leq\dlt\q\}\qd\mbox{ and }\qd h\leq0\qd\mbox{ on }\q\{\q1-\dlt\leq\omg\leq1\q\}.$$
Therefore, for any $\dlt\in(0,\dlt_0)$, we obtain that

$$\begin{cases}
    \q\calG[\omg](x)-Wx_2-\gmm\leq0\qd \mbox{a.e.}&\mbox{ on }\q\{\q0\leq\omg\leq\dlt\q\},\\
    \q\calG[\omg](x)-Wx_2-\gmm=0\qd \mbox{a.e.}&\mbox{ on }\q\{\q\dlt<\omg<1-\dlt\q\},\\
    \q\calG[\omg](x)-Wx_2-\gmm\geq0\qd\mbox{a.e.}&\mbox{ on }\q\{\q1-\dlt\leq\omg\leq1\q\}.
\end{cases}$$

\noindent Due to $\{\q0<\omg<1\q\}=\displaystyle\bigcup_{\dlt\q\in\q(0,\dlt_0)}\{\q\dlt<\omg<1-\dlt\q\}$, we get
$$\begin{cases}
    \q\calG[\omg](x)-Wx_2-\gmm\leq0\qd \mbox{a.e.}&\mbox{ on }\q\{\q\omg=0\q\},\\
    \q\calG[\omg](x)-Wx_2-\gmm=0\qd \mbox{a.e.}&\mbox{ on }\q\{\q0<\omg<1\q\},\\
    \q\calG[\omg](x)-Wx_2-\gmm\geq0\qd \mbox{a.e.}&\mbox{ on }\q\{\q\omg=1\q\}.
\end{cases}$$
So we obtain $\{\q0<\omg<1\q\}\q\subset\q\set{\q\calG[\omg]-Wx_2-\gmm=0\q}$. Define $\Psi(\bfx):=\calG[\omg](\bfx)-Wx_2-\gmm$. It remains to prove that 
$$\{\q\Psi=0\q\}\q\subset\q\{\q\omg=0\q\}\qd \mbox{a.e.}\qd \mbox{in }\qd \bbR^2_+,$$ which would imply that $$\{\q0<\omg<1\q\}\q\subset\q\{\q\omg=0\q\}\qd \mbox{a.e.}\qd \mbox{in }\qd \bbR^2_+,\q$$ which contradicts to our assumption that $\big|\{0<\omg<1\}\big|>0$.\\

\noindent\textit{(Claim 2)} For each $\omg\in\tld{S}_M$, we have 
$\{\q\Psi=0\q\}\q\subset\q\{\q\omg=0\q\}\qd \mbox{a.e.}\qd \mbox{in }\q\bbR^2_+.$\\

By Lemma \ref{lem_gw:C1}, we have $\Psi\in H^2_{loc}(\bbR^2_+)$ and $-\lap\Psi=\omg$ in $\bbR^2_+$. Therefore we have 
$$\set{\q\Psi=0\q}\subset \set{\q\nb\Psi=0\q}
\subset\set{\q\lap\Psi=0\q}=\set{\q\omg=0\q}.$$
This completes the proof of \textit{(Claim 2)} and we finish the proof of Lemma \ref{lem_two max.set are same}. \end{proof}

\subsubsection{Every maximizer generates a traveling wave solution}\q

Note that Lemma \ref{lem_existence of maximizer in larger set} and Lemma \ref{lem_two max.set are same} imply that the set $S_M$ is nonempty. Proposition \ref{prop_maximizer is TW} says that each maximizer in $S_M$ generates a traveling wave solution to \eqref{eq_Euler eq.}. 
\begin{prop}\label{prop_maximizer is TW}
For any $M>0$ and for each $\omg\in S_M,$ there exist some constants $W, \gmm\geq0$ that are uniquely determined by $\omg$ satisfying that  
$$
\omg\q=\q\bfone_{\big\{\bfx\in\bbR^2_+\q:\q\calG[\omg](\bfx)-Wx_2-\gmm>0\big\}}\qd\q \mbox{a.e. }\qd { in} \qd \bbR^2_+.
$$
\end{prop} 
\begin{proof} Let $M>0$. We choose any $\omg=\bfone_\Omg\in S_M$ for some measurable set $\Omg\subset\bbR^2_+$. For the proof, we first assume the existence of the constants $W,\gmm\in\bbR$ satisfying that $$\Omg\q=\q\Big\{\calG[\omg]-Wx_2-\gmm>0\Big\}.$$ In \textit{(Claim 1)}, we will prove that $W,\gmm$ are unique. The existence and non-negativity of such constants will be proved later in \textit{(Claim 2)} and \textit{(Claim 3)} each.\\

\noindent \textit{(Claim 1)} If there exist some constants $W,\gmm\in\bbR$ satisfying that $\Omg=\Big\{\calG[\omg]-Wx_2-\gmm>0\Big\},$ then they are uniquely determined by $\omg$.\\

Assume that we have $\Omg=\Big\{\calG[\omg]-Wx_2-\gmm>0\Big\}$ for some constants $W,\gmm\in\bbR$. We choose any two points $\bfy,\bfz\in\rd \Omg\cap\bbR^2_+$ such that $y_2\neq z_2$. 
Observe that  $$\rd \Omg\cap\bbR^2_+\q\subset\q\Big\{\q\calG[\omg]-Wx_2-\gmm=0\q\Big\}$$ by the continuity of $\calG[\omg]$. Therefore 
$$\calG[\omg](\bfy)-Wy_2-\gmm\q=\q\calG[\omg](\bfz)-Wz_2-\gmm\q=\q0$$ and thus
\begin{equation}\label{eq_W,gmm are unique}
   W\q=\q\f{\calG[\omg](\bfy)-\calG[\omg](\bfz)}{y_2-z_2},
\qd\gmm\q=\q-\Big(\f{z_2}{y_2-z_2}\Big)\q\calG[\omg](\bfy)+\Big(\f{y_2}{y_2-z_2}\Big)\q\calG[\omg](\bfz). 
\end{equation}
It implies that $W,\gmm$ are explicitly expressed by $\omg$, so they are unique. This completes the proof of \textit{(Claim 1)}.\\

\noindent \textit{(Claim 2)} There exist some constants $W,\gmm\in\bbR$ such that we have $\Omg=\Big\{\calG[\omg]-Wx_2-\gmm>0\Big\}$.\\

The main part of the proof is to generate perturbation of $\omg=\bfone_\Omg$ in the set $\tld{K}_M$ to discover the specific form of $\Omg$. As we can see in the proof of the uniqueness of $W,\gmm\geq0$, these constants depend on the boundary points of $\Omg$. However, we yet only know that $\Omg$ is a measurable set, which means that we cannot use the definition of boundary points as usual. We naturally want to define boundary-like points for each measurable set which are called \textit{exceptional points}. This notion appeared in several works including \cite{three.lattices-points.problems,exceptional.points.for.Lebesgue.density.theorem,metric.darboux.property,Stability.of.Hill.vortex}.

\begin{defn}\label{def_excep.pt}\cite[Definition 5.4]{Stability.of.Hill.vortex}\q 
    Let $U\subset\bbR^N$ be open. For any (Lebesgue) measurable set $E\subset U$, the density $D_e(E)$ of $E$ is the collection of $\bfx\in U$ such that 
$$\liminf_{r\searrow0}\f{|B_r(\bfx)\cap E|}{|B_r(\bfx)|}=1.
$$
Similarly, define the dispersion $D_i(E)$ of $E$ by the collection of $\bfx\in U$ such that 
$$\limsup_{r\searrow0}\f{|B_r(\bfx)\cap E|}{|B_r(\bfx)|}=0.$$
Set $\calE(E):= U\backslash(D_e(E)\cup D_i(E))$, the set of exceptional points. \end{defn}\q

The following two propositions say that the set of exceptional points has zero measure, but in most cases, it is a nonempty set. The proof of Proposition \ref{prop_density.dispersion} can be found in \cite[Theorem 7.13]{Measure.and.integral.Zygmund.Wheeden}, and the proof of Proposition \ref{prop_density:nonempty} is given in \cite[Appendix B]{Stability.of.Hill.vortex}.

\begin{prop}\label{prop_density.dispersion}\cite[Lemma 5.5]{Stability.of.Hill.vortex}\qd $  E=D_e(E)\qd \mbox{and}\qd  U\backslash E=D_i(E) \qd a.e.$
\end{prop} 

\begin{prop}\label{prop_density:nonempty}\cite[Lemma 5.6]{Stability.of.Hill.vortex}\qd 
    Let $ U\subset\bbR^N$ be a non-empty connected open set. Then for any measurable set $E\subset U$ with $|E|,| U\backslash E|\in(0,\ift]$, we have $\calE(E)\neq\emptyset$.
\end{prop} \q

\noindent By Proposition \ref{prop_density:nonempty}, we can obtain exceptional points $\bfy,\bfz\in\bbR^2_+$ s.t. $y_2>z_2$,  and $\{r_n\}\searrow0$ such that 
\begin{eqnarray*}
    \big|B_{r_n}(\bfy)\cap \Omg\big|,\q\q\big| B_{r_n}(\bfy)\cap \Omg^c\big|&\in&(0,\ift),\\
    \big|B_{r_n}(\bfz)\cap \Omg\big|,\q\q\big|B_{r_n}(\bfz)\cap  \Omg^c\big|&\in&(0,\ift),\qd\mbox{ for each } n\geq1.
\end{eqnarray*} 
To show this, we apply Proposition \ref{prop_density:nonempty}. We first choose $l>0$ such that both $\Omg\cap\{x_2>l\}$ and $\Omg\cap\{x_2<l\}$ have finite and positive measures. In order to obtain such $\bfy\in\bbR^2_+$, put $$ U:=\{x_2>l\},\qd E:=\Omg\cap\{x_2>l\}$$ and choose $\bfy\in\calE(E)$ so that $y_2>l$. 
Then for any decreasing sequence of positive numbers $\{r_n\}\searrow0$, we can simply check that 
$$\big|B_{r_n}(\bfy)\cap  \Omg\big|,\q\q\big|B_{r_n}(\bfy)\cap \Omg^c\big|\q\in\q(0,\ift),\qd\mbox{ for each } n\geq1.$$ 
Similarly, we can choose $\bfz\in\bbR^2_+$ such that $0<z_2<l$ and do the same process.\\

\noindent Define sequences of sets of positive measures near $\bfy,\bfz\in\bbR^2_+$ as 
\begin{eqnarray*}
Y_n^+:=B_{r_n}(\bfy)\cap  \Omg,&&Y_n^-:=B_{r_n}(\bfy)\cap \Omg^c,\\
Z_n^+:=B_{r_n}(\bfz)\cap  \Omg,&&Z_n^-:=B_{r_n}(\bfz)\cap  \Omg^c.
\end{eqnarray*}
and sequences of compactly supported and bounded functions as $$f_n^\pm\q:=\q\f{1}{|Y_n^\pm|}\bfone_{Y_n^\pm},\qd g_n^\pm\q:=\q\f{1}{|Z_n^\pm|}\bfone_{Z_n^\pm}.$$ 
These functions will play as approximations of Dirac delta at $\bfy$ and $\bfz$ each, as $n\to\ift$. Recalling the proof of \textit{(Claim 1)} (see the identity \eqref{eq_W,gmm are unique}), define the constants 
\begin{eqnarray*}
W&:=&\Bigg(\f{1}{y_2-z_2}\Bigg)\q\calG[\omg](\bfy)\q-\q\Bigg(\f{1}{y_2-z_2}\Bigg)\q\calG[\omg](\bfz)\q\in\bbR, \qquad \gmm := -\q\Bigg(\f{z_2}{y_2-z_2}\Bigg)\q\calG[\omg](\bfy)\q+\q\Bigg(\f{y_2}{y_2-z_2}\Bigg)\q\calG[\omg](\bfz)\q\in\bbR.
\end{eqnarray*}
As $n\to\ift$, we have 
$$\int x_2f_n^\pm(\bfx)\q d\bfx\to y_2\qd \mbox{and} \qd \int x_2g_n^\pm(\bfx)\q d\bfx\to z_2.$$
We denote $\q\displaystyle y_n^\pm:=\int x_2f_n^\pm(\bfx)\q d\bfx\q\in\bbR$ and $\q\displaystyle z_n^\pm:=\int x_2g_n^\pm(\bfx)\q d\bfx\q\in\bbR$. Moreover, by the continuity of $\calG[\omg]$, 
$$\int \calG[\omg]\q f_n^\pm\q d\bfx \q\to\q \calG[\omg](\bfy)\qd \mbox{and} \qd \int \calG[\omg]\q g_n^\pm\q d\bfx \q\to\q \calG[\omg](\bfz)$$
as $n\to\ift$. In order to construct sequences converging to $W$ and $\gmm$, define the sequences of coefficients as 
$$a_n^\pm:=\f{y_n^\mp}{y_n^\mp-z_n^\pm},\qd 
b_n^\pm:=\f{z_n^\pm}{y_n^\mp-z_n^\pm},\qd
c_n^\pm:=\f{1}{y_n^\pm-z_n^\mp}.$$
Note that they are all positive for large $n\geq1$. Assume that $a_n^\pm, b_n^\pm, c_n^\pm>0$, for all $n\geq1$. Using these sequences, we construct functions as 
$$h_{n,1}^\pm:=a_n^\pm g_n^\pm-b_n^\pm f_n^\mp,\qd h_{n,2}^\pm:=c_n^\pm(f_n^\pm-g_n^\mp).$$
These functions will be used to construct the perturbation for $\omg=\bfone_\Omg$ in the set $\tld{K}_M$. Observe that
$$\int h_{n,1}^\pm\q d\bfx=1,\qd \int x_2 h_{n,1}^\pm\q d\bfx=0,$$
$$\int h_{n,2}^\pm\q d\bfx=0,\qd \int x_2 h_{n,2}^\pm\q d\bfx=1.$$
Observe that the signs $+$ and $-$ in each coefficient and function are chosen to satisfy 
\begin{eqnarray*}
h_{n,1}^+,\q h_{n,2}^+\q\geq0&\mbox{on}&\Omg,\\
h_{n,1}^+,\q h_{n,2}^+\q\leq0&\mbox{on}&\Omg^c,
\end{eqnarray*}
and 
\begin{eqnarray*}
h_{n,1}^-,\q h_{n,2}^-\q\leq0&\mbox{on}&\Omg,\\
h_{n,1}^-,\q h_{n,2}^-\q\geq0&\mbox{on}&\Omg^c.
\end{eqnarray*}

\noindent Now we are ready to prove that $\q \Omg=\Big\{\calG[\omg]-Wx_2-\gmm>0\Big\}\q$ by using perturbation of the function $\omg=\bfone_\Omg$ in the set $\tld{K}_M$. Consider any function $h\in (L^1\cap L^\ift)(\bbR^2_+)$ such that $\nrm{x_2h}_1<\ift$ and
$$h\geq0\qd \mbox{on} \q \Omg^c,\qd h\leq0\qd \mbox{on}\q \Omg.$$
For any fixed function $h$, we choose the sign superscripts for $h_{n,1}^\pm$ and $h_{n,2}^\pm$ as follows. We define $h_{n,1}$ and $h_{n,2}$ to satisfy
$$h_{n,1}:=\left\{\begin{matrix}
    \q h_{n,1}^+\qd \mbox{if} \qd \int h\geq0,\\
    \q h_{n,1}^-\qd \mbox{if} \qd \int h<0,
\end{matrix}\right. $$ and 
$$h_{n,2}:=\left\{\begin{matrix}
    \q h_{n,2}^+\qd \mbox{if} \qd \int x_2h\geq0,\\
    \q h_{n,2}^-\qd \mbox{if} \qd \int x_2h<0.
\end{matrix}\right. $$ Moreover we define 
$$W_n:=\int\calG[\omg]\q h_{n,2}\q d\bfx,\qd \gmm_n:=\int\calG[\omg]\q h_{n,1} \q d\bfx.$$ Observe that $W_n\to W$ and $\gmm_n\to\gmm$ as $n\to\ift$.
We set 
$$\eta_n:=h-\Bigg(\int h\q d\bfx\Bigg)\q h_{n,1}-\Bigg(\int x_2 h\q d\bfx\Bigg)\q h_{n,2}.$$
Naturally $\eta_n\leq0$ on $\Omg$ and $\eta_n\geq0$ on $\Omg^c$. Moreover, observe that $\int\eta_n=\int x_2\eta_n=0$. As $\eta_n\in L^\ift$, by the similar arguments in the proof of \textit{(Claim 1)} of \ref{lem_two max.set are same}, we may show that there exists a small $\eps_0>0$ such that 
$$\omg+\eps\eta_n\q\subset\q\tld{K}_M\qd 
\mbox{ for each } \eps\in(0,\eps_0),$$ 
and using the fact that $E(\cdot)$ is maximized by $\omg$ in the admissible set $\tld{K}_M$, we obtain
$$0\geq\f{d}{d\eps}E(\omg+\eps\eta_n)\Big|_{\eps=0}=\int \calG[\omg]\q\eta_n\q d\bfx=\int\Big(\calG[\omg]-W_n x_2-\gmm_n \Big)\q h\q d\bfx,$$ for each $n\geq1$. We can take the limit $n\to\ift$ to obtain  
$$0\q\geq\q\int\Big(\calG[\omg]-W x_2-\gmm\Big)\q h\q d\bfx.$$
Since $h$ is arbitrary, we have 
\begin{eqnarray*}
    \calG[\omg]-Wx_2-\gmm\geq0\qd \mbox{a.e. on}\qd \Omg,\\
    \calG[\omg]-Wx_2-\gmm\leq0\qd \mbox{a.e. on}\qd \Omg^c.
\end{eqnarray*}
Put $\Psi:=\calG[\omg]-Wx_2-\gmm\q$ and observe $$\{\q\Psi<0\q\}\q\subset\q \Omg^c \q\subset\q \{\q\Psi\leq0\q\}.$$ Since $-\lap\Psi=\omg$ a.e. in $\bbR^2_+$, we have $$\q\{\q\Psi=0\q\}\q\subset\q\{\q\omg=0\q\}\q=\q \Omg^c\q \mbox{ a.e. }$$ and thus $\{\q\Psi\leq0\q\}\q=\q \Omg^c.$ In conclusion, we have 
$$\Omg\q=\q\big\{\q\Psi>0\q\big\}\q=\q\big\{\q\calG[\omg]-W x_2-\gmm>0\q\big\}\qd \mbox{ a.e. }$$
This completes the proof of \textit{(Claim 2)}.\\

\noindent\textit{(Claim 3)} The constants $W,\gmm\in\bbR$ are non-negative. \\

If $\gmm<0$, observe that 
$$\Bigl\{\q\bfx\in\bbR^2_+\q:\q |W|x_2<-\gmm\q\Bigl\}\q\subset\q
\Bigl\{\q\bfx\in\bbR^2_+\q:\q Wx_2<-\gmm\q\Bigl\}\q=\q
\Bigl\{\q\calG[\omg]-Wx_2-\gmm>0\q\Bigl\}\q=\q \Omg.$$ 
It implies that $|\Omg|=\ift$, which contradicts to $|\Omg|\leq1$. Therefore we have $\gmm\geq0$. If $W<0$, we have 
$$\Bigl\{\q\bfx\in\bbR^2_+\q:\q -Wx_2>\gmm \q\Bigl\}\q\subset \q 
\Bigl\{\q\calG[\omg]-Wx_2-\gmm>0\q\Bigl\}\q=\q \Omg,$$
which also contradicts to $|\Omg|\leq1$. It follows that $W\geq0$. This completes the proof of \textit{(Claim 3)} and we finish the proof of Proposition \ref{prop_maximizer is TW}. 
\end{proof}\q

Lemma \ref{lem_lack of mass to gmm=0} below says that, if any maximizer in $K_M$ does not attain the full mass, then the corresponding flux constant $\gmm$ vanishes.
\begin{lem}\label{lem_lack of mass to gmm=0}
For any $M>0$ and for each $\omg\in S_M,$ if we have $\nrm{\omg}_1<1$, then we get $$\gmm=0$$ where $\gmm=\gmm(\omg)\geq0$ is given in Proposition \ref{prop_maximizer is TW}.
\end{lem} 
\begin{proof} Let $M>0$. We choose any $\omg=\bfone_\Omg\in S_M$ where we have $$\Omg\q=\q\Big\{\calG[\omg]-Wx_2-\gmm>0\Big\}$$ where the constants $W,\gmm\geq0$ are given in Proposition \ref{prop_maximizer is TW}. We assume that $\nrm{\omg}_1<1$. We now follow the same process in the proof of \textit{(Claim 2)} in Proposition \ref{prop_maximizer is TW} in a way that, instead of using $\eta_n$, we use the following function 
$$\hat{\eta}_n:=h-\Big(\int x_2h\q d\bfx\Big)h_{n,2}$$ with the same definitions for $h, h_{n,2}$ in the proof of \textit{(Claim 2)} in Proposition \ref{prop_maximizer is TW}. As $\nrm{\omg}_1<1$, we have $$\nrm{\omg+\eps\hat{\eta}_n}_1\leq\nrm{\omg}_1+\eps\nrm{\hat{\eta}_n}_1\leq1\qd$$ for small $\eps>0$. Moreover, as $\hat{\eta}_n\in L^\ift$, for small $\eps>0$ we have $0\leq\omg+\eps\hat{\eta}_n\leq1$, and due to $\int x_2\hat{\eta}_n \q d\bfx=0,$ we get $$ \nrm{x_2(\omg+\eps\hat{\eta}_n)}_1
=\int_{\bbR^2_+}x_2\omg\q d\bfx+\eps\int_{\bbR^2_+}x_2\hat{\eta}_n\q d\bfx=M.$$  This implies that there exists a small $\eps_0>0$ such that 
$$\omg+\eps\hat{\eta}_n\q \in \q \tld{K}_M$$ for each $\eps\in (0,\eps_0)$. Then we have 
$$0\q\geq\q \f{d}{d\eps} E(\omg+\eps\hat{\eta}_n)\q=\int \Big(\calG[\omg]-W_n x_2\Big)\q h\q d\bfx$$ with the same definition for $W_n$ with the proof of \textit{(Claim 2)} in Proposition \ref{prop_maximizer is TW}. It follows that  
$$\omg\q=\q \bfone_{\{\q\calG[\omg]-Wx_2>0\q\}}\qd \mbox{ a.e. \q in }\q\bbR^2_+.$$ By the uniqueness of the constant $\gmm\geq0$ given in Proposition \ref{prop_maximizer is TW}, we complete the proof of Lemma \ref{lem_lack of mass to gmm=0}. 
\end{proof}\q

\subsubsection{Positivity of traveling speed }\q

\begin{rmk}\label{rmk_W formula} For any $\omg\in (L^1\cap L^\ift)(\bbR^2_+)$, whose odd-extension $\tld{\omg}$ is a traveling wave solution of the equation \eqref{eq_Euler eq.}, in the form $$\omg(\bfx)=\omg_0(\bfx-(W,0)t)$$ for some $\omg_0\in (L^1\cap L^\ift)(\bbR^2_+)$ having bounded support with some constant $W\in\bbR$, its traveling speed $W$ is determined by $\omg_0$ by the identity in \cite[p.1062]{Tu83}:
\begin{equation}\label{eq_W formula}
    W\q=\q\nrm{\omg_0}_1^{-1}\cdot\int_{\bbR^2_+}\calG[\omg_0]_{x_2}\cdot\omg_0\q d\bfx\q=\q
\nrm{\omg_0}_1^{-1}\cdot\int_{\bbR^2_+}\int_{\bbR^2_+}\f{1}{2\pi}\f{x_2+y_2}{|\bfx-\bfy^*|^2}\q\omg_0(\bfy)\q\omg_0(\bfx)\q d\bfy d\bfx
\end{equation} where $\bfy^*=(y_1,-y_2)$. Indeed, for the first equality, observe that
$$W\int_{\bbR^2_+}\omg_0\q d\bfx=\f{d}{dt}\int_{\bbR^2_+}x_1\omg\q d\bfx=
\int_{\bbR^2_+}x_1\omg_t\q d\bfx=-\int_{\bbR^2_+}x_1(\bfu\cdot\nb\omg)\q d\bfx=\int_{\bbR^2_+}u_1\cdot\omg\q d\bfx$$ where we used integration by parts. The second equality is obtained by considering obvious cancellation.
\end{rmk}\q
\begin{lem}\label{lem_W formula}
 For any $M>0$ and for each $\omg\in S_M$, we have $$W\q>\q0$$ where $\q W=W(\omg)$ is given in Proposition \ref{prop_maximizer is TW}. 
\end{lem}
\begin{proof}
    Let $M>0$ and choose any $\omg\in S_M$. We suppose $W=0$. Then $\gmm>0$ due to the assumption $\nrm{\omg}_{1}\leq1$. By Lemma \ref{lem_gw:decay}, it is clear that the set $\set{\q\calG[\omg]-\gmm>0\q}$ is bounded. Then we can derive $W>0$ by the identity \eqref{eq_W formula} in Remark \ref{rmk_W formula}, which contradicts our assumption that $W=0$. 
\end{proof}

\subsubsection{Bounded support of maximizers}\q

In the following lemma, we show that each maximizer $\omg$ has a bounded support. 
\begin{lem}\label{lem_w:cpt_supp}
    For any $M>0$ and for each $\omg$ in $S_M$, the set $\big\{\calG[\omg]-Wx_2-\gmm>0\big\}$ is bounded, for the corresponding constant $W,\gmm$ given in Proposition \ref{prop_maximizer is TW}. Moreover, there exists a constant $C>0$ such that, for any $M>0$ and for each $\omg\in S_M$, we have  
\begin{equation}\label{eq_est of gw/x_2}
    \Bigg|\f{\calG[\omg](\bfx)}{x_2}\Bigg|\q\leq\q C\q\Big( \q\eps^{-1}\nrm{\omg}_1
+\eps^{1/3}\nrm{\omg}_{L^1\big([x_1-\eps, x_1+\eps]\times(0,\ift)\big)}^{1/3}\q\Big),
\end{equation}
for any $\eps>0$ and $\bfx=(x_1,x_2)\in\bbR^2_+$.
\end{lem}
\begin{rmk}\label{rmk_Gw/x_2 decays}
    We can deduce from the estimate \eqref{eq_est of gw/x_2} in Lemma \ref{lem_w:cpt_supp} below that $\big|\calG[\omg](\bfx)/x_2\big|\to0$ as $|x_1|\to\ift$ where the convergence is uniform in $x_2$. This estimate will be used later in the proofs of Theorem \ref{thm_boundry is continuous} and Lemma \ref{lem_size of w}. 
\end{rmk}

\begin{proof}[Proof of Lemma \ref{lem_w:cpt_supp}]\q Let $M>0$. As $S_M$ is nonempty, we choose any  $\omg\in S_M$. By Proposition \ref{prop_maximizer is TW}, we have
    $$\omg=\bfone_{\big\{\calG[\omg]-Wx_2-\gmm>0\big\}}
    \qd\mbox{ in }\bbR^2_+$$ for some unique constants $W=W(\omg)>0$ and $\gmm=\gmm(\omg)\geq0$. By \eqref{eq_gw:bdd} in Lemma \ref{lem_gw:estiate}, we have $\calG[\omg]\in L^\ift(\bbR^2_+)$. Therefore we obtain 
    $$\Big\{\q\calG[\omg]-Wx_2-\gmm>0\q\Big\}\q\subset\q 
    \Big\{\q\calG[\omg]-Wx_2>0\q\Big\}\q\subset\q
    \Big\{\q x_2<\f{\nrm{\calG[\omg]}_{L^\ift}}{W}\q\Big\}.$$
    It remains to prove that there exists a constant $L>0$ such that 
    $$\Big\{\q\calG[\omg]-Wx_2-\gmm>0\q\Big\}\q\subset\q
    \Big\{\q|x_1|\leq L\q\Big\}.$$ 
    Observe that 
    $$\Big\{\q\calG[\omg]-Wx_2-\gmm>0\q\Big\}\q\subset\q
    \Big\{\q\calG[\omg]/x_2>W\q\Big\}.$$ By the mean value theorem, for each $\bfx\in\bbR^2_+$, there exists a constant $c=c(\bfx)\in(0,x_2)$ such that 
    $$\f{\calG[\omg](\bfx)}{x_2}=
    \f{\calG[\omg](\bfx)-\calG[\omg](x_1,0)}{x_2}=\calG[\omg]_{x_2}(x_1,c).$$ Here we have 
    $$\calG[\omg]_{x_2}(x_1,c)\q=\q\f{1}{2\pi}\int_{\bbR^2}\f{(y_2-c)}{(y_1-x_1)^2+(y_2-c)^2}\q\tld{\omg}(\bfy)\q d\bfy$$ where $\tld{\omg}$ is the odd-extension of $\omg$ to $\bbR^2$. Note that, for each $\eps>0$, we have 
    $$\Big|\calG[\omg]_{x_2}(x_1,c)\Big|\q\lesssim 
    \left[\int_{B^c}+\int_B\right] \q  \f{1}{|(y_1,y_2)-(x_1,c)|}\q|\tld{\omg}(\bfy)|\q d\bfy $$ where $B:=B_\eps(x_1,c)$. The first integral satisfies 
    $$\int_{B^c}\f{|\tld{\omg}(\bfy)|}{|(y_1,y_2)-(x_1,c)|}\q d\bfy\q\leq\q\f{2}{\eps}\q\nrm{\omg}_1.$$ For the second integral, by H\"older's inequality, we have 
    $$\int_B\f{|\tld{\omg}(\bfy)|}{|(y_1,y_2)-(x_1,c)|}\q d\bfy\q\lesssim\q\Bigg(\int_0^\eps r^{-1/2}dr\Bigg)
^{2/3}\nrm{\omg}_{L^3(B)}\q=\q\Bigg(\int_0^\eps r^{-1/2}dr\Bigg)
^{2/3}\nrm{\omg}_{L^1(B)}^{1/3}\q\lesssim\q \eps^{1/3}\nrm{\omg}_{L^1\big([x_1-\eps, x_1+\eps]\q\times\q(0,\ift)\big)}^{1/3}. $$
In sum, we obtain \eqref{eq_est of gw/x_2} as we have 
$$\Bigg|\f{\calG[\omg](\bfx)}{x_2}\Bigg|\q\leq\q C_0\q\Big( \q\eps^{-1}\nrm{\omg}_1
+\eps^{1/3}\nrm{\omg}_{L^1\big([x_1-\eps, x_1+\eps]\times(0,\ift)\big)}^{1/3}\q\Big)$$
for some absolute constant $C_0>0$. Moreover, note that the right-hand side does not depend on $x_2$. Here we put $$\eps=\f{3C_0\nrm{\omg}_1}{W}$$ to have 
$$\Big|\f{\calG[\omg](\bfx)}{x_2}\Big|\q\leq\q \f{W}{3}+C_0\eps^{1/3}\nrm{\omg}_{L^1([x_1-\eps, x_1+\eps]\times(0,\ift))}^{1/3}.$$ Due to $\omg\in L^1(\bbR^2_+)$, the last term tends to 0 as $|x_1|\to\ift$, uniformly in $x_2$. So there exists a constant $L=L(\omg)>0$ such that, for any $\bfx\in\bbR^2_+$,  
$$|x_1|>L\q\Rightarrow\q \Bigg|\f{\calG[\omg](\bfx)}{x_2}\Bigg|\q\leq\q \f{W}{3}+\f{W}{3}<W,
$$ and thus 
$$\Big\{\q\calG[\omg]-Wx_2-\gmm>0\q\Big\}\q\subset\q
    \Big\{\q\calG[\omg]/x_2>W\q\Big\}\q\subset\q\Big\{\q|x_1|\leq L\q \Big\}.$$
    This completes the proof.
\end{proof}\q

\begin{rmk}
    In the proof above, we essentially showed that, for any $0\leq f\in (L^1\cap L^\ift)(\bbR^2_+)$ with $\nrm{x_2f}_1<\ift$ and for each $W>0$, the set $\set{\q\bfx\in\bbR^2_+\q:\q\calG[f](\bfx)-Wx_2>0\q}$ is bounded in $\bbR^2_+$.
\end{rmk}\q


\begin{proof}[Proof of Theorem \ref{thm_maximizer}]\q Let $M>0$. By Lemma \ref{lem_existence of maximizer in larger set} and Lemma \ref{lem_two max.set are same}, the set $S_M$ is nonempty. Then for any $\omg\in S_M$, by Proposition \ref{prop_maximizer is TW}, we have $$\omg\q=\q1_{\big\{\calG[\omg]-Wx_2-\gmm>0\big\}}\qd\q \mbox{a.e. }\qd { in} \qd \bbR^2_+$$ for some constants $W,\q\gmm\geq0$ which are uniquely determined by $\omg$, and by Lemma \ref{lem_W formula}, we have $W>0$. Since $E(\omg)$ is the absolute maximum of $E(\cdot)$, we may assume that $\omg$ satisfies the Steiner symmetry condition in Definition \ref{def_Steiner sym} up to translation. Lemma \ref{lem_w:cpt_supp} implies that $\omg$ has a bounded support, up to measure zero set. Lastly, by Lemma \ref{lem_lack of mass to gmm=0}, if we have $\nrm{\omg}_1<1$, then we get $\gmm=0$.
\end{proof}\q
	
\subsection{Estimate on maximal energy}\q

  Using the specific form of the maximizers given in Theorem \ref{thm_maximizer}, we obtain the following crucial identity concerning the maximal energy, traveling speed, and impulse. 
\begin{lem}\label{lem_WM=energy} 
    For any $M>0$, we have 
    \begin{equation}\label{eq_energy and W}
        I_M=\f{3}{4}WM+\f{1}{2}\gmm\nrm{\omg}_1,
    \end{equation}
    for each $\omg\in S_M$ with the constants $W=W(\omg)>0,\q \gmm=\gmm(\omg)\geq0$ in Theorem \ref{thm_maximizer}. 
\end{lem}

\begin{rmk}
    The identity in the above lemma can be essentially found in Lemma 9 of \cite{Burton.1988}. It is a so-called Poho\v{z}aev-type identity (see \cite{Pohozaev.1965}). We present the proof below for completeness. For vortex rings, we refer to Lemma 3.2 of \cite{FT81} (also see Proposition 5.13 of \cite{Stability.of.Hill.vortex}).
\end{rmk}

\begin{proof}[Proof of Lemma \ref{lem_WM=energy}]\q Let $M>0$. As $S_M$ is nonempty, we take any $\omg\in S_M$. Put $\omg=\bfone_{\big\{ \calG[\omg]-Wx_2-\gmm>0\big\}}$ for the constants $W=W(\omg)>0$ and $\gmm=\gmm(\omg)\geq0$ in Theorem \ref{thm_maximizer}.
Let $\tld{\omg}$ be the odd extension of $\omg$ to $\bbR^2$ and denote $\psi:=\tld{\calG[\omg]}$, the odd extension of $\calG[\omg]$. We first make the following claim.\\

\noindent\textit{(Claim)} We have  $\int_{\bbR^2}\tld{\omg}\q(\bfx\cdot\nb\psi)\q d\bfx\q=\q 0.$\\

We first observe that, for any scalar function $f\in C^2(\bbR^2)$, we have 
$$\lap f\cdot(\bfx\cdot\nb f) \q=\q
\nb\cdot\Big(\q\f{1}{2}x_1\big(f_{x_1}^2-f_{x_2}^2\big)+
x_2f_{x_1}f_{x_2},\q\q 
\f{1}{2}x_2\big(-f_{x_1}^2+f_{x_2}^2\big)+
x_1f_{x_1}f_{x_2}\q\Big).$$ Knowing that $\tld{\omg}$ is compactly supported by Lemma \ref{lem_w:cpt_supp} and $\psi\in H^2_{loc}(\bbR^2)$ with the relation $-\lap\psi=\tld{\omg}$ by Lemma \ref{lem_gw:C1}, the approximation to smooth function in the bounded domain in $\bbR^2$ gives  $$\int_{\bbR^2}\tld{\omg}\q(\bfx\cdot\nb\psi)\q d\bfx\q=\q 0.$$  This completes the proof of \textit{(Claim)}.\\

\noindent Since the function $\tld{\omg}\q(\bfx\cdot\nb\psi)$ is an even function, we get 
$$\int_{\bbR^2_+}\omg\q(\bfx\cdot\nb\calG[\omg])\q d\bfx=0.$$ Here, by the relation $\omg=\bfone_{\set{\calG[\omg]-Wx_2-\gmm>0}},$ observe that $$\omg\cdot(\calG[\omg]-Wx_2-\gmm)\q=\q(\calG[\omg]-Wx_2-\gmm)_+$$ where $s_+:=\max\set{s,0},\q s\in\bbR.\q$ We define $K:=(\calG[\omg]-Wx_2-\gmm)_+$, and using the chain rule (see \cite[Theorem 7.8]{G-B.elliptic.pde}), we observe that $K\in H^1_{loc}(\bbR^2_+)$  with the relation $$\nb K=\omg\cdot(\nb\calG[\omg]-(0,W)).$$ Note that we have 
    $$0\q=\q \int_{\bbR^2_+}\omg\q(\bfx\cdot\nb\calG[\omg])\q d\bfx\q=\q
    \int_{\bbR^2_+} \bfx\cdot\nb K\q d\bfx+
    \int_{\bbR^2_+}Wx_2\omg\q d\bfx$$ where 
    $$ \int_{\bbR^2_+} \bfx\cdot\nb K\q d\bfx=-2 \int_{\bbR^2_+}K\q d\bfx$$ using approximation to smooth function and integration by parts. As we know that  
    $$\int_{\bbR^2_+} K\q d\bfx=\int_{\bbR^2_+}\omg\cdot \Big(\calG[\omg]-Wx_2-\gmm\Big)\q d\bfx=2I_M-WM-\gmm\nrm{\omg}_1 \qd\mbox{ and }\qd \int_{\bbR^2_+}Wx_2\omg\q d\bfx=WM,$$ we get \eqref{eq_energy and W}.
\end{proof}\
By the scaling given in Remark \ref{rmk_prop.of.G}-(v) and Remark \ref{rmk_prob.congruent}, we have the following corollary.
\begin{cor}\label{cor_WM=energy}
    For any $\mu,\nu,\lmb>0$, we have 
    \begin{equation}\label{eq_(general)energy and W}
        I_{\mu,\nu,\lmb}=\f{3}{4}W\mu+\f{1}{2}\gmm\nrm{\omg}_1,
    \end{equation}
    for each $\omg\in S_{\mu,\nu,\lmb}$ with the constants $W=W(\omg)>0,\q \gmm=\gmm(\omg)\geq0$ in Corollary \ref{cor_maximizer}. 
\end{cor}

\begin{rmk}\label{rmk_dismiss measure 0 }
    By Theorem \ref{thm_maximizer}, for any $M>0$,  the set of maximizers $S_M$ is nonempty, and that each $\omg\in S_M$ satisfies 
    $$\omg\q=\q \bfone_{\big\{\q\calG[\omg]-Wx_2-\gmm>0\q\big\}}\qd { a.e. }\qd \mbox{ in }\q\bbR^2_+ $$ for some constants $W=W(\omg)>0,\q\gmm=\gmm(\omg)\geq0$ that are uniquely determined by $\omg$. Hereafter, if there is no confusion, we will simply assume that 
    $$\omg\q=\q \bfone_{\big\{\q\calG[\omg]-Wx_2-\gmm>0\q\big\}}\qd \mbox{ in }\q\bbR^2_+. $$ That is, we dismiss measure zero sets and assume that each maximizer $\omg$ is a characteristic function on the open bounded set $\set{\calG[\omg]-Wx_2-\gmm>0}$. 
\end{rmk}

\section{Maximizers under small impulse with unit mass bound}\label{sec:small-impulse}\q

In this section, we will prove that, if the constraint on the impulse of the variational problem is sufficiently small, then the flux constant $\gmm$ of each maximizer vanishes. Accordingly, each traveling vortex patch, generated from any maximizer, touches the horizontal line from which a continuous boundary streamline emerges. The main results of this section are given in Theorem \ref{thm_touching of dipole} and Theorem \ref{thm_boundry is continuous} below. In addition, as we obtain $\gmm=0$ due to a small impulse $M$, we have the relation 
$$I_M=\f{3}{4}WM$$ which is deduced from the identity \eqref{eq_energy and W} in Lemma \ref{lem_WM=energy}, and the relation says that $W=W(\omg)$ is completely determined by the choice of $M$. It leads to various estimates for quantities concluding mass and traveling speed, using the $M$ only. These results are given in Lemma \ref{lem_ests for small impulse} and Lemma \ref{lem_size of w}, which will be used in Section \ref{sec_ no mass bound}. 

\begin{thm}\label{thm_touching of dipole}\q  There exists an absolute constant $M_1>0$ such that,  for any $M\in(0,M_1)$ and for each maximizer $\omg=\bfone_\Omg\in S_M$ satisfying the Steiner symmetry condition in Definition \ref{def_Steiner sym}, there exists $\eps>0$ such that 
$$B_\eps(0,0)\q\subset\q \overline{\Omg\cup\Omg_-}.$$
\end{thm}

  Theorem \ref{thm_boundry is continuous} below says that the boundary streamline of the touching vortex patch continuously runs to its end-point on $x_1$-axis. 

\begin{thm}\label{thm_boundry is continuous}\q    Let $M_1>0$ be the constant in Theorem \ref{thm_touching of dipole}. For any $M\in(0,M_1)$ and for each $\omg\in S_M$ satisfying the Steiner symmetry condition in Definition \ref{def_Steiner sym}, with the constant $W=W(\omg)>0$ in Theorem \ref{thm_maximizer}, the following holds: \\

    \noindent(i) Define $A:=\set{\q y_2>0\q:\q \omg(0,y_2)=1\q}$. Then the set $A$ is a nonempty, bounded, and open subset of $(0,\ift)$, and there exists $\eps>0$ such that $(0,\eps)\subset A$. In addition, we can define a function $l: (0,\ift)\to [0,\ift)$ such that $$\Big\{\q y_1\in\bbR\q:\q \omg(y_1,y_2)=1\q\Big\}\q=\q \Big(-l(y_2),\q l(y_2)\Big)\qd \mbox{for }\q y_2\in A,\qd \mbox{ and }\qd l(y_2)=0\qd \mbox{ for }\q y_2\in (0,\ift)\backslash A.$$ 
    
    \noindent (ii) The function $l$ is continuous in $(0,\ift)$ and lies in $C^{1,r}$ locally in $A$, for each $r\in(0,1)$. Furthermore, for each $y_2\in A$, we have $$\calG[\omg]\big(\q l(y_2),\q y_2\big)=Wy_2.$$

    \noindent(iii) There exists a unique constant $a\in(0,\ift)$ such that $$\tld{\calG[\omg]}_{x_2}(a,0)=W.$$ Moreover the limit  $\lim_{y_2\to 0^+}l(y_2)$ exists and is equal to the number $a\in(0,\ift)$. 
\end{thm}

\subsection{Vanishing of flux constant}\q

In Lemma \ref{lem_small impulse to gmm 0} we will prove $\gmm=0$ given that the impulse of the maximizer is small enough. Lemma \ref{lem_lower bound of energy} below will be used in the proof of Lemma \ref{lem_small impulse to gmm 0}. 

\begin{lem}\label{lem_lower bound of energy}
    For any $M\in(0,1)$, we have 
    $$I_M\q\geq\q M^{4/3}\cdot I_1.$$
\end{lem}
\begin{proof}\q By Theorem \ref{thm_maximizer}, the set $S_1$ of maximizers is nonempty. Choose any $f\in S_1$ so that $$\nrm{f}_1\leq1,\qd\nrm{x_2f}_1=1.$$

After the scaling $g(\bfx)=f(M^{-1/3}\bfx)$, we have 
        $$\nrm{g}_1=M^{2/3}\nrm{f}_1\leq M^{2/3}<1,
        \qd\nrm{x_2g}_1=M\nrm{x_2f}_1=M,\qd\mbox{and }\q E(g)=M^{4/3}E(f)=M^{4/3}I_1,$$ so we have $g\in K_M$ and $$I_M\geq E(g)=M^{4/3}I_1.$$
\end{proof}

\begin{lem}\label{lem_small impulse to gmm 0}
There exists a constant $M_1>0$ such that, for any $M\in(0,M_1)$ and for each $\omg\in S_M$, we have  $$\gmm=0,$$ where $\gmm=\gmm(\omg)\geq0$ is given in Theorem \ref{thm_maximizer}.
\end{lem}
\begin{proof}\q Let $M>0$. As $S_M$ is a nonempty set, we take any $\omg\in S_M$ with the constants $W=W(\omg)>0,\q \gmm=\gmm(\omg)\geq0$ given in Theorem \ref{thm_maximizer}. Then for any $L>0$, we split the integral as 
    $$\nrm{\omg}_1\q=\q \int_{0<x_2\leq L}\omg(\bfx)\q d\bfx+\int_{x_2>L}\omg(\bfx)\q d\bfx$$ and observe that 
    $$M\q=\q\nrm{x_2\omg}_1\q\geq\q \int_{x_2>L}x_2\omg(\bfx)\q d\bfx \q\geq\q L\int_{x_2> L}\omg(\bfx)\q d\bfx$$ and so 
    $$\nrm{\omg}_1\q\leq\q \int_{0<x_2\leq L}\omg(\bfx)\q d\bfx+ML^{-1}.$$ 
    On the other hand, for any $\alp>0$, we have 
\begin{equation}\label{eq_sum to mass}
    \int_{\alpha<x_2\leq2\alpha}\omg\q d\bfx\q\leq\q
\int_{\alpha<x_2\leq2\alpha}\omg\cdot\f{\calG[\omg]}{Wx_2}\q d\bfx\q\leq\q
\f{1}{W\alp}\int_{\alpha<x_2\leq2\alpha}\calG[\omg]\omg\q d\bfx.
\end{equation} 
Using the symmetry of $G$ and Lemma \ref{lem_est.of.G}, we obtain 
$$\int_{\alpha<x_2\leq2\alpha}\calG[\omg]\omg\q d\bfx=
\int_{\bbR^2_+}\Bigg[\int_{\alpha<x_2\leq2\alpha}G(\bfy,\bfx)\omg(\bfx)\q d\bfx\Bigg]\omg(\bfy)\q  d\bfy\q\lesssim\q
\alp^{3/2}\int_{\bbR^2_+}y_2^{1/2}\omg(\bfy)\q d\bfy
\q\leq\q \alp^{3/2}\nrm{x_2\omg}^{1/2}_1\nrm{\omg}_1^{1/2}\leq\alp^{3/2}M^{1/2},$$ and so we get 
$$\int_{\alpha<x_2\leq2\alpha}\omg\q d\bfx\q\lesssim\q 
\f{M^{1/2}}{W}\alp^{1/2}.$$ Therefore we have 
\begin{equation}\label{eq_pure bound of mass}
    \nrm{\omg}_1\q\leq\q \int_{0<x_2\leq L}\omg(\bfx)\q d\bfx+ML^{-1}\q=\q 
\sum_{n\geq0}\int_{L/2^{n+1}<\q  x_2\q\leq\q L/2^n}\omg(\bfx)\q d\bfx+ML^{-1}\q\lesssim\q \f{M^{1/2}}{W}L^{1/2}+ML^{-1}.
\end{equation}
As the choice of $L>0$ is arbitrary, we put $L=W^{2/3}M^{1/3}$ to obtain 
$$\nrm{\omg}_1\q\lesssim\q \Bigg(\f{M}{W}\Bigg)^{2/3}.$$

\noindent Knowing that $\nrm{\omg}_1\lesssim (M/W)^{2/3}$, we will show that the following never holds: \textit{there exists a sequence of impulse $\set{M_n}\searrow0$, and a sequence ${\omg_n}\subset S_{M_n}$, together with a sequence of positive flux constant $\gmm_n>0$. }\\

\noindent To prove it, we assume that such sequences $\set{M_n},\q \set{\omg_n}$ and $\set{\gmm_n}$ exist. We may assume that $M_n<1$ for each $n\geq1$, and we fix any $n\geq1$. By the assumption $\gmm_n>0$, we have $\nrm{\omg_n}_1=1$ by Lemma \ref{lem_lack of mass to gmm=0}. Therefore, by the relation $\nrm{\omg}_1\lesssim (M/W)^{2/3}$ and Lemma \ref{lem_WM=energy},  we get $$W_n\lesssim M_n,\qd\mbox{ and  }\qd I_{M_n}\lesssim {M_n}^2+\gmm_n.$$  Moreover, as $M_n\in(0,1)$, by Lemma \ref{lem_lower bound of energy} we have the relation  
$$M_n^{4/3}\q\lesssim\q M_n^2+\gmm_n.$$ Take $N>1$ large enough so that we have $M_n^{4/3}\lesssim \gmm_n$ for each $n\geq N$. Recall that we used the relation $\q\omg\q\leq\q \omg\cdot\calG[\omg]/Wx_2$ in the inequality \eqref{eq_sum to mass}. As $M_n^{4/3}\lesssim\gmm_n$, we can instead use the relation $\q\omg_n\q\leq\q \omg_n\cdot \calG[\omg_n]/\gmm_n$ to obtain that

$$\int_{\alpha<x_2\leq2\alpha}\omg_n\q d\bfx\q\lesssim\q 
M_n^{-5/6}\alp^{3/2},\qd\mbox{ for any }\q \alp>0.$$
Similarly with \eqref{eq_pure bound of mass}, we then obtain the following relation:
$$1\q=\q\nrm{\omg_n}_1\q\leq\q  
\sum_{k\geq0}\int_{L/2^{k+1}<\q  x_2\q\leq\q L/2^k}\omg_n(\bfx)\q d\bfx+M_nL^{-1}\q\lesssim\q M_n^{-5/6}L^{3/2}+M_nL^{-1},\qd\mbox{ for any }\q L>0.$$ Then by choosing $L=M_n^{11/15}$, we have 
$$1\q\lesssim\q M_n^{4/15},$$ which contradicts to our assumption that $\set{M_n}\searrow0$. This completes the proof.
\end{proof}\q

\subsection{Touching of pair of uniform vortices }\q

By Lemma \ref{lem_small impulse to gmm 0}, for small $M>0$ and for each maximizer $\omg\in S_M$, the corresponding flux constant $\gmm$ vanishes such that we have $$\omg=\bfone_{\set{\q\calG[\omg]-Wx_2>0\q}}$$ for the traveling speed $W>0$. 
We now assume that the set $\Omg:=\set{\q\calG[\omg]-Wx_2>0\q}\subset\bbR^2_+$ satisfies the Steiner symmetry condition, i.e. $\Omg$ is concentrated along $x_2$-axis. In order to prove that the set $\Omg$ touches $x_1$-axis at the origin, we have to prove
$$\f{\rd}{\rd x_2}\Big(\calG[\omg](x_1,x_2)-Wx_2\Big)\q>\q0\qd\mbox{ for any small }\q |\bfx|>0.$$ 
In other words, we have to estimate the horizontal velocity $u^1$ of the fluid near the center of the dipolar vortex. Lemma \ref{lem_ central speed > 2W } below gives us a strong estimate, which would guarantee that vanishing of $\gmm$ is an equivalent condition to touching of a vortex dipole patch.

\subsubsection{Central speed estimation}\q

\begin{lem}\label{lem_ central speed > 2W }
For some constants $W>0$ and $\gmm\geq0$, if a function $\omg$ is given by the relation $\omg=\bfone_\Omg\q$ for some open bounded set $\Omg\subset\bbR^2_+$ satisfying the Steiner symmetry condition in Definition \ref{def_Steiner sym} and the identity
$$\Omg=\set{\q\bfx\in\bbR^2_+\q:\q\calG[\omg](\bfx)-Wx_2-\gmm>0\q},$$ we have 
    $$ \tld{\calG[\omg]}_{x_2}(0,0)>2W $$ where $\tld{\calG[\omg]}$ is the odd extention of $\calG[\omg]$ to $\bbR^2$. 
\end{lem}
\begin{proof}\q Let $\omg\q=\q\bfone_{\set{\bfx\in\bbR^2_+\q:\q\calG[\omg](\bfx)-Wx_2-\gmm>0}}$ with some constants $W>0$ and $\gmm\geq0$. For each $\bfx\in\bbR^2$, we put
    $$\tld{\calG[\omg]}_{x_2}(\bfx)\q=\q \f{1}{2\pi}\int_{\bbR^2}\f{-x_2+y_2}{|\bfx-\bfy|^2}\q\tld{\omg}(\bfy)\q d\bfy\q=\q
\f{1}{2\pi}\int_{\bbR^2_+}\f{-x_2+y_2}{|\bfx-\bfy|^2}\q\omg(\bfy)\q d\bfy\q+\q 
\f{1}{2\pi}\int_{\bbR^2_+}\f{x_2+y_2}{|\bfx-\bfy^*|^2}\q\omg(\bfy)\q d\bfy\q:=\q u^1_l(\bfx)+u^1_{nl}(\bfx)
$$ where $\bfy^*=(y_1,-y_2)$. By the relation $\tld{\calG[\omg]}_{x_2}(0,0)=2 u^1_{nl}(0,0)$, we need to show that $$u^1_{nl}(0,0)> W.$$ 
Note that the identity \eqref{eq_W formula} in Remark \ref{rmk_W formula} implies that $$W=\nrm{\omg}_1^{-1}\cdot\int_{\bbR^2_+}u^1_{nl}\cdot\omg\q d\bfx,$$ which means that $W$ is the average of $u^1_{nl}$ in 
 the set $\supp{(\omg)}$. We will claim that, assuming the Steiner symmetry condition of $\omg$, the function $u^1_{nl}$ attains its unique maximum at $\bfx=(0,0)$ among all points in $\overline{\bbR^2_+}$, which would imply that $$u^1_{nl}(0,0)>u^1_{nl}(\bfx)\qd\mbox{ for all }\bfx\in\supp{(\omg)}\backslash\set{(0,0)}$$ and therefore $u^1_{nl}(0,0)> W.$\\
 
 \noindent\textit{(Claim 1)\q} For any $(x_1,x_2)\in\overline{\bbR^2_+}$, we have $ u^1_{nl}(x_1,x_2)=u^1_{nl}(-x_1,x_2)$. Moreover, for fixed $x_2\geq0$, the function $u^1_{nl}(x_1,x_2)$ decreases in $x_1\geq0$.\\
 
Let $(x_1,x_2)\in\overline{\bbR^2_+}$ and assume that $x_1\neq 0$. We have  
$$u^1_{nl}(x_1,x_2)\q=\q
\f{1}{2\pi}\int_{\bbR^2_+}\f{y_2+x_2}{y_1^2+(y_2+x_2)^2}\omg(y_1+x_1,y_2)\q d\bfy.$$
For any fixed $y_2>0$ satisfying that $\omg(0,y_2)=1$, observe that there exists a constant $l=l(y_2)>0$ such that 
$$\Big\{\q y_1\in\bbR\q:\q \omg(y_1,y_2)=1\q\Big\}\q=\q\Big(-l(y_2),\q l(y_2)\Big).$$ Put $A:=\big\{\q y_2>0\q:\q\omg(0,y_2)=1\q\big\}$, which is a nonempty open set, and we get 
$$u^1_{nl}(x_1,x_2)\q=\q
\f{1}{2\pi}\int_A\int_{-l(y_2)-x_1}^{l(y_2)-x_1}\f{y_2+x_2}{y_1^2+(y_2+x_2)^2}dy_1\q dy_2.$$ Observe that, for fixed $y_2\in A$, the following function 
$$f(x_1)\q:=\int_{-l(y_2)-x_1}^{l(y_2)-x_1}\f{y_2+x_2}{y_1^2+(y_2+x_2)^2}dy_1,\qd x_1\in \bbR$$ satisfies $f(x_1)=f(-x_1)$ and decreases in $x_1\geq0$. This completes the proof of \textit{(Claim 1)}.\\

\noindent\textit{(Claim 2)\q} The function $g(x_2):=u^1_{nl}(0,x_2)$ decreases in $x_2\geq0$. \\

For each $x_2\geq0$, 
$$g(x_2)\q=\q\f{1}{2\pi}\int_{\bbR^2_+}\f{y_2+x_2}{y_1^2+(y_2+x_2)^2}\omg(\bfy)\q d\bfy.$$ 
As in the proof of \textit{(Claim 1)}, we have 
$$g(x_2)\q=\q\f{1}{2\pi}\int_A\int_{-l(y_2)}^{l(y_2)}\f{y_2+x_2}{y_1^2+(y_2+x_2)^2}\q dy_1\q dy_2 
\q=\q \f{1}{\pi}\int_A \arctan{\Bigg(\f{l(y_2)}{y_2+x_2}\Bigg)}\q dy_2.$$ It implies that $g(x_2)$ decreases strictly in $x_2\geq0$. This completes the proof of \textit{(Claim 2)}.\\

By \textit{(Claim 1)} and \textit{(Claim 2)}, we have 
$$ u^1_{nl}(0,0)\q>\q
\nrm{\omg}_1^{-1}\cdot\int_{\bbR^2_+}u^1_{nl}\cdot\omg\q d\bfx
\q=\q W,$$ and so we have $\tld{\calG[\omg]}_{x_2}(0,0)>2W.\q$ This completes the proof.
\end{proof} \q

\subsubsection{Proof of Theorem \ref{thm_touching of dipole}}\q

Lemma \ref{lem_ central speed > 2W } implies that each maximizer $\omg\in S_M$ touches $x_1$-axis if and only if we get $\gmm=0$. Recall Lemma \ref{lem_small impulse to gmm 0} which says that a sufficient condition for $\gmm$ to vanish is the smallness condition of the impulse $M>0$. We will prove  Theorem \ref{thm_touching of dipole} here.

\begin{proof}[Proof of Theorem \ref{thm_touching of dipole}]\q Let $M_1>0$ be the constant in Lemma \ref{lem_small impulse to gmm 0}. For any $M\in(0,M_1)$, as the set $S_M$ is nonempty, we choose any $\omg\in S_M$. By Lemma \ref{lem_small impulse to gmm 0}, we have $\gmm=0$ and thus 
    $$\omg=\bfone_{\set{\q\calG[\omg]-Wx_2>0\q}},$$ where $W,\gmm$ are the constants in Theorem \ref{thm_maximizer}. It suffices to show that there exists $\eps>0$ such that $$B_\eps(0,0)\cap\bbR^2_+\q\subset\q \Big\{\q\calG[\omg]-Wx_2>0\q\Big\}.$$
    By Lemma \ref{lem_ central speed > 2W }, we have $$\tld{\calG[\omg]}_{x_2}(0,0)\q>\q2W.$$ By $\tld{\calG[\omg]}\in C^1(\bbR^2)$, there exists a small $\eps>0$ such that $$\q\tld{\calG[\omg]}_{x_2}>2W\qd\mbox{ in }\q B_\eps(0,0).$$ Then for each $\bfx=(x_1,x_2)\in B_\eps(0,0)\cap\bbR^2_+$, we have 
    $$\calG[\omg](\bfx)=\ii{0}{x_2}\calG[\omg]_{x_2}(x_1,s)\q ds\q>\q 2Wx_2>Wx_2,$$ and therefore $\bfx\in\{\q\calG[\omg]-Wx_2>0\q\}.$ This completes the proof.
    
\end{proof}\q

\subsubsection{Proof of Theorem \ref{thm_boundry is continuous}}\q

In the proof below, we obtain the continuity of the boundary of each vortex patch by the continuity of the stream function and the implicit function theorem for which we prove the monotonicities of the stream function in $x_1>0$ and the horizontal speed $u^1$ on the axis of symmetry $\set{x_2=0}$. 

\begin{proof}[Proof of Theorem \ref{thm_boundry is continuous}]\q Let $M_1>0$ be the constant in Lemma \ref{lem_small impulse to gmm 0}. For any $M\in(0,M_1)$, as the set $S_M$ is nonempty, we choose any $\omg=\bfone_{\set{\calG[\omg]-Wx_2>0}}\in S_M$ where $W=W(\omg)>0$ is the constant in Theorem \ref{thm_maximizer}. Note that $\gmm=\gmm(\omg)=0$ by 
    Lemma \ref{lem_small impulse to gmm 0}. We assume that $\omg$ satisfies the Steiner symmetry condition. To prove (i), we define a set $$A:=\set{\q y_2>0\q:\q\omg(0,y_2)=1\q}=
    \set{\q y_2>0\q:\q\calG[\omg](0,y_2)-Wy_2>0\q}$$ which is clearly a nonempty set.  We observe that, by Theorem \ref{thm_touching of dipole}, there exists $\eps>0$ satisfying $(0,\eps)\subset A$. Moreover, the set $\set{\calG[\omg]-Wx_2>0}$ is open bounded set in $\bbR^2_+$ by Theorem \ref{thm_maximizer}, the set $A$ is open bounded subset of $(0,\ift)$. By the Steiner symmetry condition of $\omg$, we can define a function $l:(0,\ift)\to[0,\ift)$ satisfying that, for each $y_2\in A,$
    $$l(y_2)>0\qd\mbox{ and }\qd \set{\q y_1\in\bbR\q:\q\omg(y_1,y_2)=1\q}=(-l(y_2),\q l(y_2)),$$
    and $l\equiv0$ in $(0,\ift)\backslash A$. This proves (i). For the proof of (ii), we first make the following claims. \\
    
 \noindent\textit{(Claim 1)}\q For each $y_2\in \overline{A}$, we have $\q\calG[\omg]\big(\q l(y_2),\q y_2\big)=Wy_2$. \\

    Let $y_2\in A$. For any $y_1\in \big(0,\q l(y_2)\big)$, we have $\omg(y_1,y_2)=1$, and so $\calG[\omg](y_1,y_2)>Wy_2.\q$ If $y_1\in \big(\q l(y_2),\q \ift\big)$, then $\omg(y_1,y_2)=0$, and so we have 
    $\calG[\omg](y_1,y_2)\leq Wy_2.\q$ By the continuity of $\calG[\omg]$, we have 
    $$\calG[\omg]\big(\q l(y_2),\q y_2\big)=Wy_2.$$ For each $y_2'\in \rd{A}$, we have $l(y_2')=0$ by definition of $l$, and a similar argument can be done to obtain that 
    $$\calG[\omg](0,y_2')-Wy_2'=0.$$
    This proves \textit{(Claim 1)}.\\

    \noindent\textit{(Claim 2)} \q For each $\bfx=(x_1,x_2)\in \bbR^2_+$ satisfying $x_1>0$, we have $\q\calG[\omg]_{x_1}(\bfx)<0$. \\
    
    For each $\bfx=(x_1,x_2)\in \bbR^2_+$ satisfying $x_1>0$, observe that 
    \begin{eqnarray*}
        2\pi\cdot\calG[\omg]_{x_1}(\bfx)&=&\int_{\bbR^2_+}\omg(\bfy)(y_1-x_1)\Bigg[\q\f{1}{|\bfy-\bfx|^2}-\f{1}{|\bfy^*-\bfx|^2}\q\Bigg]d \bfy\\
        &=&\int_A\ii{-l(y_2)}{l(y_2)}\q\f{4x_2y_2\q(y_1-x_1)}{|\bfy-\bfx|^2|\bfy^*-\bfx|^2}\q dy_1\q dy_2\\
        &=& \int_A\ii{-l(y_2)-x_1}{l(y_2)-x_1}\f{y_1\cdot(4x_2y_2)}{(y_1^2+(y_2-x_2)^2)\cdot(y_1^2+(y_2+x_2)^2)}\q dy_1\q dy_2\\
        &=& \int_A \Big[F\big(l(y_2)-x_1\big)-F\big(-l(y_2)-x_1\big) \Big]\q dy_2
    \end{eqnarray*} 
   where, for each fixed $y_2\in A\backslash\set{x_2}$, the function $F:\bbR\to\bbR$ is defined as  
    $$F(s):=\ii{0}{s}\f{y_1\cdot(4x_2y_2)}{(y_1^2+(y_2-x_2)^2)\cdot(y_1^2+(y_2+x_2)^2)}\q dy_1.$$ Note that $F$ is an even function in $s\in\bbR$ by the odd symmetry of the integrand in $y_1\in\bbR$. Moreover, we observe that $F(s)$ increases strictly in $|s|\geq0$, since the integrand is positive in $y_1>0$. For each $y_2\in A\backslash\set{x_2}$, we have $l(y_2)>0$ and so we get 
    $$\big|l(y_2)-x_1\big|<\big|-l(y_2)-x_1\big|,\qd\mbox{ implying that }\qd 
    F\big(l(y_2)-x_1\big)<F\big(-l(y_2)-x_1\big).$$ 
    This gives $\calG[\omg]_{x_1}(\bfx)<0$ and completes the proof of \textit{(Claim 2)}.\\

      We will now prove (ii). First we want to prove that $l\in C^{1,r}_{loc}(A)$ for each $r\in(0,1).$ We define a $C^1$ function $\Psi:\bbR^2_+\to\bbR$ as $$\Psi(\bfx):=\calG[\omg](\bfx)-Wx_2.$$
    For each $y_2\in A$, by \textit{(Claim 1)} and \textit{(Claim 2)},  we have $\Psi\big( l(y_2),\q y_2\big)=0$ and $\Psi_{x_1}\big( l(y_2),\q y_2\big)<0$. Then the implicit function theorem can be applied to $\Psi$ at each point $\big( l(y_2),\q y_2\big)$. That is, for each $y_2\in A$, there exists a small neighborhood $U\subset A$ containing $y_2$ such that we have $$\Psi(f(s),\q s)=0, \qd s\in U$$ for some function $f:U\to \bbR$ which is continuously differentiable in $U$ and satisfies $f(y_2)=l(y_2)$. In \textit{(Claim 2)}, we proved that $\Psi(x_1,x_2)$ is strictly decreasing in $x_1>0$, and it follows that $f(s)=l(s)$ for each $s\in U$. In conclusion, $l$ is continuously differentiable in $A$. To show that $l'\in C^{r}_{loc}(A)$ for each $r\in(0,1)$, we fix any  $r\in(0,1)$. It suffices to show that, for any $s\in A$ and $\eps>0$ satisfying that $[s-\eps,s+\eps]\subset A$, we have 
    $$
    \sup_{0<|s-t|\leq\eps}\f{|\q l'(s)-l'(t)\q|}{|s-t|^r}<\ift.
    $$ By the formula $\Psi(l(s),s)=0$, we have 
    $$l'(s)=-\f{\Psi_{x_2}(l(s),s)}{\Psi_{x_1}(l(s),s)}.$$ Fix $s\in A$ and $\eps>0$ to satisfy $[s-\eps,s+\eps]\subset A$. Then for any $t\in [s-\eps,s+\eps]$, we have 
    $$|\q l'(s)-l'(t)\q|\q\lesssim\q\f{\big|(\nb \Psi)(l(t),t)\q\big|}{\big|\q \Psi_{x_1}(l(s),s)\cdot \Psi_{x_1}(l(t),t)\q \big|}\cdot\big|\q(\nb \Psi)(l(s),s)-(\nb \Psi)(l(t),t)\q\big|.$$ We deal with the denominator first. As the continuous function $g(y):=\Psi_{x_1}(l(y),y)$ satisfies $g<0$ in the domain $A$, there exists a constant $C_0>0$ such that 
    $g<-C_0$ in $[s-\eps,s+\eps]$. Therefore we have 
    $$|\q l'(s)-l'(t)\q|\q\lesssim\big|(\nb \Psi)(l(t),t)\q\big|\cdot\big|\q(\nb \Psi)(l(s),s)-(\nb \Psi)(l(t),t)\q\big|.$$ By Lemma \ref{lem_gw:C1}, the relation $\nb \Psi=\nb\calG[\omg]-(0,W)$ leads to $\nrm{\nb \Psi}_{L^\ift}\lesssim \nrm{\nb\calG[\omg]}_{L^\ift}+W<\ift$, so we finally get 
    $$|\q l'(s)-l'(t)\q|\q\lesssim\big|\q(\nb \Psi)(l(s),s)-(\nb \Psi)(l(t),t)\q\big|.$$
    As $\nb \Psi\in C^r(\bbR^2_+)$ by Lemma \ref{lem_gw:C1}, we have 
    $$|\q l'(s)-l'(t)\q|\q\lesssim_r\q|\q(l(s)-l(t))^2+(s-t)^2\q|^{r/2}\q\leq\q
    |s-t|^r\cdot\big|\q\nrm{l'}_{L^\ift([s-\eps,s+\eps])}^2+1\q\big|^{r/2}$$ where $\nrm{l'}_{L^\ift([s-\eps,s+\eps])}<\ift$ due to $l'\in C(A).$ Therefore we have $l'\in C^r_{loc}(A)$. \\
    
    \noindent It remains to show that $l\in C((0,\ift))$. Note that, at each point in $A$ and $(0,\ift)\backslash\overline{A}$, the continuity of $l$ is trivially obtained using the fact that $A$ is an open set. For each $y_2\in (0,\ift)\cap \rd A$, we observe that $$\calG[\omg](0,y_2)=Wy_2\qd\mbox{ and }\qd l(y_2)=0$$ by \textit{(Claim 1)}. We need to show that $\lim_{s\to y_2}l(s)=0$. As $\omg$ is boundedly supported in $\bbR^2_+$, we have 
    $$\limsup_{s\to y_2}l(s)<\ift.$$ If $l(s)$ does not converges to 0 as $s\to y_2$, there exists a sequence $\set{s_n}\to y_2$ and a number $L\in (0,\ift)$ such that $\lim_{n\to\ift}l(s_n)=L.$ By the continuity of $\calG[\omg]$, we have $$\calG[\omg](L,y_2)-Wy_2=0,$$ which means that $\calG[\omg](0,y_2)=\calG[\omg](L,y_2)$. It contradicts to \textit{(Claim 2)}. In sum, we have $l\in C((0,\ift))$.  \\

    Finally, we now prove (iii), starting from the following claim. \\
    
    \noindent\textit{(Claim 3)} $\qd\tld{\calG[\omg]}_{x_2}(s,0)$ strictly decreases in $s\geq0$. \\

    For $s\geq0$, we have 
    $$\tld{\calG[\omg]}_{x_2}(s,0)=\f{1}{\pi}\int_{\bbR^2_+}\f{y_2}{(y_1-s)^2+y_2^2}\omg(\bfy)\q d\bfy=\f{1}{\pi}\int_A\ii{-l(y_2)-s}{l(y_2)-s}\f{y_2}{y_1^2+y_2^2}dy_1dy_2.$$ 
    For fixed $y_2\in A$, observe that the integrand 
    $$\ii{-l(y_2)-s}{l(y_2)-s}\f{y_2}{y_1^2+y_2^2}dy_1$$ strictly decreases in $s\geq0$. This completes the proof of \textit{(Claim 3)}. \\
    
    Note that we have $$\tld{\calG[\omg]}_{x_2}(0,0)>2W$$ by Lemma \ref{lem_ central speed > 2W }. By \eqref{eq_est of gw/x_2} in Lemma \ref{lem_w:cpt_supp}, there exists $K\in(0,\ift)$ such that we have $$x_1>K\q\imp\q\tld{\calG[\omg]}_{x_2}(x_1,0)=
    \lim_{x_2\searrow0}\f{\calG[\omg](x_1,x_2)}{x_2}\q\leq\q W.$$ As $\tld{\calG[\omg]}_{x_2}(s,0)$ decreases strictly in $s\geq0$, there exists a unique constant $a\in(0,\ift)$ satisfying 
    $$\tld{\calG[\omg]}_{x_2}(a,0)=W.$$
    It remains to show that $\lim_{s\to 0^+}l(s)=a$. As $\omg$ is boundedly supported, we have  
    $$\limsup_{s\searrow0}{l(s)}<\ift.$$
   Suppose that $l(s)$ does not converge to $a$ as $s\searrow0$. Then there exists a sequence $s_n'\searrow 0$ and a number $a'\in(0,\ift)$ satisfying that 
   $$\lim_{n\to\ift} l(s_n')\q=\q a'\q\neq\q a.$$
   In the relation $$W=\f{\calG[\omg]\big(\q l(s_n'),\q s_n'\big)}{s_n'},$$ due to $\calG[\omg]\big(\q l(s_n'),\q 0\big)=0,$ we can use the mean value theorem for each $n\geq1$ to find a sequence $0<s_n''<s_n'$ such that 
   $$W=\tld{\calG[\omg]}_{x_2}\big(\q l(s_n'),\q s_n''\big).$$ Letting $n\to\ift$ gives $W=\calG[\omg]_{x_2}(a',0)$, which contradicts to $a'\neq a$. This completes the proof of (iii).
\end{proof}\q

\subsection{Estimates when flux constant vanishes due to small impulse}\q

Under the small impulse condition, knowing that $\gmm=0$ by Lemma \ref{lem_small impulse to gmm 0}, we can obtain several estimates depending only on $M$. \\

\begin{lem}\label{lem_ests for small impulse}
    There exists an absolute constant $C\geq1$ such that the following hold: \\ 
    
    Let $M_1>0$ be the constant in Lemma \ref{lem_small impulse to gmm 0}. For any $M\in(0,M_1)$ and for each $\omg\in S_M$, we have 

\begin{equation}\label{eq_small impulse-mass}
        \f{1}{C}\cdot M^{2/3}\q\leq\q\nrm{\omg}_1\q\leq\q C\cdot M^{2/3},
        \end{equation} 
\begin{equation}\label{eq_small impulse-energy}
    \f{1}{C}\cdot M^{4/3}\q\leq\q I_M\q\leq\q C\cdot M^{4/3},
\end{equation} 

\begin{equation}\label{eq_small impulse-W}
    \f{1}{C}\cdot M^{1/3}\q\leq\q W \q\leq\q C\cdot M^{1/3},
\end{equation} 
\begin{equation}\label{eq_small impulse-gw<mu}
        \nrm{\calG[\omg]}_{L^\ift(\bbR^2_+)}\q\leq\q C\cdot M^{2/3},
\end{equation} 
where $W=W(\omg)>0$ is given in Theorem \ref{thm_maximizer}. 

\end{lem}

\begin{proof}\q   Let $M_1>0$ be the constant in Lemma \ref{lem_small impulse to gmm 0} and choose $M\in(0,M_1)$. As $S_M$ is nonmepty, we choose any $\omg\in S_M$ with the constants $W=W(\omg)>0$ and $\gmm=\gmm(\omg)\geq0$ given in Theorem \ref{thm_maximizer}. Note that $\gmm=0$ by Lemma \ref{lem_small impulse to gmm 0}.  We will prove each inequality separately.\\ 

        \noindent\textit{(Claim 1)}  $\qd M^{4/3}\lesssim I_M\q$ and $\q M^{1/3}\lesssim W$.\\

        Lemma \ref{lem_lower bound of energy} says that $M^{4/3}\lesssim I_M$ given that $M<1$.  Using the identity $I_M= (3/4)WM$ in Lemma \ref{lem_WM=energy}, we obtain $M^{1/3}\lesssim W$. This completes the proof of \textit{(Claim 1)}.\\

        \noindent\textit{(Claim 2)}  $\qd\nrm{\omg}_1\lesssim M^{2/3}$.\\

       Before we estimate $\nrm{\omg}_1$, recall that $\calG[\omg]\in L^\ift(\bbR^2_+)$ by \eqref{eq_gw:bdd} in Lemma \ref{lem_gw:estiate} and that we have 
\begin{equation}\label{eq_supp w is bdd from above }
    \Big\{\q\calG[\omg]-Wx_2>0\q\Big\}\subset \Big\{\q x_2<\nrm{\calG[\omg]}_{L^\ift(\bbR^2_+)}/W\q\Big\}\qd\mbox{ which implies }\qd 
        \sup_{\bfx\in\supp{(\omg)}}{x_2}  \q\leq\q
        \f{\nrm{\calG[\omg]}_{L^\ift(\bbR^2_+)}}{W}.
\end{equation}
    
        \noindent With the fact that $\nrm{\omg}_2=\nrm{\omg}_1^{1/2}$, the estimate \eqref{eq_gw:bdd} in Lemma \ref{lem_gw:estiate} gives 
$$\nrm{\calG[\omg]}_{L^\ift(\bbR^2_+)}\q\lesssim\q \nrm{\omg}_1^{1/2}M^{1/3}+\nrm{\omg}_1^{1/4}M^{1/2}.$$ 
Using the relation $M^{1/3}\lesssim W$ in \textit{(Claim 1)}, we have 
 $$\sup_{\bfx\in\supp{(\omg)}}{x_2}\q\lesssim\q \nrm{\omg}_1^{1/2}+\nrm{\omg}_1^{1/4}M^{1/6}.$$ 
Put $K:=\sup_{\bfx\in\supp{(\omg)}}{x_2}$ and observe that 
 $$\nrm{\omg}_1\q=\q\int_{x_2\q\leq\q K}\omg\q d\bfx \qd\mbox{ where }\qd K\q\lesssim\q\nrm{\omg}_1^{1/2}+\nrm{\omg}_1^{1/4}M^{1/6}.$$ 
We will now use Lemma \ref{lem_est.of.G} to estimate $\nrm{\omg}_1$ in terms of $K$. Observe that 
 $$\nrm{\omg}_1\q=\q\sum_{n\geq0}\int_{K/2^{n+1}<\q x_2\q \leq\q K/2^n}\omg(\bfx)\q d\bfx.$$

\noindent For any $\alp>0$, we have 
$$\int_{\alpha<x_2\leq2\alpha}\omg\q d\bfx\q\leq\q
\int_{\alpha<x_2\leq2\alpha}\omg\cdot\f{\calG[\omg]}{Wx_2}\q d\bfx\q\leq\q
\f{1}{W\alp}\int_{\alpha<x_2\leq2\alpha}\calG[\omg]\omg\q d\bfx\q\lesssim\q
\f{1}{M^{1/3}\alp}\int_{\alpha<x_2\leq2\alpha}\calG[\omg]\omg\q d\bfx.$$
 Using the symmetry of $G$ and Lemma \ref{lem_est.of.G}, we obtain 
$$\int_{\alpha<x_2\leq2\alpha}\calG[\omg]\omg\q d\bfx=
\int_{\bbR^2_+}\Bigg[\int_{\alpha<x_2\leq2\alpha}G(\bfy,\bfx)\omg(\bfx)\q d\bfx\Bigg]\omg(\bfy)\q  d\bfy\q\lesssim\q
\alp^{3/2}\int_{\bbR^2_+}y_2^{1/2}\omg(\bfy)\q d\bfy
\q\leq\q \alp^{3/2}M^{1/2}\nrm{\omg}_1^{1/2}.$$
Therefore we have 
        $$\int_{\alpha<x_2\leq2\alpha}\omg\q d\bfx\q\lesssim\q
        \alp^{1/2}M^{1/6}\nrm{\omg}_1^{1/2}.$$
        \noindent Then 
        $$\nrm{\omg}_1\q=\q\sum_{n\geq0}\int_{K/2^{n+1}<\q x_2\q \leq\q K/2^n}\omg(\bfx)\q d\bfx\q\q\lesssim\q\q M^{1/6}\nrm{\omg}_1^{1/2}\sum_{n\geq0}\Big(\f{K}{2^{n+1}}\Big)^{1/2}\q\lesssim\q K^{1/2}M^{1/6}\nrm{\omg}_1^{1/2}.$$
        In sum, we obtain 
        $$\nrm{\omg}_1\q\lesssim\q M^{1/6} \nrm{\omg}_1^{1/2}
        \Bigl(\nrm{\omg}_1^{1/2}+\nrm{\omg}_1^{1/4}M^{1/6}\Bigl)^{1/2}$$ 
        and therefore 
        $$\nrm{\omg}_1^{3/4}\q\lesssim\q M^{1/3}\nrm{\omg}_1^{1/4}+M^{1/2}.$$ Under the substitution $t:=M^{-1/6}\nrm{\omg}_1^{1/4}$, we obtain $$t^3\lesssim t+1$$ and thus $t\lesssim 1$. In sum, we get 
        $$M^{-1/6}\nrm{\omg}_1^{1/4}\q\lesssim\q 1\qd\mbox{ and so } \qd 
        \nrm{\omg}_1\q\lesssim\q M^{2/3}.$$\q

\noindent\textit{(Claim 3)} $\q I_M\lesssim M^{4/3},\q$ $\q\nrm{\calG[\omg]}_{L^\ift(\bbR^2_+)}\lesssim M^{2/3},\q$ and $\q W\lesssim M^{1/3}$.\\

Using the relation $\nrm{\omg}_1\lesssim M^{2/3}$ and the fact that $\nrm{\omg}_2=\nrm{\omg}_1^{1/2}$, the estimate \eqref{eq_energy:finite} in Lemma \ref{lem_gw:estiate} gives 
$$I_M\q\lesssim\q M^{1/2}\nrm{\omg}_1^{5/4}\q\lesssim\q M^{4/3},$$ and the estimate \eqref{eq_gw:bdd} in Lemma \ref{lem_gw:estiate} gives 
$$\nrm{\calG[\omg]}_{L^\ift(\bbR^2_+)}\q\lesssim\q 
\nrm{\omg}_1^{1/2}M^{1/3}+\nrm{\omg}_1^{1/4}M^{1/2}\lesssim M^{2/3}.$$ \\
And lastly, using  Lemma \ref{lem_WM=energy}, we obtain 
$$W\q=\q \f{4I_M}{3M}\q\lesssim\q M^{1/3}.$$ It completes the proof of \eqref{eq_small impulse-energy}, \eqref{eq_small impulse-W}, and \eqref{eq_small impulse-gw<mu}.\\

\noindent\textit{(Claim 4)} $M^{2/3}\lesssim\nrm{\omg}_1$.\\

Using the relation $\nrm{\calG[\omg]}_{L^\ift(\bbR^2_+)}\lesssim M^{2/3}$ and $M^{1/3}\lesssim W$, the inequality \eqref{eq_supp w is bdd from above } implies that 

\begin{equation}\label{eq_small impulse : height of supp w}
    \sup_{\bfx\in\supp{(\omg)}}{x_2}\q\leq\q \f{\nrm{\calG[\omg]}_{L^\ift(\bbR^2_+)}}{W}\lesssim M^{1/3}.
\end{equation}
Therefore, we can observe that 
$$M\q=\q\int_{x_2\q\lesssim\q M^{1/3}}x_2\omg(\bfx)\q d\bfx 
\q\lesssim\q M^{1/3}\nrm{\omg}_1,$$ so we get $M^{2/3}\lesssim\nrm{\omg}_1$.
\end{proof}\q

Combining some estimates in Lemma \ref{lem_ests for small impulse}, we can roughly estimate the size of the compact support of maximizers in $S_M$.  

\begin{lem}\label{lem_size of w}
    There exists a constant $C>0$ satisfying the following:\\

    Let $M_1>0$ be the constant in Lemma \ref{lem_small impulse to gmm 0}. For any $M\in(0,M_1)$ and for each $\omg\in S_M$ satisfying the Steiner symmetry condition in Definition \ref{def_Steiner sym}, we have 
    \begin{equation}\label{eq_size of supp w}
        \sup_{\bfx\in\supp{(\omg)}}{|\bfx|}\q\leq\q C\cdot M^{1/3},
    \end{equation}
\end{lem}

\begin{proof}\q Let $M\in(0,M_1)$. As $S_M$ is nonmepty, we choose any $\omg=\bfone_{\{\calG[\omg]-Wx_2 >0\}}\in S_M$ with the constant $W=W(\omg)>0$ in Theorem \ref{thm_maximizer}. We assume that $\omg$ satisfies the Steiner symmetry condition. First, we recall the relation $$ \sup_{\bfx\in\supp{(\omg)}}{x_2}\q\leq\q\f{\nrm{\calG[\omg]}_{L^\ift}}{W}\q\lesssim\q M^{1/3}$$ which was obtained by the inequality \eqref{eq_small impulse : height of supp w} in the proof of Lemma \ref{lem_ests for small impulse}. It remains to show that 
$$ \sup_{\bfx\in\supp{(\omg)}}{|x_1|}\q\lesssim\q M^{1/3}.$$ We first observe that  
$$\supp{(\omg)}\q\subset\q
\overline{\Bigg\{\q\f{\calG[\omg](\bfx)}{x_2}>W\q\Bigg\}}$$
and recall the  Lemma \ref{lem_w:cpt_supp} to have 
$$\Bigg|\f{\calG[\omg](\bfx)}{x_2}\Bigg|\q\leq\q C_0\q\Big( \q\eps^{-1}\nrm{\omg}_1
+\eps^{1/3}\nrm{\omg}_{L^1\big([x_1-\eps, x_1+\eps]\times(0,\ift)\big)}^{1/3}\q\Big)$$ which holds for any $\bfx\in\bbR^2_+$ and $\eps>0$. Here $C_0>0$ is an absolute constant. We fix $\eps = 3C_0\nrm{\omg}_1/W$ to obtain 
$$\Bigg|\f{\calG[\omg](\bfx)}{x_2}\Bigg|\q\leq\q \f{W}{3}+C_0\eps^{1/3}
\nrm{\omg}_{L^1\big([x_1-\eps, x_1+\eps]\times(0,\ift)\big)}^{1/3}.$$
Suppose $x_1>\eps$ and observe that 
$$\nrm{\omg}_{L^1\big([x_1-\eps, x_1+\eps]\times(0,\ift)\big)}=\int_0^\ift\Bigg[\int_{x_1-\eps}^{x_1+\eps}\omg(t,x_2)\q dt\Bigg]dx_2\q\leq\q \f{1}{2}\nrm{\omg}_1\cdot\int_{x_1-\eps}^{x_1+\eps}\f{1}{t}dt
\q=\q\f{1}{2}\nrm{\omg}_1\ln\Bigg[\f{x_1+\eps}{x_1-\eps}\Bigg]
$$ where we have used the relation 
$$t\cdot\omg(t,x_2)\leq\int_0^\ift\omg(s,x_2)\q ds,\qd\mbox{ for any } t,\q x_2>0$$ which follows from the Steiner symmetry condition on $\omg$.  If we additionally assume $x_1>2\eps$, we have 
$$\nrm{\omg}_{L^1\big([x_1-\eps, x_1+\eps]\times(0,\ift)\big)}\q\leq\q 2\eps\q x_1^{-1}\nrm{\omg}_1$$ and therefore
$$\Bigg|\f{\calG[\omg](\bfx)}{x_2}\Bigg|\q\leq\q \f{W}{3}+C_1\eps^{2/3}\nrm{\omg}_1^{1/3}x_1^{-1/3}$$ for the constant $C_1:=2^{1/3}C_0$. If we suppose that $$x_1>C_2\f{\eps^2}{W^3}\nrm{\omg}_1\qd\mbox{ where }\q C_2:= (3C_1)^3,$$ we get 
$$\Bigg|\f{\calG[\omg](\bfx)}{x_2}\Bigg|\q\leq\q \f{2}{3}W,$$ and therefore $\bfx\in (\supp(\omg))^c$. In sum, we have the following, 
$$\supp(\omg)\q\subset \Bigg\{\q\bfx\in\bbR^2_+\q:\q |x_1| > \max\Big\{\q2\eps,\q C_2\f{\eps^2}{W^3}\nrm{\omg}_1 \q\Big\}\Bigg\}^c,\qd \mbox{ where }\q \eps=3C_0\nrm{\omg}_1/W.$$ 
Using the estimates \eqref{eq_small impulse-mass} and \eqref{eq_small impulse-W} in Lemma \ref{lem_ests for small impulse}, we get 
$$\eps\lesssim M^{1/3}\q\mbox{ and so }\q \f{\eps^2}{W^3}\nrm{\omg}_1\lesssim M^{1/3}.$$ 
This implies that 
$$ \sup_{\bfx\in\supp{(\omg)}}{|x_1|}\q\lesssim\q M^{1/3}.$$

\end{proof}

\section{ Variational problem on exact impulse without mass bound}\label{sec_ no mass bound}\q

\subsection{Variational problem without mass bound}\q

For $M>0$, we then consider the following admissible set without mass bound,  
   $$ K_{M,\ift}\q:=\q\Bigl\{\q\omg\in L^1 (\bbR_+^2)\q:\q\omg=\bfone_A\qd \mbox{a.e.}\qd\mbox{where }A\subset\bbR_+^2 \mbox{ is (Lebesgue) measurable, } 
   \qd\nrm{x_2\omg}_1=M \q\Bigl\},$$ with the maximal energy  
   $$  I_{M,\ift}\q:=\sup_{\omg\q\in\q K_{M,\ift}}E(\omg)$$ and the set of maximizers $$S_{M,\ift}:=\Big\{\q\omg\in K_{M,\ift}\q:\q E(\omg)=I_{M,\ift}\q\Big\}. $$
 Observe that this new admissible set $K_{M,\ift}$ does not contain any upper bound condition for mass. We note that the elementary estimate \eqref{eq_energy:finite} in Lemma \ref{lem_gw:estiate}: 
 $$\Bigl|\int_{\bbR_+^2}\int_{\bbR_+^2}G(\bfx,\bfy)\omg(\bfx)\omg(\bfy)\q d\bfx d\bfy\Bigl|\q\lesssim\q
\q\nrm{\omg}_1\nrm{\omg}_2^{1/2}\nrm{x_2\omg}_1^{1/2}$$
does not guarantee that $I_{M,\ift}<\ift$. In this section, we will prove that $I_{M,\ift}<\ift$ and $S_{M,\ift}\neq\emptyset$, which is quite surprising. Moreover, this new problem inherits all results from the previous variational problems. The main result of this section is given in Theorem \ref{thm_ no mass cond} below, and this section is dedicated to its proof. \\

\begin{thm}\label{thm_ no mass cond} We have the following properties: \\

\noindent(i) \q $I_{1,\ift}<\ift$.\\

\noindent(ii) \q $S_{1,\ift}$ is nonempty. \\

\noindent (iii) \q Each $\omg\in S_{1,\ift}$ satisfies the relation 
$$\omg\q=\q \bfone_{\{\q\calG[\omg]-Wx_2>0\q\}}\qd \mbox{ a.e. 
\qd in }\q \bbR^2_+,\qd\mbox{ for the traveling speed }\q 
W:=\f{4}{3}\q I_{1,\ift}. $$ 

\noindent(iv) \q Each $\omg \in S_{1,\ift}$ has a bounded support and satisfies the Steiner symmetry condition in Definition \ref{def_Steiner sym} up to translation. More precisely, for each $\omg\in S_{1,\ift}$, there exist $\tau\in\bbR$ and an open bounded set $A\subset\bbR^2_+$ satisfying the Steiner symmetry condition such that, for the translation $\omg_\tau$ given by $\omg_\tau(\cdot):=\omg(\cdot+(\tau,0))$, we have $$\omg_\tau\q=\q\bfone_A\qd\mbox{ a.e. in }\q\bbR^2_+.$$ 
\noindent(v) \q There exists an absolute constant $C\geq1$ such that, for each $\omg\in S_{1,\ift}$, we have 
$$\sup_{\bfx\in\supp{(\omg)}}|\bfx|\q\leq\q C\qd\mbox{ and }\qd\f{1}{C}\q\leq\q\nrm{\omg}_1\q\leq\q C.$$
\end{thm}\q

\begin{rmk}
    The key idea of the proof of Theorem \ref{thm_ no mass cond} is to show that $$S_{1,\ift}=S_{1,\nu,1}\qd\mbox{ for sufficiently large }\q \nu>0.$$  
\end{rmk}\q

\begin{rmk} 
    In the statements (i)-(iv) of Theorem \ref{thm_ no mass cond}, we can replace $I_{1,\ift},\q S_{1,\ift}$ with $I_{M,\ift},\q S_{M,\ift}$ for each $M>0$ by the scaling: The scaling map $\phi:\omg\mapsto\hat{\omg}$ given by 
$$ \hat{\omg}(\bfx)=\omg(M^{-1/3}\bfx) $$ maps the set $S_{1,\ift}$  to the set $S_{M,\ift}$ bijectively, with the scaled traveling speed and the maximal energy given by $$W_M:=W\cdot M^{1/3}\qd\mbox{ and }\qd I_{M,\ift}=I_{1,\ift}\cdot M^{4/3}.$$
\end{rmk}\q

\subsection{Scaling property}\q

Recall the estimate \eqref{eq_small impulse-mass} in Lemma \ref{lem_ests for small impulse}, which gives 
$$\nrm{\omg}_1\lesssim M^{2/3}.$$
As the impulse $M$ is small enough, each maximizer cannot achieve the full mass (which is 1) to attain the maximal energy. Using this property, we can show the following two lemmas.\\

\begin{lem}\label{lem_abundant mass}
    There exists a constant $M_2>0$ satisfying the following: \\
    
    For each $M\in(0,M_2)$ and any $\alp>1$, we have 
    $$S_M\q=\q S_{M,\alp,1}.$$
\end{lem}
\begin{proof}\q Denote the constant $C>1$ in Lemma \ref{lem_ests for small impulse} as $C_0>1$. Choose any small constant $M_2\in(0,1)$ satisfying that $$M_2\leq M_1\q\mbox{ and }\q C_0\cdot M_2^{2/3}\leq1.$$
Let $M\in(0,M_2)$  and  $\alp>1$. As $S_{M,\alp,1}$ is nonempty, choose any $\omg\in S_{M,\alp,1}$ and consider the scaling map
$$\hat{\omg}(\bfx)=\omg(\alp^{1/2}\bfx)$$ to attain 
$$\nrm{\hat{\omg}}_1=\alp^{-1}\nrm{\omg}_1\leq1,\qd 
\nrm{x_2\hat{\omg}}_1=\alp^{-3/2}\nrm{x_2\omg}_1=\alp^{-3/2}M$$
and so $\hat{\omg}\in K_{\alp^{-3/2}M}$. By the bijectivity of the scaling map from $S_{M,\alp,1}$ onto $S_{\alp^{-3/2}M},\q$ we get $\q\hat{\omg}\in S_{\alp^{-3/2}M}$. Here, as $\q\alp^{-3/2}M<M_1$, we can use the estimate \eqref{eq_small impulse-mass} in Lemma \ref{lem_ests for small impulse} to have 
$$\nrm{\hat{\omg}}_1\q\leq\q C_0\cdot(\alp^{-3/2}M)^{2/3},$$ and therefore 
$$\nrm{\omg}_1\q\leq\q C_0\cdot M^{2/3}<1.$$
This implies that $\omg\in K_M$. Using the assumption that $\omg\in S_{M,\alp,1}$ and the fact that $K_M\subset K_{M,\alp,1}$, we obtain $\omg\in S_M$. In sum, we have $S_{M,\alp,1}\subset S_M$, which leads to $S_{M,\alp,1}= S_M$.\end{proof}

\begin{cor}\label{cor_abundant mass}
Let $M_2>0$ be the constant in Lemma  \ref{lem_abundant mass}. Then for any $M,\q v>0$ satisfying $M\nu^{-3/2}<M_2$, we have 
$$S_{M,\nu,1}\q=\q S_{M,\alp,1},\qd \mbox{ for all }\alp>\nu.$$ 
\end{cor}
\begin{proof}\q Let $M,\q v,\q \alp>0$ and assume that $M\nu^{-3/2}<M_2$ and $\alp>\nu$. Note that the scaling map $\phi: \omg\mapsto\hat{\omg}$ given by $$\hat{\omg}(\bfx)=\omg(\nu^{1/2}\bfx)$$ maps the set $S_{M,\nu,1}$ onto the set $S_{v^{-3/2}M}$ bijectively and the set $S_{M,\alp,1}$ onto the set $S_{v^{-3/2}M,\q \alp/\nu,\q 1}$ bijectively. As  $\q M\nu^{-3/2}<M_2$ and $\alp/\nu>1$, by Lemma \ref{lem_abundant mass} we have 
    $$S_{v^{-3/2}M}\q=\q S_{v^{-3/2}M,\q \alp/\nu,\q 1},$$ and therefore 
    $$S_{M,\nu,1}\q=\q S_{M,\alp,1}.$$
\end{proof}\q

\subsection{Proof of Theorem \ref{thm_ no mass cond}}\q

\begin{proof}[Proof of Theorem \ref{thm_ no mass cond}]\q We first prove (i),(ii) by the following claim. \\

\noindent\textit{(Claim)} There exists a constant $\nu\in(0,\ift)$ such that we have
$$I_{1,\ift}=I_{1,\nu,1}=I_{1,\alp,1}\qd\mbox{ and }\qd S_{1,\ift}=S_{1,\nu,1}=S_{1,\alp,1}\qd\mbox{ for each }\alp>\nu.$$ 
First, observe that 
    $$K_{1,\ift}\q=\q \bigcup_{\tau>0}K_{1,\tau,1}.$$ Then we have 
    $$I_{1,\ift}\q=\q\sup\q  \Big\{\q E(\omg)\q:\q \omg \in \bigcup_{\tau>0}K_{1,\tau,1}\q \Big\}\q=\q 
    \sup\q  \Bigg[\bigcup_{\tau>0}\Big\{\q E(\omg)\q:\q \omg \in K_{1,\tau,1}\q \Big\}\Bigg]\q\leq\q 
    \limsup_{\tau\to\ift}{\q I_{1,\tau,1}}. $$ For the constant $M_2>0$ in Lemma \ref{lem_abundant mass}, we fix any $v>0$ to satisfy $\nu^{-3/2}<M_2$. Then by Corollary \ref{cor_abundant mass},  we get $$I_{1,\nu,1}=I_{1,\alp,1}\qd\mbox{ and }\qd S_{1,\nu,1}=S_{1,\alp,1}$$ for any $\alp>\nu$. So we obtain
    $$I_{1,\ift}\leq I_{1,\nu,1}<\ift.$$  Moreover, observe that $I_{1,\nu,1}\leq I_{1,\ift}$ by the relation $K_{1,\nu,1}\subset K_{1,\ift}.$ So we get 
    $$I_{1,\ift}\q=\q I_{1,\nu,1}\q=\q I_{1,\alp,1}\qd\mbox{ for each }\alp>\nu.$$ Using this, we observe that $S_{1,\nu,1}\subset S_{1,\ift}$, with relation $S_{1,\nu,1}=S_{1,\alp,1}$ for each $\alp>\nu$. In addition, we can show $S_{1,\nu,1}\supset S_{1,\ift}$ in the following way. We choose any  $\omg \in S_{1,\ift}$ and put $\alp':=\max \{\nu,\q \nrm{\omg}_1\}$ to get 
    $$\omg\in K_{1,\alp',1}\qd\mbox{ and }\qd E(\omg)=I_{1,\ift}=I_{1,\nu,1}=I_{1,\alp',1},$$  so we obtain $\omg \in S_{1,\alp',1}.$ By the relation $S_{1,\nu,1}=S_{1,\alp',1}$ due to $\alp'>\nu$,  we have $\omg \in S_{1,\nu,1}.$ Therefore we obtain 
    $$S_{1,\ift}\q=\q S_{1,\nu,1}=S_{1,\alp,1}\qd\mbox{ for each }\alp>\nu.$$ It completes the proof of the claim.\\

    Note that the above claim implies that $I_{1,\ift}<\ift$ and the set $S_{1,\ift}$ is nonempty. By Corollary \ref{cor_maximizer}, for each $\omg\in S_{1,\ift}$ there exists $W=W(\omg)>0$ and $\gmm=\gmm(\omg)\geq0$ satisfying that $$\omg=\bfone_{\set{\q\calG[\omg]-Wx_2-\gmm>0\q}},$$ where we have $\gmm=0$ due to the relations $\omg\in S_{1,\ift}=S_{1,\nu,1}=S_{1,2\nu,1}$ and $\nrm{\omg}_1\leq\nu<2\nu$. Moreover, Corollary \ref{cor_maximizer} guarantees that $\omg$ may be assumed to satisfy the Steiner symmetry condition up to translation. By Corollary \ref{cor_WM=energy}, we obtain the relation $$I_{1,\ift}=\f{3}{4}W,$$ which implies that the traveling speed is unique in the class $S_{1,\ift}$. Moreover, for the scaled maximizer $\hat{\omg}\in S_{\nu^{-3/2}}$, we can apply Lemma \ref{lem_ests for small impulse} and Lemma \ref{lem_size of w} to obtain that $$\f{1}{C_0}
    \Big(\nu^{-3/2}\Big)^{2/3}\q\leq\q\nrm{\hat{\omg}}_1\q\leq\q
    C_0\Big(\nu^{-3/2}\Big)^{2/3}\qd\mbox{ and }\sup_{\bfx\in\supp{({\hat{\omg}})}}{|\bfx|}\q\leq\q C_1\cdot \Big(\nu^{-3/2}\Big)^{1/3}$$ for the absolute constants $C_0>1,\q C_1>0$ in Lemma \ref{lem_ests for small impulse} and Lemma \ref{lem_size of w} each. Therefore we have 
$$ \f{1}{C_0}
\q\leq\q\nrm{\omg}_1\q\leq\q C_0 \qd\mbox{ and } \sup_{\bfx\in\supp{{(\omg)}}}{|\bfx|}\q\leq\q C_1.$$ Note that both inequalities hold for each $\omg\in S_{1,\ift}$. It completes the proof of Theorem \ref{thm_ no mass cond}. \end{proof}

\q \section*{Acknowledgments}

{We thank the anonymous referees for various comments which have improved the manuscript.}
KC has been supported by the National Research Foundation of Korea(NRF) grant funded by the Korea government(MSIT)(grant No. 2022R1A4A1032094, RS-2023-00274499). IJ has supported by the National Research Foundation of Korea(NRF) grant
funded by the Korea government(MSIT) (No. 2022R1C1C1011051, RS-2024-00406821).

When we were finalizing the manuscript, we noted the preprint by Huang--Tong {\cite{HuangTong}} which obtained the existence of a Sadovskii vortex patch using a fixed-point approach. It is not clear to us whether their vortex patch coincides with the one in Theorem \ref{Main thm 1}. They have a better description of the patch boundary while our variational construction of the Sadovskii vortex as the kinetic maximizer is natural in view of its observed dynamical stability.


\bibliographystyle{plain}

\end{document}